\documentclass{amsart}
\usepackage{amssymb}
\usepackage{amsfonts}

\newtheorem{theorem}{Theorem}[section]

\newtheorem{corollary}[theorem]{Corollary}
\newtheorem{definition}[theorem]{Definition}
\newtheorem{example}[theorem]{Example}
\newtheorem{lemma}[theorem]{Lemma}
\newtheorem{proposition}[theorem]{Proposition}
\theoremstyle{remark}
\newtheorem{remark}[theorem]{Remark}
\numberwithin{equation}{section}

\begin{document}
\title{NULLITY CONDITIONS IN PARACONTACT GEOMETRY}
\author{B. Cappelletti Montano}
\address{Dipartimento di Matematica, Universit\'{a} degli Studi di Cagliari,
Via Ospedale 72, 09124 Cagliari, ITALY}
\email{b.cappellettimontano@gmail.com}
\author{I. K\"{u}peli Erken}
\address{Art and Science Faculty, Department of Mathematics, Uludag
University, 16059 Bursa, TURKEY}
\email{iremkupeli@uludag.edu.tr}
\author{C. Murathan }
\address{Art and Science Faculty, Department of Mathematics, Uludag
University, 16059 Bursa, TURKEY} \email{cengiz@uludag.edu.tr}

\begin{abstract}
The paper is a complete study of paracontact metric manifolds for which the
Reeb vector field of the underlying contact structure satisfies a nullity
condition (the condition \eqref{paranullity} below, for some real numbers $%
\tilde\kappa$ and $\tilde\mu$). This class of pseudo-Riemannian manifolds,
which includes para-Sasakian manifolds, was recently defined in \cite{MOTE}.
In this paper we show in fact that there is a kind of duality between those
manifolds and contact metric $(\kappa,\mu)$-spaces. In particular, we prove
that, under some natural assumption, any such paracontact metric manifold
admits a compatible contact metric $(\kappa,\mu)$-structure (eventually
Sasakian). Moreover, we prove that the nullity condition is invariant under $%
\mathcal{D}$-homothetic deformations and determines the whole curvature
tensor field completely. Finally non-trivial examples in any dimension are
presented and the many differences with the contact metric case, due to the
non-positive definiteness of the metric, are discussed.
\end{abstract}

\maketitle

\section{\textbf{Introduction}}

A contact metric $(\kappa,\mu)$-space is a contact Riemannian manifold $%
(M,\varphi,\xi,\eta,g)$ such that the Reeb vector field $\xi$ belongs to the
so-called $\left(\kappa,\mu\right)$-nullity distribution, i.e. the curvature
tensor field satisfies the condition
\begin{equation}  \label{definizione}
R_{XY}\xi=\kappa\left(\eta\left(Y\right)X-\eta\left(X\right)Y\right)+\mu%
\left(\eta\left(Y\right)hX-\eta\left(X\right)hY\right),
\end{equation}
for some real numbers $\kappa$ and $\mu$, where $2h$ denotes the Lie
derivative of $\varphi$ in the direction of $\xi$. This new class of
Riemannian manifolds was introduced in \cite{BKP} as a natural
generalization both of the Sasakian condition $R_{XY}\xi=\eta\left(Y%
\right)X-\eta\left(X\right)Y$ and of those contact metric manifolds
satisfying $R_{XY}\xi=0$ which were studied by D. E. Blair in \cite{blair-1}%
. Nowadays contact $(\kappa,\mu)$-manifolds are considered a very important
topic in contact Riemannian geometry. In fact in despite of the technical
appearance of the definition, there are good reasons for studying $%
(\kappa,\mu)$-spaces. The first is that, in the non-Sasakian case (that is
for $\kappa\neq 1$), the condition \eqref{definizione} determines the
curvature tensor field completely; next, $(\kappa,\mu)$-spaces provide
non-trivial examples of some remarkable classes of contact Riemannian
manifolds, like CR-integrable contact metric manifolds (\cite{Tan2}), $H$%
-contact manifolds (\cite{perrone}), harmonic contact metric manifolds (\cite%
{vergara1}), or contact Riemannian manifolds with $\eta$-parallel torsion
tensor (\cite{gosh}); moreover, a local classification is known (\cite{BO}).

Recently, in \cite{MOTE}, an unexpected relationship between contact $%
(\kappa,\mu)$-spaces and paracontact geometry was found. It was proved (cf.
Theorem \ref{motivation} below) that any (non-Sasakian) contact $%
(\kappa,\mu) $-space carries a canonical paracontact metric structure $(%
\tilde{\varphi},\xi,\eta,\tilde{g})$ whose Levi-Civita connection satisfies
a condition formally similar to \eqref{definizione}
\begin{equation}  \label{paranullity}
\tilde{R}_{X
Y}\xi=\tilde\kappa\left(\eta\left(Y\right)X-\eta\left(X\right)Y\right)+%
\tilde\mu(\eta\left(Y\right)\tilde{h}X-\eta\left(X\right)\tilde{h}Y)
\end{equation}
where $2\tilde{h}:={\mathcal{L}}_{\xi}\tilde\varphi$ and, in this case, $%
\tilde\kappa=(1-\mu/2)^2+\kappa-2$, $\tilde\mu=2$.

We recall that paracontact manifolds are smooth manifolds of dimension $2n+1$
endowed with a $1$-form $\eta $, a vector field $\xi $ and a field of
endomorphisms of tangent spaces $\tilde{\varphi}$ such that $\eta (\xi )=1$,
$\tilde{\varphi}^{2}=I-\eta \otimes \xi $ and $\tilde{\varphi}$ induces an
almost paracomplex structure on the codimension $1$ distribution defined by
the kernel of $\eta $ (see $\S $ \ref{preliminaries} for more details). If,
in addition, the manifold is endowed with a pseudo-Riemannian metric $\tilde{%
g}$ of signature $(n,n+1)$ satisfying
\begin{equation*}
\tilde{g}(\tilde{\varphi}X,\tilde{\varphi}Y)=-\tilde{g}(X,Y)+\eta (X)\eta
(Y),\ \ \ d\eta (X,Y)=\tilde{g}(X,\tilde{\varphi}Y),
\end{equation*}%
$(M,\eta )$ becomes a contact manifold and $(\tilde{\varphi},\xi ,\eta ,%
\tilde{g})$ is said to be a paracontact metric structure on $M$. \ The study
of paracontact geometry was initiated by Kaneyuki and Williams in \cite%
{kaneyuki1} and then it was continued by many other authors. Very recently a
systematic study of paracontact metric manifolds, and some remarkable
subclasses like para-Sasakian manifolds, was carried out by Zamkovoy (\cite%
{Za}). The importance of paracontact geometry, and in particular of
para-Sasakian geometry, has been pointed out especially in the last years by
several papers highlighting the interplays with the theory of para-K\"{a}%
hler manifolds and its role in pseudo-Riemannian geometry and mathematical
physics (cf. e.g. \cite{alekseevski1}, \cite{alekseevski2}, \cite{erdem},
\cite{cortes1}, \cite{cortes2}).

These considerations motivate us to study the class of paracontact metric
manifolds satisfying the nullity condition \eqref{paranullity}, for some
constant real numbers $\tilde\kappa$ and $\tilde\mu$. We call these
pseudo-Riemannian manifolds \emph{paracontact $(\tilde\kappa,\tilde\mu)$%
-manifolds}. As we will see the class of paracontact $(\tilde\kappa,\tilde%
\mu)$-manifolds is very large. It contains para-Sasakian manifolds, as well
as those paracontact metric manifolds satisfying $\tilde{R}_{XY}\xi=0$ for
all $X,Y\in\Gamma(TM)$ (recently studied in \cite{ZATZA}). But, unlike in
the contact Riemannian case, a paracontact $(\tilde\kappa,\tilde\mu)$%
-manifold such that $\tilde\kappa=-1$ in general is not para-Sasakian. There
are in fact paracontact $(\tilde\kappa,\tilde\mu)$-manifolds such that $%
\tilde{h}^{2}=0$ (which is equivalent to take $\tilde\kappa=-1$) but with $%
\tilde{h}\neq 0$. Another important difference with the contact Riemannian
case, due to the non-positive definiteness of the metric, is that while for
contact metric $(\kappa,\mu)$-spaces the constant $\kappa$ can not be
greater than $1$, here we have no restrictions for the constants $%
\tilde\kappa$ and $\tilde\mu$.

In $\S $ \ref{third} we study the common properties for the cases $%
\tilde\kappa<-1$, $\tilde\kappa=-1$, $\tilde\kappa>-1$. We prove for
instance that while the values of $\tilde\kappa$ and $\tilde\mu$ change, the
form of \eqref{paranullity} remains unchanged under $\mathcal{D}$-homothetic
deformations. Moreover we prove that any paracontact $(\tilde\kappa,\tilde%
\mu)$-manifold is \emph{integrable} (in the sense of \cite{Za}), i.e. its
canonical paracontact connection preserves $\tilde\varphi$, and we find some
general properties of the curvature.

Since the geometric behavior of the manifold is very different according to
the circumstance that $\tilde\kappa<-1$ or $\tilde\kappa>-1$, we study
separately the two cases. In particular, in both cases we prove that the $%
(\tilde\kappa,\tilde\mu)$-nullity condition \eqref{paranullity} determines
the whole curvature tensor field completely. In fact we are able to find an explicit formula for the curvature, depending on the tensors $\tilde\varphi$, $\tilde h$, $\tilde\varphi h$. It is interesting that the same formula holds both for the case $\tilde\kappa<-1$ and $\tilde\kappa>-1$.  Then we find the values of $%
\tilde\kappa$ and $\tilde\mu$ for which the pseudo-Riemannian metric in
question is $\eta$-Einstein, i.e. $\text{Ric} = a I + b\eta\otimes\xi$, for
some $a,b\in\mathbb{R}$ and we prove that, unlike the contact
metric case, if $\tilde\kappa<-1$ there are Einstein paracontact $%
(\tilde\kappa,\tilde\mu)$-metrics in dimension greater than $3$.

In both cases $\tilde\kappa<-1$ and $\tilde\kappa>-1 $ the geometry of the
paracontact metric manifold can be related to the theory of Legendre
foliations. Namely, if $\tilde\kappa>-1$ then the operator $\tilde{h}$ is
diagonalizable and the eigendistributions corresponding to the constant
eigenvalues $\pm\tilde\lambda$, where $\tilde\lambda=\sqrt{1+\tilde\kappa}$,
define two mutually orthogonal and totally geodesic Legendre foliations.
Whereas, if $\tilde\kappa<-1$, the role before played by $\tilde{h}$ is now
played by the operator $\tilde\varphi\tilde{h}$. Such operator is
diagonalizable with the same eigenvalues as $\tilde{h}$. The main difference
with the previous case is that, while the eigendistributions corresponding
to $\pm\tilde{\lambda}$ (where now $\tilde{\lambda}=\sqrt{-1-\tilde\kappa}$)
still define two mutually orthogonal Legendre foliations, they are not
totally geodesic but they are totally umbilical. \ Then, by using the theory
of Legendre foliations, we prove that under some natural assumptions, a
paracontact $(\tilde\kappa,\tilde\mu)$-manifold carries a contact Riemannian
structure compatible with the contact form $\eta$, which in turn satisfies a
$(\kappa,\mu)$-nullity condition, for some constant real numbers $\kappa$
and $\mu$ depending on $\tilde\kappa$ and $\tilde\mu$. Therefore, in view of
such a result and \cite[Theorem 4.7]{MOTE}, it seems that there is a kind of
duality between contact and paracontact structures satisfying nullity
conditions.

Furthermore, we find non-trivial examples of paracontact $%
(\tilde\kappa,\tilde\mu)$-manifolds. We construct examples of left-invariant
paracontact $(\tilde\kappa,\tilde\mu)$-structures on Lie groups and,
moreover, we show that the tangent sphere bundle of a Riemannian manifold of
constant sectional curvature $c$ carries two canonical paracontact $%
(\tilde\kappa_{i},\tilde\mu_{i})$-structures $(\tilde\varphi_{i},\xi,\eta,%
\tilde{g}_{i})$, $i\in\left\{1,2\right\}$, (same $\eta$ and $\xi$, where $%
\xi $ is twice the geodesic flow), with
\begin{gather*}
\tilde\kappa_{1}=(1+c)^2-1, \ \ \ \tilde\mu_{1}=2(1-|c-1|), \\
\tilde\kappa_{2}=4c-1, \ \ \ \tilde\mu_{2}=2.
\end{gather*}
Hence, according to the value of $c$, we obtain examples of paracontact $%
(\tilde\kappa,\tilde\mu)$-manifolds such that $\tilde\kappa<-1$ and $%
\tilde\kappa>-1$. Also we prove that when the base manifold $M$ is flat than
the second structure provides an example of paracontact $(\tilde\kappa,%
\tilde\mu)$-manifold such that $\tilde\kappa=-1$ but which is not
para-Sasakian. To the knowledge of the authors these are the first
paracontact metric structures defined on the tangent sphere bundle.

Many questions about paracontact $(\tilde\kappa,\tilde\mu)$-manifolds remain
open. Apart of the problem of finding other non-trivial examples, the case
of strictly non-para-Sasakian paracontact $(\tilde\kappa,\tilde\mu)$%
-manifolds with $\tilde\kappa=-1$ is worthy to be studied. In particular it
should be important to find sufficient conditions for such manifolds in
order to be para-Sasakian. Other natural questions are to provide a
classification of such manifolds, at least in the $3$-dimensional case, and
to study further the unexpected interplays with contact Riemannian geometry
which we have found in this paper.

\section{Preliminaries}

\label{preliminaries}

A differentiable manifold $M$ of dimension $2n+1$ is said to be a \emph{%
contact manifold} if it carries a global 1-form $\eta $ such that $\eta
\wedge (d\eta )^{n}\neq 0$. It is well known that then there exists a unique
vector field $\xi $ (called the \emph{Reeb vector field}) such that $i_{\xi
}\eta =1$ and $i_{\xi }d\eta =0$. The $2n$-dimensional distribution
transversal to the Reeb vector field defined by ${\mathcal{D}}:=\ker (\eta )$
is called the \emph{contact distribution}. Any contact manifold $(M,\eta )$
admits a Riemannian metric $g$ and a $(1,1)$-tensor field $\varphi $ such
that
\begin{gather}
\varphi ^{2}=-I+\eta \otimes \xi ,\ \ \varphi \xi =0,\ \ \eta (X)=g(X,\xi )
\label{B1} \\
g(\varphi X,\varphi Y)=g(X,Y)-\eta (X)\eta (Y),\ \ g(X,\varphi Y)=d\eta
(X,Y),  \label{B0}
\end{gather}%
for any vector field $X$ and $Y$ on $M$. The contact manifold $(M,\eta )$
together with the geometric structure $(\varphi ,\xi ,\eta ,g)$ is then
called \emph{contact metric manifold} (or \emph{contact Riemannian manifold}%
). Let $h$ be the operator defined by $h=\frac{1}{2}{\mathcal{L}}_{\xi
}\varphi $, where $\mathcal{L}$ denotes Lie differentiation. The tensor
field $h$ vanishes identically if and only if the vector field $\xi $ is
Killing and in this case the contact metric manifold is said to be \textit{%
K-contact}. It is well known that $h$ and $\varphi h$ are symmetric
operators, and
\begin{equation*}
\varphi h+h\varphi =0,\ \ h\xi =0,\ \ \eta \circ h=0,\ \ \text{tr}h=\text{tr}%
\varphi h=0,
\end{equation*}%
where $\text{tr}h$ denotes the trace of $h$. Since $h$ anti-commutes with $%
\varphi $, if $X$ is an eigenvector of $h$ corresponding to the eigenvalue $%
\lambda $ then $\varphi X$ is also an eigenvector of $h$ corresponding to
the eigenvalue $-\lambda $. Moreover, for any contact metric manifold $M$,
the following relation holds
\begin{equation}
\nabla _{X}\xi =-\varphi X-\varphi hX  \label{B2}
\end{equation}%
where $\nabla $ is the Levi-Civita connection of $(M,g)$. If a contact
metric manifold $M$ is \emph{normal}, in the sense that the tensor field $%
N_{\varphi }:=[\varphi ,\varphi ]+2d\eta \otimes \xi $ vanishes identically,
then $M$ is called a \textit{Sasakian manifold}. Equivalently, a contact
metric manifold is Sasakian if and only if $R_{XY}\xi =\eta (Y)X-\eta (X)Y$.
Any Sasakian manifold is $K$-contact and in dimension $3$ the converse
also holds (cf. \cite{B1}).

As a natural generalization of the above Sasakian condition one can consider
contact metric manifolds satisfying
\begin{equation}
R_{XY}\xi =\kappa (\eta (Y)X-\eta (X)Y)+\mu (\eta (Y)hX-\eta (X)hY)
\label{RXYZETA}
\end{equation}
for some real numbers $\kappa$ and $\mu$. \eqref{RXYZETA} is called \emph{$%
(\kappa ,\mu)$-nullity condition}. This type of Riemannian manifolds was
introduced and deeply studied by Blair, Koufogiorgos and Papantoniou in \cite%
{BKP} and a few years earlier by Koufogiorgos for the case $\mu=0$ (\cite%
{koufogiorgos}). Among other things, they proved the following result.

\begin{theorem}[\protect\cite{BKP}]
Let $(M,\varphi,\xi ,\eta ,g)$ be a contact metric $(\kappa,\mu)$-manifold.
Then necessarily $\kappa\leq 1$ and $\kappa= 1$ if and only if $M$ is
Sasakian. Moreover, if $\kappa <1$, the contact metric manifold $M$ admits
three mutually orthogonal and integrable distributions ${\mathcal{D}}_{h}(0)$%
, ${\mathcal{D}}_{h}(\lambda)$ and ${\mathcal{D}}_{h}(-\lambda)$ defined by the
eigenspaces of $h$, where $\lambda=\sqrt{1-\kappa}$.
\end{theorem}


The standard contact metric structure on the tangent sphere bundle $T_{1}M$
satisfies the $(\kappa,\mu)$-nullity condition if and only if the base
manifold $M$ has constant curvature $c$. In this case $\kappa=c(2-c)$ and $%
\mu=-2c$ (\cite{BKP}). Other examples can be found in \cite{BO}.


\medskip

Now we recall the notion of almost paracontact manifold (cf. \cite{kaneyuki1}%
). An $(2n+1)$-dimensional smooth manifold $M$ is said to have an \emph{%
almost paracontact structure} if it admits a $(1,1)$-tensor field $\tilde{%
\varphi}$, a vector field $\xi $ and a $1$-form $\eta$ satisfying the
following conditions:

\begin{enumerate}
\item[(i)] $\eta(\xi )=1$, \ $\tilde{\varphi}^{2}=I-\eta \otimes \xi$,

\item[(ii)] the tensor field $\tilde{\varphi}$ induces an almost paracomplex
structure on each fibre of ${\mathcal{D}}=\ker(\eta)$, i.e. the $\pm 1$%
-eigendistributions, ${\mathcal{D}}^{\pm}:={\mathcal{D}}_{\tilde\varphi}(\pm
1)$ of $\tilde\varphi$ have equal dimension $n$.
\end{enumerate}

From the definition it follows that $\tilde{\varphi}\xi=0$, $\eta \circ
\tilde{\varphi}=0$ and the endomorphism $\tilde{\varphi}$ has rank $2n$.
When the tensor field $N_{\tilde{\varphi}}:=[\tilde\varphi,\tilde\varphi]-2d%
\eta \otimes \xi$ vanishes identically the almost paracontact manifold is
said to be \emph{normal}. If an almost paracontact manifold admits a
pseudo-Riemannian metric $\tilde{g}$ such that
\begin{equation}
\tilde{g}(\tilde{\varphi}X,\tilde{\varphi}Y)=-\tilde{g}(X,Y)+\eta (X)\eta
(Y),  \label{G METRIC}
\end{equation}%
for all $X,Y\in\Gamma(TM)$, then we say that $(M,\tilde{\varphi},\xi,\eta,%
\tilde{g})$ is an \textit{almost paracontact metric manifold}. Notice that
any such a pseudo-Riemannian metric is necessarily of signature $(n,n+1)$.
For an almost paracontact metric manifold, there always exists an orthogonal
basis $\{X_{1},\ldots,X_{n},Y_{1},\ldots,Y_{n},\xi\}$ such that $\tilde{g}%
(X_{i},X_{j})=\delta_{ij}$, $\tilde{g}(Y_{i},Y_{j})=-\delta_{ij}$ and $%
Y_{i}=\tilde\varphi X_{i}$, for any $i,j\in\left\{1,\ldots,n\right\}$. Such
basis is called a $\tilde\varphi$-basis.

If in addition $d\eta (X,Y)=\tilde{g}(X,\tilde{\varphi}Y)$ for all
vector fields $X,Y$ on $M,$ $(M,\tilde{\varphi},\xi,\eta,\tilde{g})$ is said to be
a \emph{paracontact metric manifold}. \ In a paracontact metric manifold one
defines a symmetric, trace-free operator $\tilde{h}:=\frac{1}{2}{\mathcal{L}}%
_{\xi }\tilde{\varphi}$. It is known (\cite{Za}) that $\tilde{h}$
anti-commutes with $\tilde{\varphi}$ and satisfies $\tilde{h}\xi =0$ and
\begin{equation}  \label{nablaxi}
\tilde{\nabla}\xi =-\tilde{\varphi}+\tilde{\varphi}\tilde{h},
\end{equation}
where $\tilde{\nabla}$ is the Levi-Civita connection of the
pseudo-Riemannian manifold $(M,\tilde{g})$. Moreover $\tilde{h}\equiv 0$ if
and only if $\xi $ is a Killing vector field and in this case $(M,\tilde{%
\varphi},\xi ,\eta ,\tilde{g})$ is said to be a \emph{K-paracontact manifold}%
. A normal paracontact metric manifold is called a \textit{para-Sasakian
manifold}. Also in this context the para-Sasakian condition implies the $K$%
-paracontact condition and the converse holds only in dimension $3$.
Moreover, in any para-Sasakian manifold
\begin{equation}
\tilde{R}_{XY}\xi =-(\eta (Y)X-\eta (X)Y),  \label{Pasa}
\end{equation}
holds, but unlike contact metric geometry the condition (\ref{Pasa}) not
necessarily implies that the manifold is para-Sasakian.

In \cite{Za}, the author proved the following results, which will be used in
next sections:

\begin{theorem}[\protect\cite{Za}]
On a paracontact metric manifold, the following identities hold:
\begin{gather}
(\tilde{\nabla}_{\tilde{\varphi}X}\tilde{\varphi})\tilde{\varphi}Y-(\tilde{%
\nabla}_{X}\tilde{\varphi})Y=2\tilde{g}(X,Y)\xi -\eta (Y)(X-\tilde{h}X+\eta
(X)\xi ),  \label{namlafibar} \\
\tilde{R}_{\xi X}\xi +\tilde{\varphi}\tilde{R}_{\xi \tilde{\varphi}X}\xi =2(%
\tilde{\varphi}^{2}X-\tilde{h}^{2}X),  \label{FiL} \\
\tilde{R}(\xi ,X,Y,Z)=-\tilde{g}(Y,(\tilde{\nabla}_{X}\tilde{\varphi})Z)+%
\tilde{g}(X,(\tilde{\nabla}_{Y}\tilde{\varphi}\tilde{h})Z)-\tilde{g}(X,(%
\tilde{\nabla}_{Z}\tilde{\varphi}\tilde{h})Y),  \label{Curvature2} \\
\tilde{R}(\xi ,X,Y,Z)+\tilde{R}(\xi ,X,\tilde{\varphi}Y,\tilde{\varphi}Z)-%
\tilde{R}(\xi ,\tilde{\varphi}X,\tilde{\varphi}Y,Z)-\tilde{R}(\xi ,\tilde{%
\varphi}X,Y,\tilde{\varphi}Z)  \label{Curvature3} \\
=-2(\tilde{\nabla}_{\tilde{h}X}\tilde{\Phi})(Y,Z)+2\eta(Y)\tilde{g}(X-\tilde{%
h}X,Z)-2\eta(Z)\tilde{g}(X-\tilde{h}X,Y),  \notag
\end{gather}%
where $\tilde{\Phi}:=\tilde{g}(\cdot ,\tilde{\varphi}\cdot )$ is the \emph{%
fundamental $2$-form} of the paracontact metric structure.
\end{theorem}

Moreover, in any paracontact metric manifold Zamkovoy introduced a canonical
connection which plays the same role in paracontact geometry of the
generalized Tanaka-Webster connection (\cite{Tan2}) in a contact metric
manifold.

\begin{theorem}[\protect\cite{Za}]
\label{zamconn} On a paracontact metric manifold there exists a unique
connection $\tilde{\nabla}^{pc}$, called the \emph{canonical paracontact
connection}, satisfying the following properties:

\begin{enumerate}
\item[(i)] $\tilde{\nabla}^{pc}\eta=0$, \ $\tilde{\nabla}^{pc}\xi=0$, \ $%
\tilde{\nabla}^{pc}\tilde g=0$,

\item[(ii)] $(\tilde\nabla^{pc}_{X}\tilde\varphi)Y=(\tilde\nabla_{X}\tilde%
\varphi)Y+\tilde g(X-\tilde h X,Y)\xi-\eta(Y)(X-\tilde h X)$,

\item[(iii)] $\tilde T^{pc}(\xi,\tilde\varphi Y)=-\tilde\varphi \tilde
T^{pc}(\xi,Y)$,

\item[(iv)] $\tilde T^{pc}(X,Y)=2d\eta(X,Y)\xi$ on ${\mathcal{D}}=\ker(\eta)$%
.
\end{enumerate}

Moreover $\tilde\nabla^{pc}$ is given by
\begin{equation}  \label{PcTanaka}
\tilde\nabla^{pc}_{X}Y=\tilde\nabla_{X}Y+\eta(X)\tilde\varphi
Y+\eta(Y)(\tilde\varphi X-\tilde\varphi \tilde h X)+\tilde g(X-\tilde
hX,\tilde\varphi Y)\xi.
\end{equation}
\end{theorem}

An almost paracontact structure ($\tilde{\varphi},\xi ,\eta )$ is said to be
\emph{integrable} if $N_{\tilde{\varphi}}(X,Y)\in\Gamma(\mathbb{R}\xi)$
whenever $X,Y\in\Gamma({\mathcal{D}})$. For paracontact metric structures,
the integrability condition is equivalent to $\tilde{\nabla}^{pc}\tilde{%
\varphi}=0$ (\cite{Za}).

\medskip

As pointed out in \cite{MON}, paracontact geometry is strictly related to
the theory of Legendre foliations. Recall that a \textit{Legendrian
distribution} on contact manifold $(M,\eta )$ is an $n$-dimensional
subbundle $L$ of the contact distribution such that $d\eta (X,X^{\prime })=0$
for all $X,Y\in \Gamma (L)$. When $L$ is integrable, we speak of \emph{%
Legendrian foliation}. Legendre foliations have been extensively
investigated in recent years from various points of views. In particular
Pang (\cite{PAN}) provided a classification of Legendrian foliations using a
bilinear symmetric form $\Pi _{\mathcal{F}}$ on tangent bundle of the
foliation ${\mathcal{F}}$, defined by
\begin{equation}
\Pi _{\mathcal{F}}(X,X^{\prime })=-({\mathcal{L}}_{X}{\mathcal{L}}%
_{X^{\prime }}\eta )(\xi )=2d\eta ([\xi ,X],X^{\prime }).
\label{panginvariant}
\end{equation}%
Then he called $\mathcal{F}$ \emph{flat}, \emph{degenerate}, \emph{%
non-degenerate}, \emph{positive (negative) definite} according to the
circumstance that $\Pi _{\mathcal{F}}$ vanishes identically, is degenerate,
non-degenerate, positive (negative) definite, respectively. For a
non-degenerate Legendre foliation $\mathcal{F}$, Libermann (\cite{LIB})
defined a linear map $\Lambda _{\mathcal{F}}:TM\longrightarrow T{\mathcal{F}}
$, whose kernel is ${T\mathcal{F}}\oplus \mathbb{R}\xi $, such that
\begin{equation}
\Pi _{\mathcal{F}}(\Lambda _{\mathcal{F}}Z,X)=d\eta (Z,X)  \label{lambda}
\end{equation}%
for any $Z\in \Gamma (TM)$, $X\in \Gamma (T{\mathcal{F}})$. The operator $%
\Lambda _{\mathcal{F}}$ is surjective, satisfies $(\Lambda _{\mathcal{F}%
})^{2}=0$ and $\Lambda _{\mathcal{F}}[\xi ,X]=\frac{1}{2}X$ for all $X\in
\Gamma (T{\mathcal{F}})$. Then we can extend $\Pi _{\mathcal{F}}$ to a
symmetric bilinear form on all $TM$ by putting
\begin{equation*}
\overline{\Pi }_{\mathcal{F}}(Z,Z^{\prime }):=\left\{
\begin{array}{ll}
\Pi _{\mathcal{F}}(Z,Z^{\prime }) & \hbox{if $Z,Z'\in\Gamma(T{\mathcal F})$}
\\
\Pi _{\mathcal{F}}(\Lambda _{\mathcal{F}}Z,\Lambda _{\mathcal{F}}Z^{\prime
}), & \hbox{otherwise.}%
\end{array}%
\right.
\end{equation*}%
An \textit{(almost) bi-Legendrian manifold} (cf. \cite{MON}) is by
definition a contact manifold $(M,\eta )$ endowed with two transversal
Legendrian distributions (foliations) $L_{1}$ and $L_{2}$, so that $TM=$ $%
L_{1}$ $\oplus $ $L_{2}\oplus \mathbb{R}\xi $. $(L_{1}$, $L_{2})$ is then
called an \textit{(almost) bi-Legendrian structure} on the contact manifold $%
(M,\eta )$. Any paracontact metric manifold $(M,\tilde{\varphi},\xi ,\eta ,%
\tilde{g})$ carries a canonical almost bi-Legendrian structure given the
eigendistributions ${\mathcal{D}}^{+}$ and ${\mathcal{D}}^{-}$ of $\tilde{%
\varphi}$ corresponding to the eingenvalues $\pm 1$. Conversely, every
almost bi-Legendrian manifold admits a compatible paracontact metric
structure (\cite{MON}). We notice also that the integrability in the sense
of paracontact geometry, i.e. $\nabla ^{pc}\tilde{\varphi}=0$, is equivalent
to the involutiveness of the Legendre distributions ${\mathcal{D}}^{\pm }$
(cf. \cite[Corollary 3.3]{MON}).

Any almost bi-Legendrian manifold admits a canonical connection, called
\emph{bi-Legendrian connection}, which plays an important role in the study
of almost bi-Legendrian manifolds:

\begin{theorem}[\protect\cite{CAP5}]
\label{biconnection} Let $(M,\eta,L_1,L_2)$ be an almost bi-Legendrian
manifold. There exists a unique connection ${\nabla}^{bl}$ such that
\begin{enumerate}
\item[(i)] ${\nabla }^{bl}L_{1}\subset L_{1}$, \ ${\nabla }^{bl}L_{2}\subset
L_{2}$, \ ${\nabla }^{bl}\left( \mathbb{R}\xi \right) \subset \mathbb{R}\xi $,
\item[(ii)] ${\nabla}^{bl} d\eta=0$,
\item[(iii)] ${T}^{bl}\left(X,Y\right)=2d\eta\left(X,Y\right){\xi}$ \ for
all $X\in\Gamma(L_1)$, $Y\in\Gamma(L_2)$,\newline
${T}^{bl}\left(X,\xi\right)=[\xi,X_{L_1}]_{L_2}+[\xi,X_{L_2}]_{L_1}$ \ for
all $X\in\Gamma\left(TM\right)$,
\end{enumerate}
where $X_{L_1}$ and $X_{L_2}$ denote, respectively, the projections of $X$
onto the subbundles $L_1$ and $L_2$ of $TM$, according to the decomposition $%
TM=L_1\oplus L_2\oplus\mathbb{R}\xi$.
\end{theorem}

By using the properties of the bi-Legendrian connection one can point out
the relationship between contact metric $(\kappa,\mu)$-spaces and the theory
of Legendre foliations. Namely, we have the following characterization.

\begin{theorem}[\protect\cite{CAP4}]
\label{characterization} Let $(M,\varphi,\xi,\eta,g)$ be a contact metric
manifold, which is not $K$-contact. Then $(M,\varphi,\xi,\eta,g)$ is a
contact metric $\left(\kappa,\mu\right)$-manifold if and only if it admits
two mutually orthogonal Legendre distributions $L$ and $Q$ and a unique
linear connection $\bar{\nabla}$ satisfying the following properties:
\begin{enumerate}
\item[\textrm{(i)}] $\bar{\nabla}L\subset L$, \ $\bar{\nabla} Q\subset Q$,
\item[\textrm{(ii)}] $\bar{\nabla}\eta=0$, \ $\bar{\nabla}d\eta=0$, \ $\bar{%
\nabla}g=0$, \ $\bar{\nabla}\varphi=0$, \ $\bar{\nabla}h=0$,
\item[\textrm{(iii)}] $\bar{T}\left(X,Y\right)=2d\eta\left(X,Y\right){\xi}$
\ for all $X,Y\in\Gamma({\mathcal{D}})$,\newline
$\bar{T}(X,\xi)=[\xi,X_{L}]_{Q}+[\xi,X_{Q}]_{L}$ \ for all $X\in\Gamma(TM)$,
\end{enumerate}
Furthermore $\bar{\nabla}$ is uniquely determined and coincide with the
bi-Legendrian connection of $(L,Q)$, $L$ and $Q $ are integrable and
coincides with the eigenspaces ${\mathcal{D}}_{h}(\lambda)$ and ${\mathcal{D}%
}_{h}(-\lambda)$ of the operator $h$.
\end{theorem}



On the other hand contact $(\kappa,\mu)$-manifolds are also related to
paracontact geometry, as shown by the following result which is one of the
motivations for the present paper.

\begin{theorem}[\protect\cite{MOTE}]
\label{motivation} Let $(M,\varphi ,\xi ,\eta ,g)~$\ be a non-Sasakian
contact metric $(\kappa ,\mu )$-space. Then $M$ admits a canonical
paracontact metric structure $(\tilde{\varphi},\xi ,\eta ,\tilde{g})$ given
by
\begin{equation}
\tilde{\varphi}:=\frac{1}{\sqrt{1-\kappa }}h, \ \ \tilde{g}:=\frac{1}{\sqrt{%
1-\kappa }}d\eta (\cdot ,h\cdot )+\eta \otimes \eta .  \label{CAPAR1}
\end{equation}
Furthermore the curvature tensor field of the Levi-Civita connection of $(M,%
\tilde{g})$ satisfies a $(\tilde\kappa,\tilde\mu)$-nullity condition
\begin{equation}
\tilde{R}_{XY}\xi =\tilde{\kappa}(\eta (Y)X-\eta (X)Y)+\tilde{\mu}(\eta (Y)%
\tilde{h}X-\eta (X)\tilde{h}Y),  \label{CAPAR3}
\end{equation}
where
\begin{equation*}
\tilde{\kappa}=\kappa -2+\left( 1-\frac{\mu }{2}\right) ^{2}, \ \ \tilde{\mu}%
=2.
\end{equation*}
\end{theorem}

\section{Preliminary results on paracontact $(\tilde{\protect\kappa},\tilde{%
\protect\mu})$-manifolds}

\label{third}

Theorem \ref{motivation} motivates the following definition.

\begin{definition}[\protect\cite{MOTE}]
A paracontact metric $(\tilde{\kappa},\tilde{\mu})$-manifold is a
paracontact metric manifold for which the curvature tensor field satisfies
\begin{equation}  \label{PARAKMU}
\tilde{R}_{XY}\xi =\tilde{\kappa}(\eta (Y)X-\eta (X)Y)+\tilde{\mu}(\eta (Y)%
\tilde{h}X-\eta (X)\tilde{h}Y)
\end{equation}%
for all vector fields $X$, $Y$ on $M$ and for some real constants $%
\tilde\kappa $ and $\tilde\mu$.
\end{definition}

In this section, we discuss some properties of paracontact metric manifolds
satisfying the condition \eqref{PARAKMU}. We start with some preliminary
properties.

\begin{lemma}
Let $(M,\tilde{\varphi},\xi ,\eta ,\tilde{g})$ be a paracontact metric $(%
\tilde{\kappa},\tilde{\mu})$-manifold of dimension $2n+1$. Then the
following identities hold:
\begin{gather}
\tilde{h}^{2}=(1+\tilde{\kappa})\tilde{\varphi}^{2},  \label{H2} \\
\tilde{Q}\xi =2n\tilde{\kappa}\xi,  \label{Riczeta} \\
(\tilde{\nabla}_{X}\tilde{\varphi})Y=-\tilde{g}(X-\tilde{h}X,Y)\xi
+\eta(Y)(X-\tilde{h}X), \ \text{ for }\tilde{\kappa}\neq -1,  \label{NAMLAFI}
\\
(\tilde{\nabla}_{X}\tilde{h})Y-(\tilde{\nabla}_{Y}\tilde{h})X=-(1+\tilde{%
\kappa})(2\tilde{g}(X,\tilde{\varphi}Y)\xi +\eta (X)\tilde{\varphi}Y-\eta (Y)%
\tilde{\varphi}X)  \label{NAMLA X H} \\
\quad\quad\quad\quad\quad+(1-\tilde{\mu})(\eta (X)\tilde{\varphi}\tilde{h}%
Y-\eta (Y)\tilde{\varphi}\tilde{h}X)  \notag \\
\tilde{\nabla}_{\xi }\tilde{h}=\tilde{\mu}\tilde{h}\circ \tilde{\varphi}, \
\ \ \tilde{\nabla}_{\xi }\tilde{\varphi}\tilde{h}=-\tilde{\mu}\tilde{h}.
\label{NMBLA ZETAH}
\end{gather}
for any vector fields $X$, $Y$ on $M$, where $\tilde{Q}$ denotes the Ricci
operator of $(M,\tilde{g})$.
\end{lemma}

\begin{proof}
\eqref{H2} was proved in \cite{MOTE}. Next, let $\{e_{i},\tilde{\varphi}%
e_{i},\xi \}$, $i\in\left\{1,\ldots,n\right\}$, be a $\tilde{\varphi}$-basis
of $M$. Then the definition of the Ricci operator directly gives %
\eqref{Riczeta}. For \eqref{NAMLAFI}, notice that using \eqref{PARAKMU} one
can easily show that
\begin{equation}
\tilde{R}_{\xi X}Y=\tilde{\kappa}(\tilde{g}(X,Y)\xi -\eta (Y)X)+\tilde{\mu}(%
\tilde{g}(\tilde{h}X,Y)\xi -\eta (Y)\tilde{h}X)  \label{R(X,zeta)Y}
\end{equation}
By virtue of \eqref{R(X,zeta)Y}, the equation \eqref{Curvature3} reduces to
\begin{equation*}
(\tilde{\nabla}_{\tilde{h}X}\tilde{\varphi})Y=\tilde{\kappa}(\tilde{g}%
(X,Y)\xi -\eta (Y)X)-\eta (Y)(X-\tilde{h}X)+\tilde{g}(X-\tilde{h}X,Y)\xi .
\end{equation*}
By replacing $X$ by $\tilde{h}X$ in that equation and using \eqref{H2}, we
get
\begin{equation*}
(1+\tilde{\kappa})((\tilde{\nabla}_{X}\tilde{\varphi})Y+\tilde{g}(X-\tilde{h}%
X,Y)\xi -\eta (Y)(X-\tilde{h}X))=0.
\end{equation*}
Hence \eqref{NAMLAFI} holds. Next, using \eqref{NAMLAFI} and the symmetry of
$\tilde{h}$, we obtain
\begin{equation}  \label{namlaFIH}
(\tilde{\nabla}_{Z}\tilde{\varphi}\tilde{h})Y-(\tilde{\nabla}_{Y}\tilde{%
\varphi}\tilde{h})Z=\tilde{\varphi}((\tilde{\nabla}_{Z}\tilde{h})Y-(\tilde{%
\nabla}_{Y}\tilde{h})Z)
\end{equation}
for all $Y,Z\in\Gamma(TM)$. Substituting \eqref{namlaFIH} in %
\eqref{Curvature2}, we get
\begin{equation*}
\tilde{R}_{YZ}\xi =-\eta (Z)(Y-\tilde{h}Y)+\eta (Y)(Z-\tilde{h}Z)+\tilde{%
\varphi}((\tilde{\nabla}_{Y}\tilde{h})Z-(\tilde{\nabla}_{Z}\tilde{h})Y).
\end{equation*}%
Comparing this equation with \eqref{PARAKMU}, we obtain
\begin{gather}
\tilde{\varphi}((\tilde{\nabla}_{Y}\tilde{h})Z-(\tilde{\nabla}_{Z}\tilde{h}%
)Y)=(\tilde{\kappa}+1)(\eta (Z)Y-\eta (Y)Z)+(\tilde{\mu}-1)(\eta (Z)\tilde{h}%
Y-\eta (Y)\tilde{h}Z).  \label{FIBAR}
\end{gather}
Using \eqref{nablaxi} and the symmetry of $\tilde{h}$ and $\tilde{\nabla}_{Z}%
\tilde{h}$, by a direct computation we have
\begin{equation}
\tilde{g}((\tilde{\nabla}_{Z}\tilde{h})Y-(\tilde{\nabla}_{Y}\tilde{h})Z,\xi
)=2(1+\tilde{\kappa})\tilde{g}(\tilde{\varphi}Z,Y).  \label{HBAR2}
\end{equation}
By applying now $\tilde{\varphi}$ to \eqref{FIBAR} and using \eqref{HBAR2},
we obtain \eqref{NAMLA X H}. Finally, \eqref{NMBLA
ZETAH} follows from \eqref{NAMLA X H} by using the properties of $\tilde{h}$.
\end{proof}

By \eqref{NAMLAFI} we get the following corollary

\begin{corollary}
\label{integrability} Any paracontact $(\tilde\kappa,\tilde\mu)$-manifolds
such that $\tilde\kappa\neq - 1$ is integrable.
\end{corollary}

In particular from Corollary \ref{integrability} it follows that in any
paracontact $(\tilde\kappa,\tilde\mu)$-manifold such that $\tilde\kappa\neq
-1$ the canonical Legendre distributions ${\mathcal{D}}^{+}$ and ${\mathcal{D%
}}^{-}$ are integrable and so define two Legendre foliations on $M$.

Remarkable subclasses of paracontact $(\tilde\kappa,\tilde\mu)$-manifolds
are given, in view of \eqref{Pasa}, by para-Sasakian manifolds, and by those
paracontact metric manifolds such that $\tilde{R}_{XY}\xi =0$ for all vector
fields $X,Y$ on $M$. In this last case it was proved (\cite{ZATZA}) that in
dimension greater than $3$ the paracontact metric manifold $(M^{2n+1},%
\tilde{\varphi},\xi,\eta,\tilde{g})$ is locally isometric to a product of a
flat $(n+1)$-dimensional manifold and an $n$-dimensional manifold of
negative constant curvature $-4$. \ Notice that, because of \eqref{H2}, a
paracontact $(\tilde\kappa,\tilde\mu)$-manifold such that $\tilde{\kappa}=-1$
satisfies $\tilde{h}^{2}=0$. Unlike the contact metric case, since the
metric $\tilde{g} $ is pseudo-Riemannian we can not conclude that $\tilde{h}$
vanishes and so the manifold is para-Sasakian. Let us see an explicit
counterexample.

The canonical example of paracontact $(\tilde{\kappa},\tilde{\mu})$-manifold
is given by the tangent sphere bundle $T_{1}M$ of a Riemannian manifold $%
(M,g)$ of constant sectional curvature $c$. The paracontact metric structure
is defined in the following way. Let us consider the standard contact metric
structure $(\varphi ,\xi ,\eta ,g)$ of $T_{1}M$, which is in fact a $%
(c(2-c),-2c)$-structure (cf. \cite{B1}). Let us define
\begin{gather}
\tilde{\varphi}_{1}:=\frac{1}{|1-c|}\varphi h,\ \ \ \tilde{g}_{1}:=\frac{1}{%
|1-c|}d\eta (\cdot ,\varphi h\cdot )+\eta \otimes \eta  \label{sphere1} \\
\tilde{\varphi}_{2}:=\frac{1}{|1-c|}h,\ \ \ \ \ \tilde{g}_{2}:=\frac{1}{|1-c|%
}d\eta (\cdot ,h\cdot )+\eta \otimes \eta .  \label{sphere2}
\end{gather}%
Then one can easily check that $(\tilde{\varphi}_{1},\eta ,\xi ,\tilde{g}%
_{1})$ and $(\tilde{\varphi}_{2},\eta ,\xi ,\tilde{g}_{2})$ define two
paracontact metric structures on $T_{1}M$. Thus by Theorem 5.9 of \cite{CAP2}
we have that $(\tilde{\varphi}_{1},\eta ,\xi ,\tilde{g}_{1})$ is a
paracontact $(\tilde{\kappa}_{1},\tilde{\mu}_{1})$-structure and $(\tilde{%
\varphi}_{2},\eta ,\xi ,\tilde{g}_{2})$ a  paracontact $(\tilde{\kappa}_{2},%
\tilde{\mu}_{2})$-structure, where
\begin{gather*}
\tilde{\kappa}_{1}=(1+c)^{2}-1,\ \ \ \tilde{\mu}_{1}=2(1-|c-1|), \\
\tilde{\kappa}_{2}=4c-1,\ \ \ \tilde{\mu}_{2}=2.
\end{gather*}

Hence we can state the following theorem.

\begin{theorem}
\label{sphere3} The tangent sphere bundle of a Riemannian manifold of
constant curvature $c\neq 1$ is canonically endowed, via \eqref{sphere1}--%
\eqref{sphere2}, with a paracontact $((1+c)^2-1,2(1-|c-1|))$-structure and
with a paracontact $(4c-1,2)$-structure.
\end{theorem}

Consequently, if the base manifold is flat, $(\tilde\varphi_{2},\xi,\eta,%
\tilde{g}_2)$ is a paracontact $(-1,2)$-structure on $T_{1}M$ such that $%
\tilde{h}_{2}^{2}=0$, but which is not para-Sasakian because $\tilde{h}_{2}$
does not vanish. Indeed, according to \eqref{sphere2} and \cite[Lemma 4.5]%
{MOTE}, one has that $h_{2}=\frac{1}{2}{\mathcal{L}}_{\xi}h=\varphi h
+\varphi$.

\medskip

Given a paracontact metric structure $(\tilde{\varphi},\xi ,\eta ,\tilde{g})$
and $\alpha >0$, the change of structure tensors
\begin{equation}
\bar{\eta}=\alpha \eta ,\text{ \ \ }\bar{\xi}=\frac{1}{\alpha }\xi ,\text{ \
\ }\bar{\varphi}=\tilde{\varphi},\text{ \ \ }\bar{g}=\alpha \tilde{g}+\alpha
(\alpha -1)\eta \otimes \eta  \label{DHOMOTHETIC}
\end{equation}%
is called a \emph{${\mathcal{D}}_{\alpha }$-homothetic deformation}. One can
easily check that the new structure $(\bar{\varphi},\bar{\xi},\bar{\eta},%
\bar{g})$ is still a paracontact metric structure (\cite{Za}). We now show
that while ${\mathcal{D}}_{\alpha }$-homothetic deformations destroy
conditions like $\tilde{R}_{XY}\xi =0$, they preserve the class of
paracontact $(\tilde{\kappa},\tilde{\mu})$-spaces.

\begin{proposition}
\label{levicivita} Let $(\bar{\varphi},\bar{\xi},\bar{\eta},\bar{g})$ be a
paracontact metric structure obtained from $(\tilde{\varphi},\tilde{\xi},%
\tilde{\eta},\tilde{g})$ by a ${\mathcal{D}}_{\alpha }$-homothetic
deformation. Then we have the following relationship between the Levi-Civita
connections $\bar{\nabla}$ and $\tilde{\nabla}$ of $\bar{g}$ and $\tilde{g}$%
, respectively,
\begin{equation}  \label{CONNECTION}
\bar{\nabla}_{X}Y=\tilde{\nabla}_{X}Y+\frac{\alpha -1}{\alpha }\tilde{g}(%
\tilde{\varphi}\tilde{h}X,Y)\xi -(\alpha -1)\left( \eta (Y)\tilde{\varphi}%
X+\eta (X)\tilde{\varphi}Y\right).
\end{equation}
Furthermore,
\begin{equation}
\bar{h}=\frac{1}{\alpha }\tilde{h}.  \label{H BAR}
\end{equation}
\end{proposition}

\begin{proof}
Using (\ref{DHOMOTHETIC}) and the Koszul formula we obtain, for any $%
X,Y,Z\in\Gamma(TM)$,
\begin{align}
\bar{g}(\bar{\nabla}_{X}Y,Z)&=\alpha \tilde{g}(\tilde{\nabla}_{X}Y,Z)+\alpha
(\alpha -1)\eta (\tilde{\nabla}_{X}Y)\eta (Z)  \label{KOSZUL} \\
&\quad+\eta (Z)\tilde{g}(\tilde{\varphi}\tilde{h}X,Y)-\eta (Y)\tilde{g}(%
\tilde{\varphi}X,Z)-\eta (X)\tilde{g}(\tilde{\varphi}Y,Z).  \notag
\end{align}
Moreover we have
\begin{equation}  \label{CON1}
\bar{g}(\bar{\nabla}_{X}Y,Z)=\alpha \tilde{g}(\bar{\nabla}_{X}Y,Z)+\alpha
(\alpha -1)\eta (\bar{\nabla}_{X}Y)\eta (Z)
\end{equation}%
and
\begin{equation}  \label{CON2}
\eta (\bar{\nabla}_{X}Y)=\frac{1}{\alpha ^{2}}\bar{g}(\bar{\nabla}_{X}Y,\xi
).
\end{equation}%
Setting $Z=\xi $ in (\ref{KOSZUL}) we get
\begin{equation}
\bar{g}(\bar{\nabla}_{X}Y,\xi )=\alpha ^{2}\eta (\tilde{\nabla}_{X}Y)+\alpha
(\alpha -1)\tilde{g}(\tilde{\varphi}\tilde{h}X,Y)  \label{CON3}
\end{equation}
Then \eqref{CONNECTION} easily follows from \eqref{KOSZUL}, \eqref{CON1}, %
\eqref{CON2} and \eqref{CON3}. Finally, by using \eqref{CONNECTION} and the
definition of $\tilde{h}$ we get (\ref{H BAR}).
\end{proof}

After a long but straightforward calculation one can prove the following
proposition.

\begin{proposition}
Under the same assumptions of Proposition \ref{levicivita}, the curvature
tensor fields $\bar{R}$ and $\tilde{R}$ are related by
\begin{align}
\alpha \bar{R}_{XY}\bar{\xi}&=\tilde{R}_{XY}\xi -(\alpha -1)((\tilde{\nabla}%
_{X}\tilde{\varphi})Y-(\tilde{\nabla}_{Y}\tilde{\varphi})X+\eta (Y)(X-\tilde{%
h}X) -\eta (X)(Y-\tilde{h}Y))  \notag \\
&\quad-(\alpha -1)^{2}(\eta (Y)X-\eta (X)Y)  \label{CURVATURE}
\end{align}
In particular, if $(M,\tilde{\varphi},\xi ,\eta ,\tilde{g})$ is a
paracontact $(\tilde\kappa,\tilde\mu)$-manifold, then $(\bar{\varphi},\bar{%
\xi},\bar{\eta},\bar{g})$ is a paracontact $(\bar{\kappa},\bar{\mu})$%
-structure, where
\begin{equation}  \label{dhomothetic1}
\bar{\kappa}=\frac{\tilde{\kappa}+1-\alpha ^{2}}{\alpha ^{2}}, \ \ \ \bar{\mu%
}=\frac{\tilde{\mu}+2\alpha -2}{\alpha }.
\end{equation}
\end{proposition}

We pass to discuss some general curvature properties of paracontact $%
(\tilde\kappa,\tilde\mu)$-manifolds. We start with the following preliminary
result.

\begin{theorem}
Let $(M^{2n+1},\tilde{\varphi},\xi ,\eta ,\tilde{g})$ be an integrable
paracontact metric manifold. Then the following identity holds
\begin{equation}
\tilde{Q}\tilde{\varphi}-\tilde{\varphi}\tilde{Q}=\tilde{l}\tilde{\varphi}-%
\tilde{\varphi}\tilde{l}-4(n-1)\tilde{\varphi}\tilde{h}-\eta \otimes \tilde{%
\varphi}\tilde{Q}+(\eta \circ \tilde{Q}\tilde{\varphi})\otimes \xi ,
\label{QFI-FIQ}
\end{equation}
where $\tilde{l}$ denotes the Jacobi operator, defined by $\tilde{l}X=%
\tilde{R}_{X\xi}\xi$.
\end{theorem}

\begin{proof}
Differentiating $\tilde{\nabla}_{Y}\xi =-\tilde{\varphi}Y+\tilde{\varphi}%
\tilde{h}Y$, we get
\begin{equation}
\tilde{R}_{XY}\xi =-(\tilde{\nabla}_{X}\tilde{\varphi})Y+(\tilde{\nabla}_{Y}%
\tilde{\varphi})X+(\tilde{\nabla}_{X}\tilde{\varphi}\tilde{h})Y-(\tilde{%
\nabla}_{Y}\tilde{\varphi}\tilde{h})X.  \label{CURVATURE 4}
\end{equation}%
Using the integrabilty condition $\tilde{\nabla}^{pc}\tilde{\varphi}=0$, the
properties of $\tilde{h}$ and \eqref{CURVATURE 4} we have
\begin{equation}
\tilde{R}_{XY}\xi =-\eta (Y)(X-\tilde{h}X)+\eta (X)(Y-\tilde{h}Y)+\tilde{%
\varphi}((\tilde{\nabla}_{X}\tilde{h})Y-(\tilde{\nabla}_{Y}\tilde{h})X).
\label{CURVATURE 5}
\end{equation}%
Since $\tilde{h}$ is a symmetric operator we easily get
\begin{equation}
\tilde{g}((\tilde{\nabla}_{X}\tilde{h})Y-(\tilde{\nabla}_{Y}\tilde{h})X,\xi
)=\tilde{g}(((\tilde{\nabla}_{X}\tilde{h})-(\tilde{\nabla}_{Y}\tilde{h}))\xi
,Y-X)  \label{CURVATURE 6}
\end{equation}%
Using the formulas \eqref{nablaxi}, $\tilde{h}\xi =0$ and $\tilde{\varphi}%
\tilde{h}+\tilde{h}\tilde{\varphi}=0$ in \eqref{CURVATURE 6} we find
\begin{equation}
\tilde{g}((\tilde{\nabla}_{X}\tilde{h})Y-(\tilde{\nabla}_{Y}\tilde{h})X,\xi
)=2\tilde{g}(\tilde{\varphi}\tilde{h}^{2}X,Y).  \label{CURVATURE 7}
\end{equation}%
Applying $\tilde{\varphi}$ to \eqref{CURVATURE 5} and using $\tilde{\varphi}%
^{2}=I-\eta \otimes \xi $ and \eqref{CURVATURE 7} we obtain
\begin{equation}
(\tilde{\nabla}_{X}\tilde{h})Y-(\tilde{\nabla}_{Y}\tilde{h})X=\tilde{\varphi}%
\tilde{R}_{XY}\xi +2\tilde{g}(\tilde{\varphi}\tilde{h}^{2}X,Y)-\eta (X)%
\tilde{\varphi}(Y-\tilde{h}Y)+\eta (Y)\tilde{\varphi}(X-\tilde{h}X).
\label{CURVATURE 8}
\end{equation}%
Now we suppose that $P$ is a fixed point of $M$ and $X,Y,Z$ are vector
fields such that ($\tilde{\nabla}X)_{P}=(\tilde{\nabla}Y)_{P}=(\tilde{\nabla}%
Z)_{P}=0$. The Ricci identity for $\tilde{\varphi}$
\begin{equation*}
\tilde{R}_{XY}\tilde{\varphi}Z-\tilde{\varphi}\tilde{R}_{XY}Z=(\tilde{\nabla}%
_{X}\tilde{\nabla}_{Y}\tilde{\varphi})Z-(\tilde{\nabla}_{Y}\tilde{\nabla}_{X}%
\tilde{\varphi})Z-(\tilde{\nabla}_{[X,Y]}\tilde{\varphi})Z,
\end{equation*}%
at the point $P$, reduces to the form
\begin{equation}
\tilde{R}_{XY}\tilde{\varphi}Z-\tilde{\varphi}\tilde{R}_{XY}Z=\tilde{\nabla}%
_{X}(\tilde{\nabla}_{Y}\tilde{\varphi})Z-\tilde{\nabla}_{Y}(\tilde{\nabla}%
_{X}\tilde{\varphi})Z.  \label{CURVATURE 9}
\end{equation}%
By virtue of the integrabilty condition we have, at $P$,
\begin{align}
\tilde{R}_{XY}\tilde{\varphi}Z-\tilde{\varphi}\tilde{R}_{XY}Z& =\tilde{\nabla%
}_{X}(\tilde{\nabla}_{Y}\tilde{\varphi})Z-\tilde{\nabla}_{Y}(\tilde{\nabla}%
_{X}\tilde{\varphi})Z  \notag \\
& =\tilde{g}((\tilde{\nabla}_{X}\tilde{h})Y-(\tilde{\nabla}_{Y}\tilde{h}%
)X,Z)\xi -\eta (Z)((\tilde{\nabla}_{X}\tilde{h})Y-(\tilde{\nabla}_{Y}\tilde{h%
})X)  \notag \\
& \quad +\tilde{g}(Y-\tilde{h}Y,Z)\tilde{\varphi}(X-\tilde{h}X)-\tilde{g}(X-%
\tilde{h}X,Z)\tilde{\varphi}(Y-\tilde{h}Y)  \label{CURVATURE 10} \\
& \quad -\tilde{g}(\tilde{\varphi}(X-\tilde{h}X),Z)(Y-\tilde{h}Y)+\tilde{g}(%
\tilde{\varphi}(Y-\tilde{h}Y),Z)(X-\tilde{h}X).  \notag
\end{align}%
Using \eqref{CURVATURE 8} in \eqref{CURVATURE 10} we find
\begin{align}
\tilde{R}_{XY}\tilde{\varphi}Z-\tilde{\varphi}\tilde{R}_{XY}Z& =\bigl(\tilde{%
g}(\tilde{\varphi}\tilde{R}_{XY}\xi -\eta (X)\tilde{\varphi}(Y-\tilde{h}%
Y)+\eta (Y)\tilde{\varphi}(X-\tilde{h}X),Z)\bigr)\xi  \notag \\
& \quad -\eta (Z)\bigl(\tilde{\varphi}\tilde{R}_{XY}\xi -\eta (X)\tilde{%
\varphi}(Y-\tilde{h}Y)+\eta (Y)\tilde{\varphi}(X-\tilde{h}X)\bigr)  \notag \\
& \quad +\tilde{g}(Y-\tilde{h}Y,Z)\tilde{\varphi}(X-\tilde{h}X)-\tilde{g}(X-%
\tilde{h}X,Z)\tilde{\varphi}(Y-\tilde{h}Y)  \notag \\
& \quad -\tilde{g}(\tilde{\varphi}(X-\tilde{h}X),Z)(Y-\tilde{h}Y)+\tilde{g}(%
\tilde{\varphi}(Y-\tilde{h}Y),Z)(X-\tilde{h}X).  \label{CURVATURE 11}
\end{align}%
Using (\ref{G METRIC}) and the curvature tensor properties we get
\begin{equation*}
\tilde{g}(\tilde{\varphi}\tilde{R}_{\tilde{\varphi}X\tilde{\varphi}Y}Z,%
\tilde{\varphi}W)=-\tilde{g}(\tilde{R}_{ZW}\tilde{\varphi}X,\tilde{\varphi}%
Y)+\eta (\tilde{R}_{\tilde{\varphi}X\tilde{\varphi}Y}Z)\eta (W).
\end{equation*}%
Then by \eqref{G METRIC} and \eqref{CURVATURE 11} we get by a
straightforward calculation%
\begin{align}
\tilde{g}(\tilde{\varphi}\tilde{R}_{\tilde{\varphi}X\tilde{\varphi}Y}Z,%
\tilde{\varphi}W)& =\tilde{g}(\tilde{R}_{XY}Z,W)+\eta (W)\eta (\tilde{R}_{%
\tilde{\varphi}X\tilde{\varphi}Y}Z)  \notag \\
& \quad +\eta (Y)\bigl(-\eta (\tilde{R}_{ZW}X)-\eta (Z)\tilde{g}(W-\tilde{h}%
W,X)+\eta (W)\tilde{g}(Z-\tilde{h}Z,X)\bigr)  \notag \\
& \quad -\eta (X)\bigl(-\eta (\tilde{R}_{ZW}Y)-\eta (Z)\tilde{g}(W-\tilde{h}%
W,Y)+\eta (W)\tilde{g}(Z-\tilde{h}Z,Y)\bigr)  \label{CURVATURE 12} \\
& \quad +\tilde{g}(W-\tilde{h}W,X)\tilde{g}(Z-\tilde{h}Z,Y)-\tilde{g}(Z-%
\tilde{h}Z,X)\tilde{g}(W-\tilde{h}W,Y)  \notag \\
& \quad +\tilde{g}(\tilde{\varphi}(Z-\tilde{h}Z),X)\tilde{g}(W-\tilde{h}W,%
\tilde{\varphi}Y)-\tilde{g}(W-\tilde{h}W,\tilde{\varphi}X)\tilde{g}(\tilde{%
\varphi}(Z-\tilde{h}Z),Y).  \notag
\end{align}%
Replacing in (\ref{CURVATURE 11}) $X,Y$ by $\tilde{\varphi}X,\tilde{\varphi}%
Y $ respectively, and taking the inner product with $\tilde{\varphi}W$, we
get
\begin{align}
\tilde{g}(\tilde{R}_{\tilde{\varphi}X\tilde{\varphi}Y}\tilde{\varphi}Z-%
\tilde{\varphi}\tilde{R}_{\tilde{\varphi}X\tilde{\varphi}Y}Z,\tilde{\varphi}%
W)& =\tilde{g}(\tilde{\varphi}(X+\tilde{h}X),Z)\tilde{g}(\tilde{\varphi}(Y+%
\tilde{h}Y),W)  \notag  \label{CURVATURE 13} \\
& \quad -\tilde{g}(\tilde{\varphi}(Y+\tilde{h}Y),Z)\tilde{g}(\tilde{\varphi}%
(X+\tilde{h}X),W)-\eta (Y)\eta (W)\tilde{g}(X+\tilde{h}X,Z)  \notag \\
& \quad +\tilde{g}(X-\eta (X)\xi +\tilde{h}X,Z)\tilde{g}(Y+\tilde{h}%
Y,W)+\eta (X)\eta (W)\tilde{g}(Y+\tilde{h}Y,Z) \\
& \quad -\tilde{g}(Y-\eta (Y)\xi +\tilde{h}Y,Z)\tilde{g}(X+\tilde{h}%
X,W)+\eta (Z)\tilde{g}(\tilde{R}_{\tilde{\varphi}X\tilde{\varphi}Y}\xi ,W).
\notag
\end{align}%
Comparing (\ref{CURVATURE 12}) to (\ref{CURVATURE 13}) we get by direct
computation%
\begin{align}
\tilde{g}(\tilde{R}_{\tilde{\varphi}X\tilde{\varphi}Y}\tilde{\varphi}Z,%
\tilde{\varphi}W)& =\tilde{g}(\tilde{R}_{XY}Z,W)+\eta (W)\eta (\tilde{R}_{%
\tilde{\varphi}X\tilde{\varphi}Y}Z)-\eta (Z)\eta (\tilde{R}_{\tilde{\varphi}X%
\tilde{\varphi}Y}W)  \notag \\
& \quad +\eta (Y)\bigl(-\eta (\tilde{R}_{ZW}X)+2\eta (Z)\tilde{g}(\tilde{h}%
X,W)-2\eta (W)\tilde{g}(\tilde{h}Z,X)\bigr)  \label{CURVATURE 14} \\
& \quad -\eta (X)\bigl(-\eta (\tilde{R}_{ZW}Y)+2\eta (Z)\tilde{g}(\tilde{h}%
Y,W)-2\eta (W)\tilde{g}(\tilde{h}Z,Y)\bigr)  \notag \\
& \quad -2\tilde{g}(W,X)\tilde{g}(\tilde{h}Z,Y)-2\tilde{g}(Z,Y)\tilde{g}(%
\tilde{h}W,X)  \notag \\
& \quad +2\tilde{g}(W,Y)\tilde{g}(\tilde{h}Z,X)+2\tilde{g}(Z,X)\tilde{g}(%
\tilde{h}W,Y)  \notag
\end{align}%
Let \{$e_{i},\tilde{\varphi}e_{i},\xi \}$, $i\in \left\{ 1,\ldots ,n\right\}
$, be a local $\tilde{\varphi}$-basis. Setting $Y=Z=e_{i}$ in (\ref%
{CURVATURE 14}), we have%
\begin{align}
\sum\limits_{i=1}^{n}\tilde{g}(\tilde{R}_{\tilde{\varphi}X\tilde{\varphi}%
e_{i}}\tilde{\varphi}e_{i},\tilde{\varphi}W)& =\sum\limits_{i=1}^{n}\bigl(%
\tilde{g}(\tilde{R}_{Xe_{i}}e_{i},W)+\eta (W)\eta (\tilde{R}_{\tilde{\varphi}%
X\tilde{\varphi}e_{i}}e_{i})+\eta (X)\eta (\tilde{R}_{e_{i}W}e_{i})  \notag
\\
& \quad +2\eta (X)\eta (W)\tilde{g}(\tilde{h}e_{i},e_{i})-2\tilde{g}(W,X)%
\tilde{g}(\tilde{h}e_{i},e_{i})-2\tilde{g}(e_{i},e_{i})\tilde{g}(\tilde{h}%
W,X)  \label{CURVATURE 15} \\
& \quad +2\tilde{g}(W,e_{i})\tilde{g}(\tilde{h}e_{i},X)+2\tilde{g}(e_{i},X)%
\tilde{g}(\tilde{h}W,e_{i})\bigr).  \notag
\end{align}%
On the other hand, putting $Y=Z=\tilde{\varphi}e_{i}$ in \eqref{CURVATURE 14}%
, we get
\begin{align}
\sum\limits_{i=1}^{n}\tilde{g}(\tilde{R}_{\tilde{\varphi}Xe_{i}}e_{i},\tilde{%
\varphi}W)& =\sum\limits_{i=1}^{n}\bigl(\tilde{g}(\tilde{R}_{X\tilde{\varphi}%
e_{i}}\tilde{\varphi}e_{i},W)+\eta (W)\eta (\tilde{R}_{\tilde{\varphi}Xe_{i}}%
\tilde{\varphi}e_{i})+\eta (X)\eta (\tilde{R}_{\tilde{\varphi}e_{i}W}\tilde{%
\varphi}e_{i})  \notag \\
& \quad +2\eta (X)\eta (W)\tilde{g}(\tilde{h}\tilde{\varphi}e_{i},\tilde{%
\varphi}e_{i})-2\tilde{g}(W,X)\tilde{g}(\tilde{h}\tilde{\varphi}e_{i},\tilde{%
\varphi}e_{i})-2\tilde{g}(\tilde{\varphi}e_{i},\tilde{\varphi}e_{i})\tilde{g}%
(\tilde{h}W,X)  \label{CURVATURE 16} \\
& \quad +2\tilde{g}(W,\tilde{\varphi}e_{i})\tilde{g}(\tilde{h}\tilde{\varphi}%
e_{i},X)+2\tilde{g}(\tilde{\varphi}e_{i},X)\tilde{g}(\tilde{h}W,\tilde{%
\varphi}e_{i})\bigr).  \notag
\end{align}%
Using the definition of the Ricci operator, \eqref{CURVATURE 15} and %
\eqref{CURVATURE 16} it is not hard to prove that%
\begin{equation}
-\tilde{\varphi}\tilde{Q}\tilde{\varphi}X+\tilde{\varphi}\tilde{l}\tilde{%
\varphi}X+\tilde{Q}X=\tilde{l}X+\sum\limits_{i=1}^{n}(\eta (\tilde{R}_{%
\tilde{\varphi}Xe_{i}}\tilde{\varphi}e_{i})-\eta (\tilde{R}_{\tilde{\varphi}X%
\tilde{\varphi}e_{i}}e_{i}))\xi +\eta (X)\tilde{Q}\xi +4(n-1)\tilde{h}X.
\label{CURVATURE 17}
\end{equation}%
Finally, applying $\tilde{\varphi}$ to \eqref{CURVATURE 17} and using $%
\tilde{\varphi}^{2}=I-\eta \otimes \xi $, we obtain the assertion.
\end{proof}

\begin{corollary}
Let $(M^{2n+1},\tilde{\varphi},\xi ,\eta ,\tilde{g})$ be a paracontact $(%
\tilde{\kappa},\tilde{\mu})$-manifold. Then
\begin{equation}
\tilde{Q}\tilde{\varphi}-\tilde{\varphi}\tilde{Q}=2(2(n-1)+\tilde{\mu})%
\tilde{h}\tilde{\varphi}  \label{KMU QFI-FIQ}
\end{equation}
\end{corollary}

\begin{proof}
Using (\ref{CAPAR3}) and $\tilde{\varphi}\tilde{h}+\tilde{h}\tilde{\varphi}%
=0 $ we get $\tilde{l}\tilde{\varphi}-\tilde{\varphi}\tilde{l}=2\tilde{\mu}%
\tilde{h}\tilde{\varphi}$. On the other hand, by virtue of \eqref{Riczeta}
one can easily prove that both $\eta \otimes \tilde{\varphi}\tilde{Q}$ and $%
(\eta \circ \tilde{Q}\tilde{\varphi})\otimes \xi $ vanish. Thus
\eqref{KMU
QFI-FIQ} follows from \eqref{QFI-FIQ}.
\end{proof}

Recall that (\cite{CAP2}) an \emph{almost bi-paracontact structure} on a
contact manifold $(M,\eta )$ is a triplet $(\phi _{1},\phi _{2},\phi _{3})$
where $\phi _{3}$ is an almost contact structure compatible with the contact
form $\eta $, and $\phi _{1}$, $\phi _{2}$ are two anti-commuting tensors on
$M$ such that $\phi _{1}^{2}=\phi _{2}^{2}=I-\eta \otimes \xi $ and $\phi
_{1}\phi _{2}=\phi _{3}$. From the definition it easily follows that $\phi
_{1}\phi _{3}=-\phi _{3}\phi _{1}=\phi _{2}$ and $\phi _{3}\phi _{2}=-\phi
_{2}\phi _{3}=\phi _{1}$. Any almost bi-paracontact manifold is then endowed
with four distributions, ${\mathcal{D}}_{1}^{\pm}$, ${\mathcal{D}}_{2}^{\pm}$%
, given by the eigendistributions corresponding to the eigenvalues $\pm 1$
of $\phi _{1}$ and $\phi _{2}$, respectively. One proves that, for each $%
\alpha \in \left\{ 1,2\right\}$, ${\mathcal{D}}_{\alpha }^{+}$ and ${%
\mathcal{D}}_{\alpha }^{-}$ are transversal $n$-dimensional subbundles of
the contact distribution. In particular it follows that $\phi _{1}$ and $%
\phi _{2}$ are almost paracontact structures.

Now we prove that any paracontact $(\tilde{\kappa},\tilde{\mu})$-manifold
with $\tilde{\kappa}\neq -1$ is canonically endowed with an almost
bi-paracontact structure.

\begin{proposition}
\label{diagon} \label{bipara1} Let $(M,\tilde{\varphi},\xi ,\eta ,\tilde{g})$
be a paracontact $(\tilde{\kappa},\tilde{\mu})$-manifold. If $\tilde{\kappa}%
\neq -1$ then $M$ admits an almost bi-paracontact structure $(\phi _{1},\phi
_{2},\phi _{3})$ given by
\begin{equation}
\phi _{1}:=\tilde{\varphi},\ \ \phi _{2}:=\frac{1}{\sqrt{1+\tilde{\kappa}}}%
\tilde{h},\ \ \phi _{3}:=\frac{1}{\sqrt{1+\tilde{\kappa}}}\tilde{\varphi}%
\tilde{h}  \label{caseI}
\end{equation}%
in the case $\tilde{\kappa}>-1$, and
\begin{equation}
\phi _{1}:=\tilde{\varphi},\ \ \phi _{2}:=\frac{1}{\sqrt{-1-\tilde{\kappa}}}%
\tilde{\varphi}\tilde{h},\ \ \phi _{3}:=\frac{1}{\sqrt{-1-\tilde{\kappa}}}%
\tilde{h}  \label{caseII}
\end{equation}%
in the case $\tilde{\kappa}<-1$.
\end{proposition}

\begin{proof}
The proof follows by direct computations, using \eqref{H2} and the property $%
\tilde{\varphi}\tilde{h}=-\tilde{h}\tilde{\varphi}$.
\end{proof}

\begin{corollary}
\label{bipara2} Let $(M,\tilde{\varphi},\xi ,\eta ,\tilde{g})$ be a
paracontact $(\tilde{\kappa},\tilde{\mu})$-manifold such that $\tilde{\kappa}%
\neq -1$. Then the operator $\tilde{h}$ in the case $\tilde{\kappa}>-1 $ and
the operator $\tilde{\varphi}\tilde{h}$ in the case $\tilde{\kappa}<-1$ are
diagonalizable and admit three eigenvalues: $0$, associated to the
eigenvector $\xi $, $\tilde{\lambda}$ and -$\tilde{\lambda}$, of
multiplicity $n$, where $\tilde{\lambda}:=\sqrt{|1+\tilde{\kappa}|}$. The
corresponding eigendistributions ${\mathcal{D}}_{\tilde{h}}(0)=\mathbb{R}\xi$%
, ${\mathcal{D}}_{\tilde{h}}(\tilde{\lambda})$, ${\mathcal{D}}_{\tilde{h}}(-%
\tilde{\lambda})$ and ${\mathcal{D}}_{\tilde{\varphi}\tilde{h}}(0)=\mathbb{R}%
\xi$, ${\mathcal{D}}_{\tilde{\varphi}\tilde{h}}(\tilde{\lambda})$, ${%
\mathcal{D}}_{\tilde{\varphi}\tilde{h}}(-\tilde{\lambda})$ are mutually
orthogonal and one has $\tilde{\varphi}{\mathcal{D}}_{\tilde{h}}(\pm\tilde{%
\lambda})={\mathcal{D}}_{\tilde{h}}(\mp\tilde{\lambda})$ and $\tilde{\varphi}%
{\mathcal{D}}_{\tilde{\varphi}\tilde{h}}(\pm\tilde{\lambda})={\mathcal{D}}_{%
\tilde{\varphi}\tilde{h}}(\mp\tilde{\lambda})$. Furthermore,
\begin{equation}  \label{formula1}
{\mathcal{D}}_{\tilde h}(\pm\tilde\lambda)=\left\{X\pm\frac{1}{\sqrt{%
1+\tilde\kappa}}\tilde{h}X | X\in\Gamma({\mathcal{D}}^{\mp})\right\}
\end{equation}
in the case $\tilde\kappa>-1$, and
\begin{equation}  \label{formula2}
{\mathcal{D}}_{\tilde\varphi \tilde h}(\pm\tilde\lambda)=\left\{X\pm\frac{1}{%
\sqrt{-1-\tilde\kappa}}\tilde{\varphi}\tilde{h}X | X\in\Gamma({\mathcal{D}}%
^{\mp})\right\}
\end{equation}
in the case $\tilde\kappa<-1$, where ${\mathcal{D}}^{+}$ and ${\mathcal{D}}%
^{-}$ denote the eigendistributions of $\tilde\varphi$ corresponding to the
eigenvalues $1$ and $-1$, respectively. Finally any two among the four
distributions ${\mathcal{D}}^{+}$, ${\mathcal{D}}^{-}$, ${\mathcal{D}}%
_{\tilde h}(\tilde\lambda)$, ${\mathcal{D}}_{\tilde h}(-\tilde\lambda)$ in
the case $\tilde\kappa>-1$ or ${\mathcal{D}}^{+}$, ${\mathcal{D}}^{-}$, ${%
\mathcal{D}}_{\tilde\varphi\tilde h}(\tilde\lambda)$, ${\mathcal{D}}%
_{\tilde\varphi\tilde h}(-\tilde\lambda)$ in the case $\tilde\kappa<-1$ are
mutually transversal.
\end{corollary}

\begin{proof}
The first part of the statement follows from Proposition \ref{bipara1},
since $\tilde h=\tilde\lambda\phi_2$ and ${\mathcal{D}}_{\tilde
h}(\pm\tilde\lambda)={\mathcal{D}}_{2}^{\pm}$. We prove that ${\mathcal{D}}%
_{\tilde h}(\tilde\lambda)$ and ${\mathcal{D}}_{\tilde h}(-\tilde\lambda)$
are mutually orthogonal. Indeed, for any $X\in\Gamma({\mathcal{D}}_{\tilde
h}(\tilde\lambda))$ and $Y\in\Gamma({\mathcal{D}}_{\tilde
h}(-\tilde\lambda)) $ we have $\tilde\lambda\tilde g(X,Y)=\tilde g(\tilde h
X,Y)=\tilde g(X,\tilde h Y)=-\tilde\lambda\tilde g (X,Y)$, from which, since
$\tilde\lambda\neq 0$, we get $\tilde g(X,Y)=0$. Moreover, as $%
\tilde\varphi\tilde h=-\tilde h\tilde\varphi$ one has that if $\tilde h
X=\tilde\lambda X$ then $\tilde h\tilde\varphi X=-\tilde\lambda\tilde\varphi
X$, so that $\tilde\varphi{\mathcal{D}}_{\tilde h}(\pm\tilde\lambda)={%
\mathcal{D}}_{\tilde h}(\mp\tilde\lambda)$. The case $\kappa<-1$ can be
proved in a similar manner. Finally, \eqref{formula1} and \eqref{formula2}
follow from \cite[Proposition 3.3]{CAP2} and the last part is a direct
consequence of \cite[Proposition 3.2]{CAP2}.
\end{proof}

Thus paracontact $(\tilde\kappa,\tilde\mu)$-manifolds can be divided into
three main classes, according to the circumstance that $\tilde\kappa$ is
less, equal or greater than $-1$. According to \eqref{dhomothetic1}, one can
see that these three classes are preserved by ${\mathcal{D}}$-homothetic
deformations. \ Notice that the canonical paracontact metric structure $%
(\tilde\varphi_{1},\xi,\eta,\tilde{g}_1)$ on the tangent sphere bundle $%
T_{1}M$ of a manifold of constant curvature $c$ (cf. Theorem \ref{sphere3})
always satisfies $\tilde\kappa_{1}>-1$. Whereas for the other one, $%
(\tilde\varphi_{2},\xi,\eta,\tilde{g}_2)$, we have that $\tilde\kappa_2$ is
less, equal or greater than $-1$ if and only if, respectively, $c$ is less,
equal or greater than $0$. Thus $T_{1}M$ provides examples for all the above
three classes of paracontact $(\tilde\kappa,\tilde\mu)$-manifolds.

In the sequel, unless otherwise stated, we will always assume the index of ${%
\mathcal{D}}_{\tilde{h}}(\pm\lambda)$ (in the case $\tilde\kappa>-1$) and of
${\mathcal{D}}_{\tilde{\varphi}\tilde{h}}(\pm\lambda)$ (in the case $%
\tilde\kappa<-1$) to be constant.

Being $\tilde{h}$ (in the case $\tilde\kappa>-1$) or $\tilde\varphi\tilde{h}$
(in the case $\tilde\kappa<-1$) diagonalizable, one can easily prove the
following lemma.

\begin{lemma}
\label{basis} Let $(M,\tilde{\varphi},\xi ,\eta ,\tilde{g})$ be a
paracontact metric $(\tilde{\kappa},\tilde{\mu})$-manifold such that $%
\tilde\kappa\neq -1$. If $\tilde{\kappa}>-1$ (respectively, $\tilde{\kappa}%
<-1$), then there exists a local orthogonal $\tilde{\varphi}$-basis $%
\{X_{1},\ldots ,X_{n},Y_{1},\ldots ,Y_{n},\xi \}$ of eigenvectors of $\tilde{%
h}$ (respectively, $\tilde\varphi\tilde{h}$) such that $X_{1},\ldots
,X_{n}\in \Gamma ({\mathcal{D}}_{\tilde{h}}(\tilde{\lambda}))$
(respectively, $\Gamma ({\mathcal{D}}_{\tilde\varphi\tilde{h}}(\tilde{\lambda%
}))$), $Y_{1},\ldots ,Y_{n}\in \Gamma ({\mathcal{D}}_{\tilde{h}}(-\tilde{%
\lambda}))$ (respectively, $\Gamma ({\mathcal{D}}_{\tilde\varphi\tilde{h}}(-%
\tilde{\lambda}))$), and
\begin{equation}  \label{sign}
\tilde{g}(X_{i},X_{i})=-\tilde{g}(Y_{i},Y_{i})=\left\{
\begin{array}{ll}
1, & \hbox{for $1 \leq i \leq r$} \\
-1, & \hbox{for $r+1 \leq i \leq r+s$}%
\end{array}
\right.
\end{equation}
where $r=\emph{index}({\mathcal{D}}_{\tilde h}(-\tilde\lambda))$
(respectively, $r=\emph{index}({\mathcal{D}}_{\tilde\varphi\tilde
h}(-\tilde\lambda))$) and $s=n-r=\emph{index}({\mathcal{D}}_{\tilde
h}(\tilde\lambda))$ (respectively, $s=\emph{index}({\mathcal{D}}%
_{\tilde\varphi\tilde h}(\tilde\lambda))$).
\end{lemma}

As pointed out in \cite{MON}, there is a strict relationship between
paracontact metric geometry and the theory of Legendre foliations. Thus it
is interesting to investigate on the properties of the bi-Legendrian
structure $({\mathcal{D}}^{+},{\mathcal{D}}^{-})$ canonically associated to
a paracontact $(\tilde\kappa,\tilde\mu)$-manifold. In the next proposition
we prove that it is non-degenerate and we find necessary and sufficient
conditions for being positive or negative definite.

\begin{proposition}
\label{posneg1} Let $(M,\tilde{\varphi},\xi ,\eta ,\tilde{g})$ be a
paracontact metric $(\tilde{\kappa},\tilde{\mu})$-manifold such that $\tilde{%
\kappa}\neq -1$. Then the canonical Legendre foliations ${\mathcal{D}}^{+}$
and ${\mathcal{D}}^{-}$ given by the the $\pm 1$-eigendistributions of $%
\tilde{\varphi}$ are both non-degenerate. They are positive definite if and
only if, respectively, $\emph{index}({\mathcal{D}}_{\tilde{h}}(\tilde{\lambda%
}))=0$ in the case $\tilde\kappa>-1$ and $\emph{index}({\mathcal{D}}%
_{\tilde\varphi\tilde{h}}(\tilde{\lambda}))=n$ in the case $\tilde\kappa<-1$%
, and negative definite if and only if $\emph{index}({\mathcal{D}}_{\tilde{h}%
}(\tilde{\lambda}))=n$ in the case $\tilde\kappa>-1$ and $\emph{index}({%
\mathcal{D}}_{\tilde\varphi\tilde{h}}(\tilde{\lambda}))=0$ in the case $%
\tilde\kappa<-1$.
\end{proposition}

\begin{proof}
We consider the case $\tilde{\kappa}>-1$, the proof for the case $\tilde{%
\kappa}<-1$ being analogous. \ First notice that the Pang invariants
associated to the Legendre foliations ${\mathcal{D}}^{+}$ and ${\mathcal{D}}%
^{-}$ are given by
\begin{equation}
\Pi _{{\mathcal{D}}^{+}}(X,X^{\prime })=2\tilde{g}(\tilde{h}X,X^{\prime }),\
\ \ \Pi _{{\mathcal{D}}^{-}}(Y,Y^{\prime })=2\tilde{g}(\tilde{h}Y,Y^{\prime
}).  \label{invariant1}
\end{equation}%
Indeed, for any $X\in \Gamma ({\mathcal{D}}^{+})$ and for any $Y\in \Gamma ({%
\mathcal{D}}^{-})$ one has $\tilde{h}X=[\xi ,X]_{{\mathcal{D}}^{-}}$ and $%
\tilde{h}Y=-[\xi ,Y]_{{\mathcal{D}}^{+}}$ (see \cite[Corollary 3.2]{MON}).
Then we have that, for any $X,X^{\prime }\in \Gamma ({\mathcal{D}}^{+})$, $%
\Pi _{{\mathcal{D}}^{+}}(X,X^{\prime })=2d\eta ([\xi ,X],X^{\prime })=2%
\tilde{g}([\xi ,X],\tilde{\varphi}X^{\prime })=2\tilde{g}([\xi ,X],X^{\prime
})=2\tilde{g}([\xi ,X]_{-},X^{\prime })=2\tilde{g}(\tilde{h}X,X^{\prime })$.
Analogously one proves the other equality. \ Now let $\{X_{i},Y_{i}=\tilde{%
\varphi}X_{i},\xi \}$, ${i\in \{1,\ldots ,n\}}$, be a $\tilde{\varphi}$%
-basis of eigenvectors of $\tilde{h}$ as in Lemma \ref{basis}. Notice that
\begin{equation}
{\mathcal{D}}^{\pm }=\{X\pm \tilde{\varphi}X|X\in \Gamma ({\mathcal{D}}_{%
\tilde{h}}(\pm \tilde{\lambda}))\}.  \label{rel1}
\end{equation}%
This follows from \cite[Proposition 3.3]{CAP2} applied to the canonical
almost bi-paracontact structure attached to $(M,\tilde{\varphi},\xi ,\eta ,%
\tilde{g})$. Then by \eqref{invariant1}, since ${\tilde{h}}{\mathcal{D}}%
^{+}\subset {\mathcal{D}}^{-}$, we have that if $\Pi _{{\mathcal{D}}%
^{+}}(X,X^{\prime })=0$ for all $X,X^{\prime }\in \Gamma ({\mathcal{D}}^{+})$
necessarily $\tilde{h}X=0$. Hence $\tilde{h}^{2}X=0$ and, by \eqref{H2},
this implies that $X=0$. Therefore ${\mathcal{D}}^{+}$ is non-degenerate. A
similar proof works also for ${\mathcal{D}}^{-}$. \ Next, by \eqref{rel1},
in order to check whether $\mathcal{D}^{+}$ is positive or negative definite
it suffices to evaluate $\Pi _{{\mathcal{D}}^{+}}$ on the vector fields of
the form $X_{i}+\tilde{\varphi}X_{i}=X_{i}+Y_{i}$. Using \eqref{invariant1}
we have, for each $i\in \left\{ 1,\ldots ,n\right\} $,
\begin{equation*}
\Pi _{{\mathcal{D}}^{+}}(X_{i}+Y_{i},X_{i}+Y_{i})=2\tilde{g}(\tilde{h}X_{i}+%
\tilde{h}Y_{i},X_{i}+Y_{i})=2\tilde{g}(\tilde{\lambda}X_{i},X_{i})-2\tilde{g}%
(\tilde{\lambda}Y_{i},Y_{i})=\pm 4\tilde{\lambda},
\end{equation*}%
where the sign $\pm $ depends on the fact that $\text{index}({\mathcal{D}}_{%
\tilde{h}}(\tilde{\lambda}))=0$ or $\text{index}({\mathcal{D}}_{\tilde{h}}(%
\tilde{\lambda}))=n$, respectively. On the other hand as before, by %
\eqref{rel1} it is sufficient to evaluate $\Pi _{{\mathcal{D}}^{-}}$ on the
vector fields of the form $Y_{i}-\tilde{\varphi}Y_{i}=Y_{i}-X_{i}$. Then
\begin{equation*}
\Pi _{{\mathcal{D}}^{-}}(Y_{i}-Y_{i},Y_{i}-X_{i})=2\tilde{g}(\tilde{h}Y_{i}-%
\tilde{h}X_{i},Y_{i}-X_{i})=-2\tilde{g}(\tilde{\lambda}Y_{i},Y_{i})+2\tilde{g%
}(\tilde{\lambda}X_{i},X_{i})=\pm 4\tilde{\lambda},
\end{equation*}%
according to the circumstance that $\text{index}({\mathcal{D}}_{\tilde{h}}(%
\tilde{\lambda}))=0$ or $\text{index}({\mathcal{D}}_{\tilde{h}}(\tilde{%
\lambda}))=n$, respectively. Conversely, if $\Pi _{{\mathcal{D}}^{-}}$ is
positive definite then there exists a local basis of ${\mathcal{D}}^{-}$,
say $\{Z_{1},\ldots ,Z_{n}\}$, such that $\Pi _{{\mathcal{D}}%
^{-}}(Z_{i},Z_{j})=\delta _{ij}$. Then, taking \eqref{formula1} into
account, we put \ $X_{i}:=\sqrt{\tilde{\lambda}}\left( Z_{i}+\frac{1}{\tilde{%
\lambda}}\tilde{h}Z_{i}\right) $, for each $i\in \left\{ 1,\ldots ,n\right\}
$. Then $\left\{ X_{1},\ldots ,X_{n}\right\} $ is a local basis of ${%
\mathcal{D}}_{\tilde{h}}(\tilde{\lambda})$ such that
\begin{equation*}
\tilde{g}(X_{i},X_{j})=\tilde{\lambda}\left( \tilde{g}(Z_{i},Z_{j})+\frac{1}{%
\tilde{\lambda}^{2}}\tilde{g}(\tilde{h}Z_{i},\tilde{h}Z_{j})+\frac{2}{\tilde{%
\lambda}}\tilde{g}(\tilde{h}Z_{i},Z_{j})\right) =\Pi _{{\mathcal{D}}%
^{-}}(Z_{i},Z_{j})=\delta _{ij},
\end{equation*}%
since $\tilde{g}(X,X^{\prime })=0$ for all $X,X^{\prime }\in \Gamma ({%
\mathcal{D}}^{+})$. Thus $\text{index}({\mathcal{D}}_{\tilde{h}}(\tilde{%
\lambda}))=0$. The other case is analogous.
\end{proof}

\begin{remark}
Notice that in the course of the proof of Proposition \ref{posneg1} we have
proved in fact more than what we have stated. Namely, we have proved that $%
\Pi_{{\mathcal{D}}^+}$ and $\Pi_{{\mathcal{D}}^-}$ have the same signature,
which is given by $(\text{index}({\mathcal{D}}_{\tilde h}(\tilde\lambda)),%
\text{index(}{\mathcal{D}}_{\tilde h}(-\tilde\lambda)))$ in the case $%
\tilde\kappa>-1$ and by $(\text{index}({\mathcal{D}}_{\tilde\varphi\tilde
h}(\tilde\lambda)),\text{index(}{\mathcal{D}}_{\tilde\varphi\tilde
h}(-\tilde\lambda)))$ if $\tilde\kappa<-1$
\end{remark}

Proposition \ref{posneg1} motivates the following definition.

\begin{definition}
\label{posneg} A paracontact $(\tilde{\kappa},\tilde{\mu})$-manifold $(M,%
\tilde{\varphi},\xi ,\eta ,\tilde{g})$ such that $\tilde\kappa\neq -1$ will
be called \emph{positive definite} or \emph{negative definite} according to
the circumstance that the bi-Legendrian structure $({\mathcal{D}}^{+},{%
\mathcal{D}}^{-})$ canonically associated to $M$ is positive or negative
definite, respectively.
\end{definition}

Positive and negative definite paracontact $(\tilde{\kappa},\tilde{\mu})$%
-structures will play an important role $\S $ \ref{firstcase} and $\S $ \ref%
{secondcase}. \ We conclude the section with an example of negative definite
paracontact $(\tilde{\kappa},\tilde{\mu})$-manifold.

\begin{example}
\label{example1} Let $\mathfrak{g}$ be the Lie algebra with basis $%
\{e_{1},e_{2},e_{3},e_{4},e_{5}\}$ and Lie brackets
\begin{gather}  \label{constant1}
\lbrack e_{1},e_{5}]=\frac{\alpha \beta }{2}e_{2}+\frac{\alpha ^{2}}{2}%
e_{3},\ \ [e_{2},e_{5}]=-\frac{\alpha \beta }{2}e_{1}+\frac{\alpha ^{2}}{2}%
e_{4}, \\
\ \ [e_{3},e_{5}]=-\frac{\beta ^{2}}{2}e_{1}+\frac{\alpha \beta }{2}e_{4},\
\ [e_{4},e_{5}]=-\frac{\beta ^{2}}{2}e_{2}-\frac{\alpha \beta }{2}e_{3}, \\
\lbrack e_{1},e_{2}]=\alpha e_{2},\ \ [e_{1},e_{3}]=-\beta e_{2}+2e_{5},\ \
[e_{1},e_{4}]=0, \\
\lbrack e_{2},e_{3}]=\beta e_{1}-\alpha e_{4},\ \ [e_{2},e_{4}]=\alpha
e_{3}+2e_{5},\ \ [e_{3},e_{4}]=-\beta e_{3}  \label{constant2}
\end{gather}
where $\alpha,\beta$ are real numbers such that $\alpha^2-\beta^2\neq 0$. Let $G$ be a Lie group whose Lie algebra is $\mathfrak{g}$. Define on $G$ a
left invariant paracontact metric structure $(\tilde{\varphi},\xi ,\eta ,%
\tilde{g})$ by imposing that, at the identity,
\begin{equation*}
\tilde{g}(e_{1},e_{1})=-1,\ \ \tilde{g}(e_{2},e_{2})=-1,\ \ \tilde{g}%
(e_{3},e_{3})=1,\ \ \tilde{g}(e_{4},e_{4})=1,\ \ \tilde{g}(e_{5},e_{5})=1,\
\ \tilde{g}(e_{i},e_{j})=0\hbox{ $(i\neq j)$}
\end{equation*}%
and $\tilde{\varphi}e_{1}=e_{3}$, $\tilde{\varphi}e_{2}=e_{4}$, $\tilde{%
\varphi}e_{3}=e_{1}$, $\tilde{\varphi}e_{4}=e_{2}$, $\tilde{\varphi}e_{5}=0$%
, $\xi =e_{5}$ and $\eta =\tilde{g}(\cdot ,e_{5})$. \ Notice that $\tilde{h}%
e_{1}=\tilde{\lambda}e_{1}$, $\tilde{h}e_{2}=\tilde{\lambda}e_{2}$, $\tilde{h%
}\varphi e_{1}=-\tilde{\lambda}\tilde{\varphi}e_{1}$, $\tilde{h}\tilde{%
\varphi}e_{2}=-\tilde{\lambda}\tilde{\varphi}e_{2}$, $\tilde{h}\xi =0$. \
Now let $\tilde{\nabla}$ be the Levi-Civita connection of the
pseudo-Riemannian metric $\tilde{g}$ and $\tilde{R}$ be the curvature tensor
of $\tilde{g}$. Using the Koszul formula, we get
\begin{gather*}
\tilde{\nabla}_{e_{1}}\xi =(\tilde{\lambda}-1)\tilde{\varphi}e_{1},\ \
\tilde{\nabla}_{e_{2}}\xi =(\tilde{\lambda}-1)\tilde{\varphi}e_{2},\ \
\tilde{\nabla}_{\tilde{\varphi}e_{1}}\xi =-(1+\tilde{\lambda})e_{1},\ \
\tilde{\nabla}_{\tilde{\varphi}e_{2}}\xi =-(1+\tilde{\lambda})e_{2}, \\
\tilde{\nabla}_{\xi }e_{1}=-\frac{\alpha \beta }{2}e_{2}-\frac{\tilde{\mu}}{2%
}\tilde{\varphi}e_{1},\ \ \tilde{\nabla}_{\xi }e_{2}=\frac{\alpha \beta }{2}%
e_{1}-\frac{\tilde{\mu}}{2}\tilde{\varphi}e_{2},\ \ \tilde{\nabla}_{\xi }%
\tilde{\varphi}e_{1}=-\frac{\tilde{\mu}}{2}e_{1}-\frac{\alpha \beta }{2}%
\tilde{\varphi}e_{2},\ \ \tilde{\nabla}_{\xi }\tilde{\varphi}e_{2}=-\frac{%
\tilde{\mu}}{2}e_{2}+\frac{\alpha \beta }{2}\tilde{\varphi}e_{1}, \\
\tilde{\nabla}_{e_{1}}e_{1}=0,\ \ \tilde{\nabla}_{e_{1}}e_{2}=0,\ \ \tilde{%
\nabla}_{e_{1}}\tilde{\varphi}e_{1}=-(\tilde{\lambda}-1)\xi ,\ \ \tilde{%
\nabla}_{e_{1}}\tilde{\varphi}e_{2}=0, \\
\tilde{\nabla}_{e_{2}}e_{1}=-\alpha e_{2},\ \ \tilde{\nabla}%
_{e_{2}}e_{2}=\alpha e_{1},\ \ \tilde{\nabla}_{e_{2}}\tilde{\varphi}%
e_{1}=-\alpha \tilde{\varphi}e_{2},\ \ \tilde{\nabla}_{e_{2}}\tilde{\varphi}%
e_{2}=\alpha \tilde{\varphi}e_{1}-(\tilde{\lambda}-1)\xi , \\
\tilde{\nabla}_{\tilde{\varphi}e_{1}}e_{1}=\beta e_{2}-(1+\tilde{\lambda}%
)\xi ,\ \ \tilde{\nabla}_{\tilde{\varphi}e_{1}}e_{2}=-\beta e_{1},\ \ \tilde{%
\nabla}_{\tilde{\varphi}e_{1}}\tilde{\varphi}e_{1}=\beta \tilde{\varphi}%
e_{2},\ \ \tilde{\nabla}_{\tilde{\varphi}e_{1}}\tilde{\varphi}e_{2}=-\beta
\tilde{\varphi}e_{1}, \\
\tilde{\nabla}_{\tilde{\varphi}e_{2}}e_{1}=0,\ \ \tilde{\nabla}_{\tilde{%
\varphi}e_{2}}e_{2}=-(1+\tilde{\lambda})\xi ,\ \ \tilde{\nabla}_{\tilde{%
\varphi}e_{2}}\tilde{\varphi}e_{1}=0,\ \ \tilde{\nabla}_{\tilde{\varphi}%
e_{2}}\tilde{\varphi}e_{2}=0,
\end{gather*}%
where $\tilde{\lambda}=\frac{\alpha ^{2}+\beta ^{2}}{4}$, $\tilde{\kappa}=%
\frac{(\alpha ^{2}+\beta ^{2})^{2}-16}{16}$ and $\tilde{\mu}=\frac{\alpha
^{2}-\beta ^{2}}{2}+2$. From the above relations it can be easily proved
checked that $G$ is a paracontact $(\tilde{\kappa},%
\tilde{\mu})$-manifold.
\end{example}

\section{Paracontact $(\tilde{\protect\kappa},\tilde{\protect\mu})$%
-manifolds with $\tilde{\protect\kappa}>-1$}

\label{firstcase}

In this section we deal with paracontact $(\tilde\kappa,\tilde\mu)$%
-manifolds such that $\tilde{\kappa}>-1$. In this case, according to
Corollary \ref{bipara2}, $\tilde{h}$ is diagonalizable with eigenvectors $0$%
, $\pm \tilde{\lambda}$, where $\tilde{\lambda}:=\sqrt{1+\tilde{\kappa}}$.
Our first result concerns some remarkable properties of the distributions
defined by the eigenspaces of $\tilde h$.

\begin{theorem}
\label{para1} Let $(M,\tilde{\varphi},\xi ,\eta ,\tilde{g})$ be a
paracontact metric $(\tilde{\kappa},\tilde{\mu})$-manifold with $\tilde{%
\kappa}>-1$. Then the eigendistributions ${\mathcal{D}}_{\tilde{h}}(\tilde{%
\lambda})$ and ${\mathcal{D}}_{\tilde{h}}(-\tilde{\lambda})$ of $\tilde{h}$
are integrable and define two totally geodesic Legendre foliations of $M$.
Moreover, for any $X\in \Gamma ({\mathcal{D}}_{\tilde{h}}(\tilde{\lambda}))$%
, $Y\in \Gamma ({\mathcal{D}}_{\tilde{h}}(-\tilde{\lambda}))$ $\tilde{\nabla}%
_{X}Y$ (respectively, $\tilde{\nabla}_{Y}X$) has no components along ${%
\mathcal{D}}_{\tilde{h}}(\tilde{\lambda})$ (respectively, ${\mathcal{D}}_{%
\tilde{h}}(-\tilde{\lambda})$).
\end{theorem}

\begin{proof}
Replacing $Y$ with $\tilde{\varphi}Y$ in \eqref{NAMLA X H}, we get
\begin{equation*}
(\tilde{\nabla}_{X}\tilde{h})\tilde{\varphi}Y-(\tilde{\nabla}_{\tilde\varphi
Y}\tilde{h})X =-(1+\tilde{\kappa})(2\tilde{g}(X,Y)\xi -3\eta (X)\eta (Y)\xi
+\eta (X)Y)-(1-\tilde{\mu})\eta (X)\tilde{h}Y
\end{equation*}
for any $X,Y\in \Gamma (TM)$. Then, for any $X,Y,Z\in \Gamma ({\mathcal{D}})$
we have $\tilde{g}((\tilde{\nabla}_{X}\tilde{h})\tilde{\varphi}Y)-(\tilde{%
\nabla}_{\tilde{\varphi}Y}\tilde{h})X),Z)=0$, which is equivalent to
\begin{equation}  \label{lemmaf1}
\tilde{g}(\tilde{\nabla}_{X}\tilde{h}\tilde{\varphi}Y-\tilde{h}\tilde{\nabla}%
_{X}\tilde{\varphi}Y-\tilde{\nabla}_{\tilde{\varphi}Y}\tilde{h}X+\tilde{h}%
\tilde{\nabla}_{\tilde{\varphi}Y}X,Z)=0.
\end{equation}%
Now taking $X,Y,Z\in \Gamma ({\mathcal{D}}_{\tilde{h}}(\tilde{\lambda}))$ in %
\eqref{lemmaf1} it follows that $-2\tilde{\lambda}\tilde{g}(\tilde{\nabla}%
_{X}\tilde{\varphi}Y,Z)=0$. Since we are assuming $\tilde{\lambda}\neq 0$,
we get $0=\tilde{g}(\tilde{\nabla}_{X}\tilde{\varphi}Y,Z)=X(\tilde{g}(\tilde{%
\varphi}Y,Z))-\tilde{g}(\tilde{\varphi}Y,\tilde{\nabla}_{X}Z)=-\tilde{g}(%
\tilde{\nabla}_{X}Z,\tilde{\varphi}Y)$. Thus $\tilde{\nabla}_{X}Z$ is
orthogonal to ${\mathcal{D}}_{\tilde{h}}(-\tilde{\lambda})$. On the other
hand, $\tilde{g}(\tilde{\nabla}_{X}Z,\xi )=X(\tilde{g}(Z,\xi ))-\tilde{g}(Z,%
\tilde{\nabla}_{X}\xi )=-\tilde{g}(Z,\tilde{\varphi}X)-\tilde{\lambda}\tilde{%
g}(Z,\tilde{\varphi}X)=0$, so we conclude that $\tilde{\nabla}_{X}Z\in
\Gamma ({\mathcal{D}}_{\tilde{h}}(\tilde{\lambda}))$. Analogously, if $%
X,Z\in \Gamma ({\mathcal{D}}_{\tilde{h}}(-\tilde{\lambda}))$ then $\tilde{%
\nabla}_{X}Z\in \Gamma ({\mathcal{D}}_{\tilde{h}}(\tilde{\lambda}))$. Hence $%
{\mathcal{D}}_{\tilde{h}}(\tilde{\lambda})$ and ${\mathcal{D}}_{\tilde{h}}(-%
\tilde{\lambda})$ are totally geodesic. Next, if $X\in \Gamma ({\mathcal{D}}%
_{\tilde{h}}(\tilde{\lambda}))$ and $Y\in \Gamma ({\mathcal{D}}_{\tilde{h}}(-%
\tilde{\lambda}))$ then for all $Z\in \Gamma ({\mathcal{D}}_{\tilde{h}}(%
\tilde{\lambda}))$ one has $\tilde{g}(\tilde{\nabla}_{X}Y,Z)=X(\tilde{g}%
(Y,Z))-\tilde{g}(Y,\tilde{\nabla}_{X}Z)=0$, since $\tilde{\nabla}_{X}Z\in
\Gamma ({\mathcal{D}}_{\tilde{h}}(\tilde{\lambda}))$. Thus $\tilde{\nabla}%
_{X}Y\in \Gamma ({\mathcal{D}}_{\tilde{h}}(-\tilde{\lambda})\oplus \mathbb{R}%
\xi )$. In a similar manner one can prove that $\tilde{\nabla}_{Y}X$ has no
components along ${\mathcal{D}}_{\tilde{h}}(-\tilde{\lambda})$. In
particular, the total geodesicity of ${\mathcal{D}}_{\tilde{h}}(\tilde{%
\lambda})$ and ${\mathcal{D}}_{\tilde{h}}(-\tilde{\lambda})$ implies that
they are involutive distributions. Moreover they are also $n$-dimensional
because of \cite[Proposition 3.2]{CAP2}. Hence they define two Legendre
foliations on $M$.
\end{proof}

The geometry of a Legendre foliations is mainly described by its Pang
invariant \eqref{panginvariant}. Thus we find the explicit expression of the
Pang invariants of the Legendre foliations ${\mathcal{D}}_{\tilde
h}(\tilde\lambda)$ and ${\mathcal{D}}_{\tilde h}(-\tilde\lambda)$.

\begin{theorem}
\label{pang3} The Pang invariants of the Legendre foliations ${\mathcal{D}}%
_{\tilde h}(\tilde\lambda)$ and ${\mathcal{D}}_{\tilde h}(-\tilde\lambda)$
are given by
\begin{gather}
\Pi_{{\mathcal{D}}_{\tilde h}(\tilde\lambda)}=-2\left(1-\frac{\tilde\mu}{2}-%
\sqrt{1+\tilde\kappa}\right)\tilde{g}|_{{\mathcal{D}}_{\tilde
h}(\tilde\lambda)\times{\mathcal{D}}_{\tilde h}(\tilde\lambda)}
\label{pang1} \\
\Pi_{{\mathcal{D}}_{\tilde h}(-\tilde\lambda)}=-2\left(1-\frac{\tilde\mu}{2}+%
\sqrt{1+\tilde\kappa}\right)\tilde{g}|_{{\mathcal{D}}_{\tilde
h}(-\tilde\lambda)\times{\mathcal{D}}_{\tilde h}(-\tilde\lambda)}.
\label{pang2}
\end{gather}
\end{theorem}

\begin{proof}
Let $X$ be a section of ${\mathcal{D}}_{\tilde{h}}(\tilde\lambda)$. Then by %
\eqref{CAPAR3} we have
\begin{equation}  \label{pass1}
\tilde{R}_{X\xi}\xi=\tilde\kappa X + \tilde\mu \tilde{h} X.
\end{equation}
On the other hand,
\begin{align}  \label{pass2}
\tilde{R}_{X\xi}\xi&=\tilde\nabla_{X}\tilde\nabla_{\xi}\xi-\tilde\nabla_{%
\xi}\tilde\nabla_{X}\xi-\tilde\nabla_{[X,\xi]}\xi  \notag \\
&=\tilde\nabla_{\xi}\tilde\varphi X - \tilde\nabla_{\xi}\tilde\varphi\tilde{h%
}X - \tilde\varphi[\xi,X] + \tilde\varphi\tilde{h}[\xi,X]  \notag \\
&=(1-\tilde\lambda)\tilde\nabla_{\tilde\varphi X}\xi +
(1-\tilde\lambda)[\xi,\tilde\varphi X] - \tilde\varphi[\xi,X] + \tilde\varphi%
\tilde{h}[\xi,X]  \notag \\
&=(1-\tilde\lambda)(-\tilde\varphi^{2}X+\tilde\varphi\tilde{h}\tilde\varphi
X) + (1-\tilde\lambda)[\xi,\tilde\varphi X]-\tilde\varphi[\xi,X] +
\tilde\varphi\tilde{h}[\xi,X] + \tilde\varphi\tilde{h}[\xi,X]  \notag \\
&=-X+\tilde\lambda^{2}X+2\tilde{h}X-\tilde\lambda[\xi,\tilde\varphi X]%
+\tilde\varphi\tilde{h}[\xi,X]+\tilde\lambda\tilde\varphi[\xi,X]%
-\tilde\lambda\tilde\varphi[\xi,X]  \notag \\
&=-X+\tilde\lambda^{2}X+2\tilde\lambda X-2\tilde\lambda\tilde{h}%
X+\tilde\varphi\tilde{h}[\xi,X]-\tilde\lambda\tilde\varphi[\xi,X]  \notag \\
&=-(1-\tilde\lambda)^{2}X + \tilde\varphi\tilde{h}[\xi,X] -
\tilde\lambda\tilde\varphi[\xi,X].
\end{align}
By \eqref{pass1} and \eqref{pass2} it follows that \ $\tilde\varphi\tilde{h}%
[\xi,X] = \tilde\lambda[\xi,X] +
(\tilde\kappa+\tilde\mu\tilde\lambda+(1-\tilde\lambda)^2)X$. \ By applying $%
\tilde\varphi$ we obtain
\begin{equation*}
\tilde{h}[\xi,X]=\tilde\lambda[\xi,X]+(\tilde\kappa+\tilde\mu\tilde%
\lambda+(1-\tilde\lambda)^{2})\tilde\varphi X.
\end{equation*}
Notice that, since $i_{\xi}d\eta=0$, the Reeb vector field is an
infinitesimal automorphism with respect to the contact distribution, so that
$[\xi,X]$ is still a section of $\mathcal{D}$. Then by decomposing $[\xi,X]$
in its components along ${\mathcal{D}}_{\tilde h}(\tilde\lambda)$ and ${%
\mathcal{D}}_{\tilde h}(-\tilde\lambda)$, from the last equation it follows
that
\begin{equation*}
[\xi,X]_{{\mathcal{D}}_{\tilde h}(-\tilde\lambda)}=-\frac{%
\tilde\lambda^{2}-2\tilde\lambda+1+\tilde\kappa+\tilde\mu\tilde\lambda}{%
2\tilde\lambda}\tilde\varphi{X}=\left(1-\frac{\tilde\mu}{2}-\sqrt{%
1+\tilde\kappa}\right)\tilde\varphi{X}.
\end{equation*}
Therefore for any $X,X^{\prime}\in\Gamma({\mathcal{D}}_{\tilde
h}(\tilde\lambda))$
\begin{align*}
\Pi_{{\mathcal{D}}_{\tilde h}(\tilde\lambda)}(X,X^{\prime})&=2\tilde{g}%
([\xi,X]_{{\mathcal{D}}_{\tilde h}(-\tilde\lambda)},\tilde{\varphi}%
X^{\prime})=-2\left(1-\frac{\tilde\mu}{2}-\sqrt{1+\tilde\kappa}\right)\tilde{%
g}(X,X^{\prime})
\end{align*}
The proof of \eqref{pang2} is similar.
\end{proof}

\begin{proposition}
\label{para2} In any paracontact $(\tilde{\kappa},\tilde{\mu})$-manifold $(M,%
\tilde{\varphi},\xi ,\eta ,\tilde{g})$ with $\tilde{\kappa}>-1$ one has
\begin{equation}
(\tilde{\nabla}_{X}\tilde{h})Y=-\tilde{g}(X,\tilde{\varphi}{\tilde{h}}^{2}Y+%
\tilde{\varphi}\tilde{h}Y)\xi +\eta (Y)((1+\tilde{\kappa})\tilde{\varphi}X-%
\tilde{\varphi}\tilde{h}X)-\tilde{\mu}\eta (X)\tilde{\varphi}\tilde{h}Y
\label{nablah}
\end{equation}%
for any $X,Y\in \Gamma (TM)$.
\end{proposition}

\begin{proof}
Because of Theorem \ref{para1} we have that
\begin{equation}
(\tilde{\nabla}_{X}\tilde{h})Y=0  \label{step1}
\end{equation}%
for any $X,Y\in \Gamma ({\mathcal{D}}_{\tilde{h}}(\tilde{\lambda}))$ or $%
X,Y\in \Gamma ({\mathcal{D}}_{\tilde{h}}(-\tilde{\lambda}))$. Now suppose
that $X\in \Gamma ({\mathcal{D}}_{\tilde{h}}(\tilde{\lambda}))$ and $Y\in
\Gamma ({\mathcal{D}}_{\tilde{h}}(-\tilde{\lambda}))$. Let $\{X_{1},\ldots
,X_{n},\tilde{\varphi}X_{1},\ldots ,\tilde{\varphi}X_{n},\xi \}$ be a $%
\tilde{\varphi}$-basis as in Lemma \ref{basis}. Then according to %
\eqref{sign} we have
\begin{align*}
\tilde{h}\tilde{\nabla}_{X}Y& =\tilde{h}\left(-\sum_{i=1}^{r}\tilde{g}(%
\tilde{\nabla}_{X}Y,\tilde{\varphi}X_{i})\tilde{\varphi}X_{i}+%
\sum_{i=r+1}^{n}\tilde{g}(\tilde{\nabla}_{X}Y,\tilde{\varphi}X_{i})\tilde{%
\varphi}X_{i}+\tilde{g}(\tilde{\nabla}_{X}Y,\xi )\xi \right) \\
&=-\tilde\lambda\tilde\varphi\sum_{i=1}^{r}\tilde{g}(\tilde\varphi\tilde{%
\nabla}_{X}Y,X_{i})X_{i}+\tilde\lambda\tilde\varphi\sum_{i=r+1}^{n}\tilde{g}%
(\tilde\varphi\tilde{\nabla}_{X}Y,X_{i})X_{i} \\
& =-\tilde{\lambda}{\tilde{\varphi}}^{2}\tilde{\nabla}_{X}Y \\
& =-\tilde{\lambda}(\tilde{\nabla}_{X}Y-\tilde{g}(\tilde{\nabla}_{X}Y,\xi
)\xi ) \\
& =-\tilde{\lambda}(\tilde{\nabla}_{X}Y+\tilde{g}(Y,\tilde{\nabla}_{X}\xi
)\xi ) \\
& =-\tilde{\lambda}(\tilde{\nabla}_{X}Y-\tilde{g}(Y,\tilde{\varphi}X)\xi +%
\tilde{g}(Y,\tilde{\varphi}\tilde{h}X)\xi ) \\
& =\tilde\nabla _{X}\tilde{h}Y-\tilde{\lambda}(1-\tilde{\lambda})\tilde{g}(X,%
\tilde{\varphi}Y)\xi .
\end{align*}%
Thus for any $X\in \Gamma ({\mathcal{D}}_{\tilde{h}}(\tilde{\lambda}))$ and $%
Y\in \Gamma ({\mathcal{D}}_{\tilde{h}}(-\tilde{\lambda}))$ one has
\begin{equation}
(\tilde{\nabla}_{X}\tilde{h})Y=\tilde{\lambda}(1-\tilde{\lambda})\tilde{g}(X,%
\tilde{\varphi}Y)\xi  \label{step2}
\end{equation}%
and in a similar way one can prove that
\begin{equation}
(\tilde{\nabla}_{Y}\tilde{h})X=\tilde{\lambda}(1+\tilde{\lambda})\tilde{g}(X,%
\tilde{\varphi}Y)\xi .  \label{step3}
\end{equation}%
Now let $X$ and $Y$ any two vector fields on $M$. We decompose $X$ and $Y$
as $X=X_{+}+X_{-}+\eta (X)\xi $, $Y=Y_{+}+Y_{-}+\eta (Y)\xi $, according to
the decomposition $TM={\mathcal{D}}_{\tilde{h}}(\tilde{\lambda})\oplus {%
\mathcal{D}}_{\tilde{h}}(-\tilde{\lambda})\oplus \mathbb{R}\xi $. Using %
\eqref{NMBLA ZETAH}, \eqref{step1}, \eqref{step2}, \eqref{step3}, after a
straightforward computation we get
\begin{align*}
(\tilde{\nabla}_{X}\tilde{h})Y& =\tilde{\lambda}(1-\tilde{\lambda})\tilde{g}%
(X_{+},\tilde{\varphi}Y_{-})\xi -\tilde{\lambda}(1+\tilde{\lambda})\tilde{g}%
(X_{-},\tilde{\varphi}Y_{+})\xi +\eta (Y)\tilde{h}\tilde{\varphi}X \\
& \quad +(1+\tilde{\kappa})\eta (Y)\tilde{\varphi}X+\tilde{\mu}\eta (X)%
\tilde{h}\tilde{\varphi}Y.
\end{align*}%
But $\tilde{\lambda}(1-\tilde{\lambda})\tilde{g}(X_{+},\tilde{\varphi}Y_{-})-%
\tilde{\lambda}(1+\tilde{\lambda})\tilde{g}(X_{-},\tilde{\varphi}Y_{+})=%
\tilde{\lambda}\tilde{g}(X_{+},\tilde{\varphi}Y_{-})-\tilde{\lambda}\tilde{g}%
(X_{-},\tilde{\varphi}Y_{+})-{\tilde{\lambda}}^{2}(\tilde{g}(X_{+},\tilde{%
\varphi}Y_{-})+\tilde{g}(X_{-},\tilde{\varphi}Y_{+}))=\tilde{g}(\tilde{h}X,%
\tilde{\varphi}Y)-{\tilde{\lambda}}^{2}\tilde{g}(X,\tilde{\varphi}Y)$.
Therefore \eqref{nablah} follows.
\end{proof}

We recall the following general result.

\begin{theorem}[\protect\cite{CAP1}]
\label{legendre1} Let $(M,\eta)$ be a contact manifold endowed with a
bi-Legendrian structure $({\mathcal{F}}_1,{\mathcal{F}}_2)$ such that $%
\nabla^{bl}\Pi_{{\mathcal{F}}_1}=\nabla^{bl}\Pi_{{\mathcal{F}}_2}=0$, where $%
\nabla^{bl}$ denotes the bi-Legendrian connection associated to $({\mathcal{F%
}}_1,{\mathcal{F}}_2)$. Assume that one of the following conditions holds

\begin{itemize}
\item[(I)] ${\mathcal{F}}_1$ and ${\mathcal{F}}_2$ are positive definite and
there exist two positive numbers $a$ and $b$ such that $\overline\Pi_{{%
\mathcal{F}}_1}=ab\overline\Pi_{{\mathcal{F}}_2}$ on $T{\mathcal{F}}_1$ and $%
\overline\Pi_{{\mathcal{F}}_2}=ab\overline\Pi_{{\mathcal{F}}_1}$ on $T{%
\mathcal{F}}_2$,

\item[(II)] ${\mathcal{F}}_1$ is positive definite, ${\mathcal{F}}_2$ is
negative definite and there exist $a>0$ and $b<0$ such that $\overline\Pi_{{%
\mathcal{F}}_1}=ab\overline\Pi_{{\mathcal{F}}_2}$ on $T{\mathcal{F}}_1$ and $%
\overline\Pi_{{\mathcal{F}}_2}=ab\overline\Pi_{{\mathcal{F}}_1}$ on $T{%
\mathcal{F}}_2$,

\item[(III)] ${\mathcal{F}}_1$ and ${\mathcal{F}}_2$ are negative definite
and there exist two negative numbers $a$ and $b$ such that $\overline\Pi_{{%
\mathcal{\ }F}_1}=ab\overline\Pi_{{\mathcal{F}}_2}$ on $T{\mathcal{F}}_1$
and $\overline\Pi_{{\mathcal{F}}_2}=ab\overline\Pi_{{\mathcal{F}}_1}$ on $T{%
\mathcal{F}}_2$.
\end{itemize}

Then $(M,\eta)$ admits a compatible contact metric structure $%
(\varphi_{a,b},\xi,\eta,g_{a,b})$ such that

\begin{enumerate}
\item[(i)] if $a=b$, $(M,\varphi_{a,b},\xi,\eta,g_{a,b})$ is a Sasakian
manifold;

\item[(ii)] if $a\neq b$, $(M,\varphi_{a,b},\xi,\eta,g_{a,b})$ is a contact
metric $(\kappa_{a,b},\mu_{a,b})$-manifold, whose associated bi-Legendrian
structure is $({\mathcal{F}}_{1},{\mathcal{F}}_{2})$, where
\begin{equation}
\kappa_{a,b}=1-\frac{(a-b)^{2}}{16},\ \ \ \mu_{a,b}=2-\frac{a+b}{2}.
\label{costanti0}
\end{equation}
\end{enumerate}
\end{theorem}

Notice that in the proof of Theorem \ref{legendre1} the assumptions (I),
(II) or (III) are used for constructing the compatible metric structure,
whereas the hypothesis $\nabla ^{bl}\Pi _{{\mathcal{F}}_{1}}=\nabla ^{bl}\Pi
_{{\mathcal{F}}_{2}}=0$ is necessary only for proving that such contact
metric structure satisfies a nullity condition.

Now we prove one of the main results of the section, which puts in relation
the theory of paracontact $(\tilde\kappa,\tilde\mu)$-manifolds with contact
Riemannian geometry.

\begin{theorem}
\label{main1} Let $(M,\tilde{\varphi},\xi ,\eta ,\tilde{g})$ be a positive
or negative definite paracontact $(\tilde{\kappa},\tilde{\mu})$-manifold
such that $\tilde{\kappa}>-1$. Then $M$ admits a family of contact
Riemannian structures $(\varphi _{a,b},\xi ,\eta ,g_{a,b})$ parameterized by
real numbers $a$ and $b $ satisfying the relation $ab=4(1+\tilde{\kappa})$.
Each contact metric structure $(\varphi _{a,b},\xi ,\eta ,g_{a,b})$ is
explicitly given by
\begin{gather}
\varphi _{a,b}=\left\{
\begin{array}{ll}
\frac{b}{2(1+\tilde{\kappa})}\tilde{h}, & \hbox{on ${\mathcal D}^{+}$} \\
-\frac{a}{2(1+\tilde{\kappa})}\tilde{h}, & \hbox{on ${\mathcal D}^{-}$} \\
0, & \hbox{on $\mathbb{R}\xi$}%
\end{array}
\right.  \label{phi} \\
g_{a,b}=\left\{
\begin{array}{ll}
\frac{2}{a}\tilde{g}(\tilde{h}\cdot ,\cdot ), &
\hbox{on ${\mathcal
D}^{+}\times{\mathcal D}^{+}$} \\
\frac{2}{b}\tilde{g}(\tilde{h}\cdot ,\cdot ), &
\hbox{on ${\mathcal
D}^{-}\times{\mathcal D}^{-}$} \\
\eta \otimes \eta , & \hbox{otherwise}%
\end{array}
\right.  \label{metric}
\end{gather}
Furthermore, if $\tilde{\mu}=2$ then each structure $(\varphi _{a,b},\xi
,\eta ,g_{a,b})$ is a contact metric $(\kappa _{a,b},\mu _{a,b})$-structure
(eventually Sasakian if $a=b$) with $\kappa _{a,b}=1-\frac{(a-b)^{2}}{16}$
and $\mu _{a,b}=2-\frac{a-b}{2}$.
\end{theorem}

\begin{proof}
The result will follow once we will have proved that the canonical
bi-Legendrian structure $({\mathcal{D}}^{+},{\mathcal{D}}^{-})$ of $%
(M,\tilde\varphi,\xi,\eta,\tilde{g})$ satisfies one among the assumptions
(I), (II), (III) of Theorem \ref{legendre1}. The positive/negative
definiteness of ${\mathcal{D}}^{+}$ and ${\mathcal{D}}^{-}$ is ensured by
the assumption that the paracontact $(\tilde\kappa,\tilde\mu)$-structure $%
(\tilde\varphi,\xi,\eta,\tilde{g})$ is positive or negative definite. Hence
it remains to prove the existence of real numbers $a$ and $b$ such that the
Pang-Libermann invariants $\overline{\Pi}_{{\mathcal{D}}^{+}}$ and $%
\overline{\Pi}_{{\mathcal{D}}^{-}}$ are related each other as in the
assumptions (I) or (III) of Theorem \ref{legendre1}. First we find the
explicit expression of the Libermann map $\Lambda_{{\mathcal{D}}%
^{-}}:TM\longrightarrow{\mathcal{D}}^{+}$. By definition $\Lambda_{{\mathcal{%
D}}^{+}}\xi=0$ and $\Lambda_{{\mathcal{D}}^{-}}Y=0$ for all $Y\in\Gamma({%
\mathcal{D}}^{-})$. Let $X$ be a section of ${\mathcal{D}}^{+}$. Then, using %
\eqref{invariant1}, we have, for any $Y\in\Gamma({\mathcal{D}}_{+})$, \ $%
2\tilde g(\tilde h\Lambda_{{\mathcal{D}}^{-}}X,Y)=\Pi_{{\mathcal{D}}%
_{-}}(\Lambda_{{\mathcal{D}}^{-}}X,Y)=d\eta(X,Y)=\tilde g(X,\tilde\varphi
Y)=-\tilde g(X,Y)$. \ Consequently $2\tilde h\Lambda_{{\mathcal{D}}^{-}}X=-X$%
. Applying $\tilde h$ and using \eqref{H2} we easily get
\begin{equation}  \label{lambda1}
\Lambda_{{\mathcal{D}}^{-}}X=-\frac{1}{2(1+\tilde\kappa)}\tilde h X.
\end{equation}
Thus
\begin{align*}
\overline{\Pi}_{{\mathcal{D}}^{-}}(X,X^{\prime})&=\Pi_{{\mathcal{D}}%
^{-}}(\Lambda_{{\mathcal{D}}^{-}}X,\Lambda_{{\mathcal{D}}^{-}}X^{\prime})=%
\frac{1}{4(1+\tilde\kappa)^2}\Pi_{{\mathcal{D}}^{-}}(\tilde h X,\tilde h
X^{\prime})
\end{align*}
\begin{align*}
&=\frac{1}{2(1+\tilde\kappa)^2}\tilde g(\tilde h X,X^{\prime})=\frac{1}{%
4(1+\tilde\kappa)}\Pi_{{\mathcal{D}}^{+}}(X,X^{\prime}),
\end{align*}
so that
\begin{equation}  \label{relation1}
\Pi_{{\mathcal{D}}^{+}}(X,X^{\prime})=4(1+\tilde\kappa)\overline{\Pi}_{{%
\mathcal{D}}^{-}}(X,X^{\prime})
\end{equation}
for any $X,X^{\prime}\in\Gamma({\mathcal{D}}^{+})$. Arguing in a similar
manner one finds that
\begin{equation}  \label{lambda2}
\Lambda_{{\mathcal{D}}^{+}}Y=\frac{1}{2(1+\tilde\kappa)}\tilde h Y
\end{equation}
for all $Y\in\Gamma({\mathcal{D}}^{-})$ and
\begin{equation}  \label{relation2}
\Pi_{{\mathcal{D}}^{-}}(Y,Y^{\prime})=4(1+\tilde\kappa)\overline{\Pi}_{{%
\mathcal{D}}^{+}}(Y,Y^{\prime})
\end{equation}
for any $Y,Y^{\prime}\in\Gamma({\mathcal{D}}^{-})$. Comparing %
\eqref{relation1} with \eqref{relation2} we conclude that the bi-Legendrian
structure $({\mathcal{D}}^{+},{\mathcal{D}}^{-})$ satisfies the assumption
(I) or (III) of Theorem \ref{legendre1}, where $a$ and $b$ are any two real
numbers such that $ab=4(1+\tilde\kappa)$, both positive or negative
according to the fact that ${\mathcal{D}}^{+}$ and ${\mathcal{D}}^{-}$ are
positive or negative definite, respectively. By Theorem \ref{legendre1} this
proves the existence of a family of contact Riemannian structures $%
(\varphi_{a,b},\xi,\eta,g_{a,b})$ on $M$. The expressions \eqref{phi} and %
\eqref{metric} follow from \cite[(3.4)--(3.5)]{CAP1} and from \eqref{lambda1}%
, \eqref{lambda2} and \eqref{invariant1}. Concerning the last part of the
theorem we have to prove that the bi-Legendrian structure $({\mathcal{D}}%
^{+},{\mathcal{D}}^{-})$ satisfies the further assumption of Theorem \ref%
{legendre1}, i.e. that the corresponding bi-Legendrian connection $%
\nabla^{bl} $ preserves the Pang invariant of both the foliations. Notice
that by \cite[Theorem 3.6]{MON}, since ${\mathcal{D}}^{+}$, ${\mathcal{D}}%
^{-}$ are integrable, $\nabla^{bl}$ coincides in fact with the canonical
paracontact connection $\nabla^{pc}$. Now, by \eqref{PcTanaka} and %
\eqref{nablah}, for any $X,Y\in\Gamma(TM)$,
\begin{align*}
(\nabla^{pc}_{X}\tilde h)Y&=\tilde\nabla_{X}\tilde h Y +
\eta(X)\tilde\varphi\tilde h Y + \eta(\tilde h Y)(\tilde\varphi
X-\tilde\varphi\tilde h X) + \tilde g(X,\tilde\varphi\tilde h Y)\xi \\
&\quad - \tilde g(\tilde h X,\tilde\varphi\tilde h Y)\xi - \tilde
h\tilde\nabla_{X}Y - \eta(X)\tilde h\tilde\varphi Y - \eta(Y)(\tilde
h\tilde\varphi X - \tilde h\tilde\varphi\tilde h X) \\
&=(\tilde\nabla_{X}\tilde h)Y + 2\eta(X)\tilde\varphi h Y +
\eta(Y)\tilde\varphi\tilde h X - (1+\tilde\kappa)\eta(Y)\tilde\varphi X +
\tilde g(X,\tilde\varphi\tilde h Y)\xi - \tilde g(\tilde h
X,\tilde\varphi\tilde h Y)\xi \\
&=\tilde g(X,\tilde h \tilde\varphi Y)\xi-(1+\tilde\kappa)\tilde
g(X,\tilde\varphi Y)\xi+\eta(Y)\tilde h\tilde\varphi
X+(1+\tilde\kappa)\eta(Y)\tilde\varphi X + \tilde\mu\eta(X)\tilde
h\tilde\varphi Y \\
&\quad+2\eta(X)\tilde\varphi\tilde h Y + \eta(Y)\tilde\varphi\tilde h X -
(1+\tilde\kappa)\eta(Y)\tilde\varphi X + \tilde g(X,\tilde\varphi\tilde h
Y)\xi + (1+\tilde\kappa)\tilde g(X,\tilde\varphi Y)\xi \\
&=(\tilde\mu - 2)\eta(X)\tilde h\tilde\varphi Y.
\end{align*}
Consequently, if $\tilde\mu=2$ then $\nabla^{bl}\tilde h=\nabla^{pc}\tilde
h=0$. On the other hand, by (i) of Theorem \ref{zamconn} we have $%
\nabla^{bl}\tilde g=\nabla^{pc}\tilde g=0$. Thus, using the expression %
\eqref{invariant1} of $\Pi_{{\mathcal{D}}^{+}}$ in terms of the paracontact
structure, we get, for all $X,X^{\prime}\in\Gamma({\mathcal{D}}^{+})$ and
for all $Z\in\Gamma(TM)$
\begin{align*}
(\nabla^{bl}_{Z}\Pi_{{\mathcal{D}}^{+}})(X,X^{\prime})&=2Z(\tilde g(\tilde h
X,X^{\prime}))-2\tilde g(\tilde h\nabla^{pc}_{Z}X,X^{\prime})-2\tilde
g(\tilde h X,\nabla^{pc}_{Z}X^{\prime}) \\
&=2Z(\tilde g(\tilde h X,X^{\prime}))-2\tilde g(\nabla^{pc}_{Z}\tilde h
X,X^{\prime})-2\tilde g(\tilde h X,\nabla^{pc}_{Z}X^{\prime}) \\
&=2(\nabla^{pc}_{Z}\tilde g)(\tilde h X,X^{\prime})=0.
\end{align*}
\end{proof}

\begin{corollary}
Every positive or negative definite paracontact $(\tilde\kappa,\tilde\mu)$%
-manifold such that $\tilde\kappa>-1$ admits a $K$-contact structure.
\end{corollary}

\begin{proof}
It is sufficient to take $a=b=\pm 2(1+\tilde{\kappa})$ in Theorem \ref{main1}%
, since in the proof of Theorem \ref{legendre1} (cf. \cite{CAP1}) it is
shown that if $a=b$ then $h_{a,b}=0$ and so the contact metric structure is $%
K$-contact.
\end{proof}

Actually, we now prove that one can define a distinguished contact metric
structure on any positive or negative definite paracontact $%
(\tilde\kappa,\tilde\mu)$-space such that $\tilde\kappa>-1$.

\begin{theorem}
\label{principal1} Any positive or negative definite paracontact metric $%
(\tilde\kappa,\tilde\mu)$-manifold such that $\tilde\kappa>-1$ carries a
canonical contact Riemannian structure $(\phi,\xi,\eta,g)$ given by
\begin{equation}  \label{def1}
\phi:=\mp\frac{1}{\sqrt{1+\tilde\kappa}}\tilde\varphi\tilde{h}, \ \
g:=-d\eta(\cdot,\phi\cdot)+\eta\otimes\eta,
\end{equation}
where the sign $\mp$ depends on the positive or negative definiteness of the
paracontact $(\tilde\kappa,\tilde\mu)$-manifold. Moreover, if $\tilde\mu=2$
then $(\phi,\xi,\eta,g)$ is Sasakian, whereas if $\tilde\mu\neq 2$ then $%
(\phi,\xi,\eta,g)$ is a non-Sasakian contact metric $(\kappa,\mu)$%
-structure, where
\begin{equation}  \label{valori}
\kappa=1-\left(1-\frac{\tilde\mu}{2}\right)^2, \ \ \ \mu=2\left(1\mp\sqrt{%
1+\tilde\kappa}\right),
\end{equation}
the sign $\mp$ depending, respectively, on the positive or negative
definiteness of the paracontact metric structure $(\tilde\varphi,\xi,\eta,%
\tilde{g})$.
\end{theorem}

\begin{proof}
Let us define a $(1,1)$-tensor field $\phi $ and a tensor $g$ of type $(0,2)$
by \eqref{def1}. First of all, using \eqref{H2}, one can easily prove that $%
\phi ^{2}=-I+\eta \otimes \xi $. Next we prove that $g$ is a Riemannian
metric. By using the symmetry of the operator $\tilde{h}$
with respect to $\tilde{g}$, one has, for any $X,Y\in \Gamma (TM)$,
\begin{align}
g(X,Y)& =\pm \frac{1}{\sqrt{1+\tilde{\kappa}}}d\eta (X,\tilde{\varphi}\tilde{%
h}Y)+\eta (X)\eta (Y)  \notag  \label{metrica1} \\
& =\pm \frac{1}{\sqrt{1+\tilde{\kappa}}}\tilde{g}(X,\tilde{\varphi}^{2}%
\tilde{h}Y)+\eta (X)\eta (Y)  \notag \\
& =\pm \frac{1}{\sqrt{1+\tilde{\kappa}}}\tilde{g}(X,\tilde{h}Y)+\eta (X)\eta
(Y) \\
& =\pm \frac{1}{\sqrt{1+\tilde{\kappa}}}\tilde{g}(Y,\tilde{h}X)+\eta (X)\eta
(Y)  \notag \\
& =g(Y,X),  \notag
\end{align}%
so that $g$ is symmetric. In order to prove that it is also positive
definite, let us consider a $\tilde{\varphi}$-basis \ $\{X_{1},\ldots
,X_{n},Y_{1},\ldots ,Y_{n},\xi \}$ \ as \ in \ Lemma \ref{basis}. \ Then \
we \ have \ that \ $g(\xi ,\xi )=1$, \ $g(X_{i},X_{i})=\pm \frac{1}{\sqrt{1+%
\tilde{\kappa}}}\tilde{g}(X_{i},\tilde{h}X_{i})$ \ $=$ \ $\pm \frac{1}{\sqrt{%
1+\tilde{\kappa}}}\tilde{\lambda}\tilde{g}(X_{i},X_{i})$ \ $=$ \ $\pm \tilde{%
g}(X_{i},X_{i})=(\pm 1)(\pm 1)=1$ and $\tilde{g}(Y_{i},Y_{i})=\pm \frac{1}{%
\sqrt{1+\tilde{\kappa}}}\tilde{g}(Y_{i},\tilde{h}Y_{i})$ $=\mp \tilde{g}%
(Y_{i},Y_{i})=(\mp 1)(\mp 1)=1$. Finally one can straightforwardly check
that $g(\phi X,\phi Y)$ $=g(X,Y)-\eta (X)\eta (Y)$ and $g(X,\phi Y)=d\eta
(X,Y)$. Thus $(\phi ,\xi ,\eta ,g)$ is a contact Riemannian structure. We
prove the second part of the theorem. Let us compute the operator $h$
associated to the contact metric structure $(\phi ,\xi ,\eta ,g)$. We have
\begin{equation}
h=\frac{1}{2}{\mathcal{L}}_{\xi }\phi =\mp \frac{1}{2\sqrt{1+\tilde{\kappa}}}%
{\mathcal{L}}_{\xi }(\tilde{\varphi}\tilde{h})=\mp \frac{1}{2\sqrt{1+\tilde{%
\kappa}}}\left( ({\mathcal{L}}_{\xi }\tilde{\varphi})\tilde{h}+\tilde{\varphi%
}({\mathcal{L}}_{\xi }\tilde{h})\right) .  \label{passo1}
\end{equation}%
On the other hand, by using \eqref{NMBLA ZETAH}, we have, for any $X\in
\Gamma (TM)$,
\begin{align*}
({\mathcal{L}}_{\xi }\tilde{h})X& =[\xi ,\tilde{h}X]-\tilde{h}[\xi ,X] \\
& =\tilde{\nabla}_{\xi }{\tilde{h}}X-\tilde{\nabla}_{\tilde{h}X}\xi -\tilde{h%
}\tilde{\nabla}_{\xi }X+\tilde{h}\tilde{\nabla}_{X}\xi \\
& =(\tilde{\nabla}_{\xi }\tilde{h})X+\tilde{\varphi}\tilde{h}X-\tilde{\varphi%
}\tilde{h}^{2}X-\tilde{h}\tilde{\varphi}X+\tilde{h}\tilde{\varphi}\tilde{h}X
\\
& =(2-\tilde{\mu})\tilde{\varphi}\tilde{h}-2(1+\tilde{\kappa})\tilde{\varphi}%
X.
\end{align*}%
Thus \eqref{passo1} becomes
\begin{equation}
h=\mp \frac{1}{2\sqrt{1+\tilde{\kappa}}}\left( 2-\tilde{\mu}\right) \tilde{h}%
.  \label{passo2}
\end{equation}%
We distinguish the cases $\tilde{\mu}\neq 2$ and $\tilde{\mu}=2$. In the
first case by \eqref{passo2} we see that $h$ is diagonalizable, it admits
the eigenvalues $0$, $\pm \lambda $, where
\begin{equation}
\lambda :=1-\frac{\tilde{\mu}}{2},  \label{passo3}
\end{equation}%
and the same eigendistributions as $\tilde{h}$. We prove that the Legendre
foliations ${\mathcal{D}}_{h}(\lambda )$, ${\mathcal{D}}_{h}(-\lambda )$ and
the corresponding bi-Legendrian connection $\bar{\nabla}^{bl}$ satisfy the
conditions stated in Theorem \ref{characterization}, so concluding that $%
(\phi ,\xi ,\eta ,g)$ is a contact metric $(\kappa ,\mu )$-structure. First
of all, notice that ${\mathcal{D}}_{h}(\lambda )$ and ${\mathcal{D}}%
_{h}(-\lambda )$ are mutually $g$-orthogonal. Indeed by using %
\eqref{metrica1} one has, for any $X\in \Gamma ({\mathcal{D}}_{h}(\lambda ))$
and $Y\in \Gamma ({\mathcal{D}}_{h}(-\lambda ))$, \ $g(X,Y)=\pm \frac{1}{%
\sqrt{1+\tilde{\kappa}}}\tilde{g}(X,\tilde{h}Y)=\mp \tilde{g}(X,Y)=0$, \
since the eigendistributions of $\tilde{h}$ are $\tilde{g}$-orthogonal
(Corollary \ref{bipara2}). Next, by definition of bi-Legendrian connection,
also the conditions (i), (iii) and $\bar{\nabla}^{bl}\eta =\bar{\nabla}%
^{bl}d\eta =0$ of Theorem \ref{characterization} are satisfied. Moreover, $%
\bar{\nabla}^{bl}h=0$ because $\bar{\nabla}^{bl}$ preserves ${\mathcal{D}}%
_{h}(\lambda )$ and ${\mathcal{D}}_{h}(-\lambda )$. Thus it remains to prove
that $\bar{\nabla}^{bl}g=0$ and $\bar{\nabla}^{bl}\phi =0$. Let us recall (%
\cite{CAP5}) that, by definition, $\bar{\nabla}_{X}^{bl}Y=[X,Y]_{{\mathcal{D}%
}_{h}(-\lambda )}$ and $\bar{\nabla}_{Y}^{bl}X=[Y,X]_{{\mathcal{D}}%
_{h}(\lambda )}$ for any $X\in \Gamma ({\mathcal{D}}_{h}(\lambda ))$ and $%
Y\in \Gamma ({\mathcal{D}}_{h}(-\lambda ))$. Then for any $X,X^{\prime }\in
\Gamma ({\mathcal{D}}_{h}(\lambda ))$ and $Y,Y^{\prime }\in \Gamma ({%
\mathcal{D}}_{h}(-\lambda ))$ one has
\begin{align*}
(\bar{\nabla}_{Y}^{bl}\tilde{g})(X,X^{\prime })& =Y(\tilde{g}(X,X^{\prime
}))-\tilde{g}([Y,X]_{{\mathcal{D}}_{h}(\lambda )},X^{\prime })-\tilde{g}%
([Y,X^{\prime }]_{{\mathcal{D}}_{h}(-\lambda )},X) \\
& =Y(\tilde{g}(X,X^{\prime }))-\tilde{g}([Y,X],X^{\prime })-\tilde{g}%
([Y,X^{\prime }],X) \\
& =-2\tilde{g}(\tilde{\nabla}_{X}X^{\prime },Y)=0,
\end{align*}%
and, analogously, using the $\tilde{g}$-orthogonality and totally
geodesicity of ${\mathcal{D}}_{\tilde{h}}(\pm \tilde{\lambda})={\mathcal{D}}%
_{{h}}(\pm \lambda )$, one has that $(\bar{\nabla}_{X}^{bl}\tilde{g}%
)(Y,Y^{\prime })=-2\tilde{g}(\tilde{\nabla}_{Y}Y^{\prime },X)=0$, $(\bar{%
\nabla}_{\xi }^{bl}\tilde{g})(X,X^{\prime })=-2\tilde{g}(\tilde{\nabla}%
_{X}X^{\prime },\xi )=0$ and $(\bar{\nabla}_{\xi }^{bl}\tilde{g}%
)(Y,Y^{\prime })=-2\tilde{g}(\tilde{\nabla}_{Y}Y^{\prime },\xi )=0$.
Moreover, for any $X,X^{\prime },X^{\prime \prime }\in \Gamma ({\mathcal{D}}%
_{h}(\lambda ))$, by using $\bar{\nabla}^{bl}d\eta =0$,
\begin{align*}
(\bar{\nabla}_{X}^{bl}\tilde{g})(X^{\prime },X^{\prime \prime })& =X(\tilde{g%
}(X^{\prime },X^{\prime \prime }))-d\eta (\bar{\nabla}_{X}^{bl}X^{\prime },%
\tilde{\varphi}X^{\prime \prime })-d\eta (X^{\prime },\tilde{\varphi}\bar{%
\nabla}_{X}^{bl}X^{\prime \prime }) \\
& =X(\tilde{g}(X^{\prime },X^{\prime \prime }))-X(d\eta (X^{\prime },\tilde{%
\varphi}X^{\prime \prime }))+d\eta (X^{\prime },\bar{\nabla}_{X}^{bl}\tilde{%
\varphi}X^{\prime \prime }) \\
& \quad -X(d\eta (\tilde{\varphi}X^{\prime },X^{\prime \prime }))+d\eta (%
\bar{\nabla}_{X}^{bl}\tilde{\varphi}X^{\prime },X^{\prime \prime }) \\
& =X(\tilde{g}(X^{\prime },X^{\prime \prime }))-X(\tilde{g}(X^{\prime
},X^{\prime \prime }))+\tilde{g}(X^{\prime },\tilde{g}\bar{\nabla}_{X}^{bl}%
\tilde{\varphi}X^{\prime \prime }) \\
& \quad -X(\tilde{g}(\tilde{\varphi}X^{\prime },\tilde{\varphi}X^{\prime
\prime }))+\tilde{g}(\bar{\nabla}_{X}^{bl}\tilde{\varphi}X^{\prime },\tilde{%
\varphi}X^{\prime \prime }) \\
& =-X(\tilde{g}(\tilde{\varphi}X^{\prime },\tilde{\varphi}X^{\prime \prime
}))-\tilde{g}(\tilde{\varphi}X^{\prime },[X,\tilde{\varphi}X^{\prime \prime
}]_{{\mathcal{D}}_{h}(\lambda )})+\tilde{g}([X,\tilde{\varphi}X^{\prime }]_{{%
\mathcal{D}}_{h}(-\lambda )},\tilde{\varphi}X^{\prime \prime }) \\
& =-X(\tilde{g}(\tilde{\varphi}X^{\prime },\tilde{\varphi}X^{\prime \prime
}))-\tilde{g}(\tilde{\varphi}X^{\prime },[X,\tilde{\varphi}X^{\prime \prime
}])+\tilde{g}([X,\tilde{\varphi}X^{\prime }],\tilde{\varphi}X^{\prime \prime
}) \\
& =2\tilde{g}(\tilde{\nabla}_{\tilde{\varphi}X^{\prime }}\tilde{\varphi}%
X^{\prime \prime },X)=0
\end{align*}%
and, by similar computations, for any $Y,Y^{\prime },Y^{\prime \prime }\in
\Gamma ({\mathcal{D}}_{h}(-\lambda ))$, \ $(\bar{\nabla}_{Y}^{bl}\tilde{g}%
)(Y^{\prime },Y^{\prime \prime })=2\tilde{g}(\tilde{\nabla}_{\tilde{\varphi}%
Y^{\prime }}\tilde{\varphi}Y^{\prime \prime },Y)=0$, \ where we used again
the total geodesicity of ${\mathcal{D}}_{\tilde{h}}(\pm \tilde{\lambda})$.
Since, by definition, $\bar{\nabla}^{bl}\xi =0$, we conclude that $\bar{%
\nabla}^{bl}\tilde{g}=0$. Thus by \eqref{metrica1} and \eqref{passo2} we
have, for all $X,Y,Z\in \Gamma (TM)$,
\begin{equation*}
(\bar{\nabla}_{X}^{bl}g)(Y,Z)=\pm \frac{1}{\sqrt{1+\tilde{\kappa}}}(\bar{%
\nabla}_{X}^{bl}\tilde{g})(Y,\tilde{h}Z)\mp \frac{2}{2-\tilde{\mu}}\tilde{g}%
(Y,(\bar{\nabla}_{X}^{bl}h)Z)\pm \eta (Z)(\bar{\nabla}_{X}^{bl}\eta )(Y)\pm
\eta (Y)(\bar{\nabla}_{X}^{bl}\eta )(Z)=0
\end{equation*}%
since $\bar{\nabla}^{bl}\tilde{g}=0$, $\bar{\nabla}^{bl}h=0$ and $\bar{\nabla%
}^{bl}\eta =0$. On the other hand, from $\bar{\nabla}^{bl}g=0$, $\bar{\nabla}%
^{bl}d\eta =0$ and the relation $d\eta =g(\cdot ,\phi \cdot )$ it easily
follows that the bi-Legendrian connection $\bar{\nabla}^{bl}$ preserves also
the tensor field $\phi $. Therefore, according to Theorem \ref%
{characterization}, $(\phi ,\xi ,\eta ,g)$ is a contact metric $(\kappa ,\mu
)$-structure. In order to find the explicit expression of the constants $%
\kappa $ and $\mu $, notice that, by \eqref{passo3}, $\sqrt{1-\kappa }%
=\left\vert 1-\frac{\tilde{\mu}}{2}\right\vert $, from which it follows that
\begin{equation}
\kappa =1-\left( 1-\frac{\tilde{\mu}}{2}\right) ^{2}.  \label{passo5}
\end{equation}%
In \ order \ to \ find \ $\tilde{\mu}$, \ notice \ that \ since \ the \
bi-Legendrian \ structures \ $({\mathcal{D}}_{\tilde{h}}(-\tilde{\lambda}),{%
\mathcal{D}}_{\tilde{h}}(\tilde{\lambda}))$ and \ $({\mathcal{D}}%
_{h}(-\lambda ),{\mathcal{D}}_{h}(\lambda ))$ coincide, also the
corresponding Pang invariants must be equal. More precisely, by %
\eqref{passo2} one can find that
\begin{equation}
{\mathcal{D}}_{\tilde{h}}(\tilde{\lambda})=\left\{
\begin{array}{ll}
{\mathcal{D}}_{h}(\pm |\lambda |), & \hbox{ if $\tilde\mu>2$} \\
{\mathcal{D}}_{h}(\mp |\lambda |), & \hbox{ if $\tilde\mu<2$}%
\end{array}%
\right.  \label{relazione1}
\end{equation}%
\begin{equation}
{\mathcal{D}}_{\tilde{h}}(-\tilde{\lambda})=\left\{
\begin{array}{ll}
{\mathcal{D}}_{h}(\mp |\lambda |), & \hbox{ if $\tilde\mu>2$} \\
{\mathcal{D}}_{h}(\pm |\lambda |), & \hbox{ if $\tilde\mu<2$}%
\end{array}%
\right.  \label{relazione2}
\end{equation}%
where the sign $\pm $ depends on the positive or negative definiteness of
the paracontact $(\tilde{\kappa},\tilde{\mu})$-manifold $(M,\tilde{\varphi}%
,\xi ,\eta ,\tilde{g})$. Let us assume that $(M,\tilde{\varphi},\xi ,\eta ,%
\tilde{g})$ is positive definite and that $\tilde{\mu}>2$. Then, by using %
\eqref{relazione1}--\eqref{relazione2} and by comparing \cite[(11)]{CAP1}
with \eqref{pang1} we get
\begin{equation}
2\left( 1-\frac{\mu }{2}+|\lambda |\right) g(X,X^{\prime })=-2\left( 1-\frac{%
\tilde{\mu}}{2}-\sqrt{1+\tilde{\kappa}}\right) \tilde{g}(X,X^{\prime })
\label{passo4}
\end{equation}%
for any $X,X^{\prime }\in \Gamma ({\mathcal{D}}_{\tilde{h}}(\tilde{\lambda}%
)) $. By \eqref{metrica1} and \eqref{passo5}, \eqref{passo4} becomes
\begin{equation*}
2\left( 1-\frac{\mu }{2}-\left( 1-\frac{\tilde{\mu}}{2}\right) \right)
\tilde{g}(X,X^{\prime })=-2\left( 1-\frac{\tilde{\mu}}{2}-\sqrt{1+\tilde{%
\kappa}}\right) \tilde{g}(X,X^{\prime }),
\end{equation*}%
It follows that
\begin{equation}
\mu =2\left( 1-\sqrt{1+\tilde{\kappa}}\right) .  \label{value1}
\end{equation}%
If we assume $\tilde{\mu}<2$ we use \cite[(12)]{CAP1} and we get
\begin{equation*}
2\left( 1-\frac{\mu }{2}-|\lambda |\right) g(X,X^{\prime })=-2\left( 1-\frac{%
\tilde{\mu}}{2}-\sqrt{1+\tilde{\kappa}}\right) \tilde{g}(X,X^{\prime })
\end{equation*}%
and, as $|\lambda |=1-\frac{\tilde{\mu}}{2}$, so we obtain again %
\eqref{value1}. The case when $(M,\tilde{\varphi},\xi ,\eta ,\tilde{g})$ is
negative definite is similar and one can prove that
\begin{equation}
\mu =2\left( 1+\sqrt{1+\tilde{\kappa}}\right) .  \label{value2}
\end{equation}%
Now let us assume that $\tilde{\mu}=2$. Then \eqref{passo2} implies that the
operator $h$ vanishes, so that the contact metric structure $(\phi ,\xi
,\eta ,g)$ is $K$-contact. In particular one has $N_{\phi }(\xi ,X)=\phi
^{2}[\xi ,X]-\phi \lbrack \xi ,\phi X]=-2\phi {h}X=0$ for all $X\in \Gamma
(TM)$. Moreover, since ${\mathcal{D}}^{+}$, ${\mathcal{D}}^{-}$, ${\mathcal{D%
}}_{\tilde{h}}(\tilde{\lambda})$, ${\mathcal{D}}_{\tilde{h}}(-\tilde{\lambda}%
)$ are Legendre foliations, the canonical almost bi-paracontact structure %
\eqref{caseI} is integrable. Thus by \cite[Corollary 3.9]{CAP2} we deduce
that $N_{\phi }(X,Y)=0$ for all $X,Y\in \Gamma ({\mathcal{D}})$.
Consequently the tensor field $N_{\phi }$ vanishes identically and $(M,\phi
,\xi ,\eta ,g)$ is a Sasakian manifold.
\end{proof}

\begin{example}
Let us consider the paracontact $(\tilde{\kappa},\tilde{\mu})$-manifold $(G,%
\tilde{\varphi},\xi ,\eta ,\tilde{g})$ described in Example \ref{example1}
and let us apply the procedure of Theorem \ref{principal1}. Then a canonical
contact $(\kappa ,\mu )$-structure $(\phi ,\xi ,\eta ,g)$ is defined on $G$,
where, according to \eqref{valori}, $\kappa =1-\frac{(\alpha ^{2}-\beta
^{2})^{2}}{16}$ and $\mu =2\left( 1+\frac{\alpha ^{2}+\beta ^{2}}{4}\right) $%
. Explicitly, the contact Riemannian structure is defined as follows
\begin{gather*}
\phi e_{1}=e_{3},\ \ \phi e_{2}=e_{4},\ \ \phi e_{3}=-e_{1},\ \ \phi
e_{4}=-e_{2},\ \ \phi e_{5}=0, \\
g(e_{i},e_{j})=\delta _{ij}\ \hbox{ for any }i,j\in \left\{ 1,\ldots
,5\right\} .
\end{gather*}%
In order to understand where such a contact metric $(\kappa ,\mu )$%
-structure on the Lie group $G$ stays in the Boeckx's classification, let us
compute the value of the Boeckx invariant $I_{G}$ (\cite{BO}). An easy
computation shows that $I_{G}=\frac{1-\frac{\mu }{2}}{\sqrt{1-\kappa }}=-%
\frac{\alpha ^{2}+\beta ^{2}}{|\alpha ^{2}-\beta ^{2}|}$. Then one can
straightforwardly check that $I_{G}<-1$ provided that $\alpha ,\beta \neq 0$%
, and $I_{G}=-1$ if $\alpha =0$ ($\beta \neq 0$) or $\alpha =0$ ($\beta \neq
0$). Hence the contact Riemannian manifold $(G,\phi ,\xi ,\eta ,g)$ is
locally isometric to one among the contact Riemannian Lie groups described
in $\S 4$ of \cite{BO}, namely that one  whose Lie algebra has the same constant
structures as \eqref{constant1}--\eqref{constant2}.
\end{example}

\begin{remark}
\label{remarkreferee} Let $(M,\phi ,\xi ,\eta ,g)$ be a (non-Sasakian)
contact metric $(\kappa ^{\prime },\mu ^{\prime })$-space. Then by applying
the procedure described in \cite{MOTE} one obtains a paracontact $(\tilde{%
\kappa},\tilde{\mu})$-structure $(\tilde{\varphi},\xi ,\eta ,\tilde{g})$ on $%
M$, being
\begin{equation*}
\tilde{\kappa}=\kappa -2+\left( 1-\frac{\mu }{2}\right) ^{2},\ \ \ \tilde{\mu%
}=2.
\end{equation*}%
Since the bi-Legendrian structure $(\mathcal{D}^{+},\mathcal{D}^{-})$
coincides with that one $(\mathcal{D}_{h}(\lambda ),\mathcal{D}_{h}(-\lambda
))$ defined by the eigendistribution of $h$, due to \cite[Theorem 4]{CAP1} one
has that the paracontact metric structure $(\tilde{\varphi},\xi ,\eta ,%
\tilde{g})$ is positive or negative definite if and only if $I_{M}^{2}>1$,
where $I_{M}$ denotes the of the contact metric $(\kappa ,\mu )$-structure $%
(\phi ,\xi ,\eta ,g)$. But $I_{M}^{2}>1$ if and only if $\frac{\left( 1-%
\frac{\mu }{2}\right) ^{2}}{1-\kappa }>1$, that is $\kappa -1+\left( 1-\frac{%
\mu }{2}\right) ^{2}>0$, which is equivalent to require that $\tilde{\kappa}%
>-1$. Therefore the only positive or negative definite paracontact $(\tilde{%
\kappa},\tilde{\mu})$-structures determined via the above procedure are
those ones with $\tilde{\kappa}>-1$. Then we are under the assumption of
Theorem \ref{principal1} and so we obtain a new contact Riemannian structure
$(\phi ^{\prime },\eta ,\xi ,g^{\prime })$. Since $\tilde{\mu}=2$, $(\phi
^{\prime },\eta ,\xi ,g^{\prime })$ is in fact a Sasakian structure. Hence
we have proved that any non-Sasakian contact metric $(\kappa ,\mu )$%
-manifold such that $|I_{M}|>1$ admits a Sasakian metric compatible with the
same underlying contact form $\eta $. The same result was proved using
different techniques in \cite{CAP1}.
\end{remark}

Now we pass to study some important curvature properties.

\begin{lemma}
Let $(M,\tilde{\varphi},\xi ,\eta ,\tilde{g})$ be a paracontact $(\tilde{%
\kappa},\tilde{\mu})$-manifold such that $\tilde{\kappa}>-1$. Then for any
vector fields $X$, $Y$, $Z$ on $M$ we have
\begin{align}  \label{RHZ}
\tilde{R}_{XY}\tilde{h}Z-\tilde{h}\tilde{R}_{XY}Z &=\bigl(\tilde{\kappa}%
(\eta (X)\tilde{g}(\tilde{h}Y,Z)-\eta (Y)\tilde{g}(\tilde{h}X,Z)) +\tilde{\mu%
}(1+\tilde{\kappa})(\eta (X)\tilde{g}(Y,Z)-\eta (Y)\tilde{g}(X,Z))\bigr)\xi
\notag \\
&\quad+\tilde{\kappa}\bigl(\tilde{g}(X,\tilde{\varphi}Z)\tilde{\varphi}%
\tilde{h}Y-\tilde{g}(Y,\tilde{\varphi}Z)\tilde{\varphi}\tilde{h}X +\tilde{g}%
(Z,\tilde{\varphi}\tilde{h}X)\tilde{\varphi}Y-\tilde{g}(Z,\tilde{\varphi}%
\tilde{h}Y)\tilde{\varphi}X \\
&\quad+\eta (Z)(\eta (X)\tilde{h}Y-\eta (Y)\tilde{h}X)\bigr) -\tilde{\mu}%
\bigl((1+\tilde{\kappa})\eta (Z)(\eta (Y)X-\eta (X)Y) +2\tilde{g}(X,\tilde{%
\varphi}Y)\tilde{\varphi}\tilde{h}Z\bigr).  \notag
\end{align}
\end{lemma}

\begin{proof}
The Ricci identity for $\tilde{h}$ is
\begin{equation}
\tilde{R}_{XY}\tilde{h}Z-\tilde{h}\tilde{R}_{XY}Z=(\tilde{\nabla}_{X}\tilde{%
\nabla}_{Y}\tilde{h})Z-(\tilde{\nabla}_{Y}\tilde{\nabla}_{X}\tilde{h})Z-(%
\tilde{\nabla}_{\left[ X,Y\right] }\tilde{h})Z  \label{RHZ2}
\end{equation}%
Using (\ref{H2}), (\ref{nablah}) and the facts that $\tilde{h}$
anti-commutes with $\tilde{\varphi}$ and $\tilde{\nabla}_{X}\tilde{\varphi}$
is antisymmetric, we get by direct calculation
\begin{align*}
(\tilde{\nabla}_{X}\tilde{\nabla}_{Y}\tilde{h})Z& =-\left( (1+\tilde{\kappa})%
\tilde{g}(\tilde{\nabla}_{X}Y,\tilde{\varphi}Z)+(1+\tilde{\kappa})\tilde{g}%
(Y,\tilde{\nabla}_{X}\tilde{\varphi}Z)+\tilde{g}(\tilde{\nabla}_{X}Y,\tilde{%
\varphi}\tilde{h}Z)+\tilde{g}(Y,\tilde{\nabla}_{X}\tilde{\varphi}\tilde{h}%
Z)\right) \xi \\
& \quad -\left( (1+\tilde{\kappa})\tilde{g}(Y,\tilde{\varphi}Z)+\tilde{g}(Y,%
\tilde{\varphi}\tilde{h}Z)\right) \tilde{\nabla}_{X}\xi +(\eta (\tilde{\nabla%
}_{X}Z)+\tilde{g}(Z,\tilde{\nabla}_{X}\xi ))((1+\tilde{\kappa})\tilde{\varphi%
}Y-\tilde{\varphi}\tilde{h}Y) \\
& \quad +\eta (Z)((\tilde{\kappa}+1)\tilde{\nabla}_{X}\tilde{\varphi}Y-%
\tilde{\nabla}_{X}\tilde{\varphi}\tilde{h}Y) \\
& \quad -\tilde{\mu}(\eta (\tilde{\nabla}_{X}Y)+\tilde{g}(Y,\tilde{\nabla}%
_{X}\xi ))\tilde{\varphi}\tilde{h}Z-\tilde{\mu}\eta (Y)\tilde{\nabla}_{X}%
\tilde{\varphi}\tilde{h}Z.
\end{align*}%
So, using also \eqref{nablah} and \eqref{NAMLA X H}, equation \eqref{RHZ2}
yields
\begin{align}
\tilde{R}_{XY}\tilde{h}Z-\tilde{h}\tilde{R}_{XY}Z& =\bigl((1+\tilde{\kappa})%
\tilde{g}((\tilde{\nabla}_{X}\tilde{\varphi})Y-(\tilde{\nabla}_{Y}\tilde{%
\varphi})X,Z)+\tilde{g}((\tilde{\nabla}_{X}\tilde{h}\tilde{\varphi})Y-(%
\tilde{\nabla}_{Y}\tilde{h}\tilde{\varphi})X,Z)\bigr)\xi  \notag
\label{RHZ3} \\
& \quad -\bigl((1+\tilde{\kappa})\tilde{g}(Y,\tilde{\varphi}Z)-\tilde{g}(Y,%
\tilde{h}\tilde{\varphi}Z)\bigr)\tilde{\nabla}_{X}\xi +\bigl((1+\tilde{\kappa%
})\tilde{g}(X,\tilde{\varphi}Z)-\tilde{g}(X,\tilde{h}\tilde{\varphi}Z)\bigr)%
\tilde{\nabla}_{Y}\xi  \notag \\
& \quad +\tilde{g}(Z,\tilde{\nabla}_{X}\xi )(\tilde{h}\tilde{\varphi}Y+(1+%
\tilde{\kappa})\tilde{\varphi}Y)-\tilde{g}(Z,\tilde{\nabla}_{Y}\xi )(\tilde{h%
}\tilde{\varphi}X+(\tilde{\kappa}+1)\tilde{\varphi}X) \\
& \quad +\eta (Z)\bigl((\tilde{\nabla}_{X}\tilde{h}\tilde{\varphi})Y-(\tilde{%
\nabla}_{Y}\tilde{h}\tilde{\varphi})X+(1+\tilde{\kappa})((\tilde{\nabla}_{X}%
\tilde{\varphi})Y-(\tilde{\nabla}_{Y}\tilde{\varphi})X)\bigr)  \notag \\
& \quad -\tilde{\mu}\bigl(\eta (Y)(\tilde{\nabla}_{X}\tilde{\varphi}\tilde{h}%
)Z-\eta (X)(\tilde{\nabla}_{Y}\tilde{\varphi}\tilde{h})Z+2\tilde{g}(X,\tilde{%
\varphi}Y)\tilde{\varphi}\tilde{h}Z\bigr).  \notag
\end{align}%
Using now \eqref{NAMLAFI}, \eqref{nablah} and $\tilde{h}\xi =0$, we get
\begin{equation*}
(\tilde{\nabla}_{X}\tilde{\varphi}\tilde{h})Y=\tilde{g}(\tilde{h}^{2}X-%
\tilde{h}X,Y)\xi +\eta (Y)(\tilde{h}^{2}X-\tilde{h}X)-\tilde{\mu}\eta (X)%
\tilde{h}Y.
\end{equation*}%
Therefore, by using \eqref{NAMLAFI} again, \eqref{RHZ3} reduces to %
\eqref{RHZ}.
\end{proof}

\begin{theorem}
\label{curv} Let $(M,\tilde{\varphi},\xi ,\eta ,\tilde{g})$ be a paracontact
$(\tilde{\kappa},\tilde{\mu})$-manifold such that $\tilde{\kappa}>-1$. Then
we have, for any $X,X^{\prime },X^{\prime \prime }\in {\mathcal{D}}_{\tilde{h%
}}(\tilde{\lambda})$ and $Y,Y^{\prime },Y^{\prime \prime }\in $ ${\mathcal{D}%
}_{\tilde{h}}(-\tilde{\lambda})$,
\begin{align}
\tilde{R}_{XX^{\prime }}X^{\prime \prime }& =(2(\tilde{\lambda}-1)+\tilde{\mu%
})(\tilde{g}(X^{\prime },X^{\prime \prime })X-\tilde{g}(X,X^{\prime \prime
})X^{\prime })  \label{v} \\
\tilde{R}_{XX^{\prime }}Y& =(\tilde{\kappa}+\tilde{\mu})(-\tilde{g}(\tilde{%
\varphi}X^{\prime },Y)\tilde{\varphi}X+\tilde{g}(\tilde{\varphi}X,Y)\tilde{%
\varphi}X^{\prime })  \label{i} \\
\tilde{R}_{XY}X^{\prime }& =\tilde{\kappa}\tilde{g}(\tilde{\varphi}%
Y,X^{\prime })\tilde{\varphi}X-\tilde{\mu}\tilde{g}(\tilde{\varphi}Y,X)%
\tilde{\varphi}X^{\prime } \\
\tilde{R}_{XY}Y^{\prime }& =-\tilde{\kappa}\tilde{g}(\tilde{\varphi}%
X,Y^{\prime })\tilde{\varphi}Y+\tilde{\mu}\tilde{g}(\tilde{\varphi}X,Y)%
\tilde{\varphi}Y^{\prime }  \label{iii} \\
\tilde{R}_{YY^{\prime }}X& =(\tilde{\kappa}+\tilde{\mu})(-\tilde{g}(\tilde{%
\varphi}Y^{\prime },X)\tilde{\varphi}Y+\tilde{g}(\tilde{\varphi}Y,X)\tilde{%
\varphi}Y^{\prime }) \\
\tilde{R}_{YY^{\prime }}Y^{\prime \prime }& =(-2(\tilde{\lambda}+1)+\tilde{%
\mu})(\tilde{g}(Y^{\prime },Y^{\prime \prime })Y-\tilde{g}(Y,Y^{\prime
\prime })Y^{\prime }).  \label{vi}
\end{align}
\end{theorem}

\begin{proof}
We start by proving \eqref{i}. We can choose a local orthogonal $\tilde{%
\varphi}$-basis $\{e_{i},\tilde{\varphi}e_{i},\xi \}$, $i\in \{1,\ldots ,n\}$%
, as in Lemma \ref{basis}. Then we have
\begin{align}
\tilde{R}_{XX^{\prime }}Y& =\tilde{g}(\tilde{R}_{XX^{\prime }}Y,\xi )\xi
-\sum_{i=1}^{r}\tilde{g}(\tilde{R}_{XX^{\prime
}}Y,e_{i})e_{i}+\sum_{i=r+1}^{n}\tilde{g}(\tilde{R}_{XX^{\prime
}}Y,e_{i})e_{i}  \notag \\
& \quad +\sum_{i=1}^{r}\tilde{g}(\tilde{R}_{XX^{\prime }}Y,\tilde{\varphi}%
e_{i})\tilde{\varphi}e_{i}-\sum_{i=r+1}^{n}\tilde{g}(\tilde{R}_{XX^{\prime
}}Y,\tilde{\varphi}e_{i})\tilde{\varphi}e_{i}.  \label{RXYZ}
\end{align}%
Notice that, because of \eqref{PARAKMU}, $\tilde{g}(\tilde{R}_{XX^{\prime
}}Y,\xi )=-\tilde{g}(\tilde{R}_{XX^{\prime }}\xi ,Y)=0$. Moreover, due to
Theorem \ref{para1}, also the terms $\tilde{g}(\tilde{R}_{XX^{\prime
}}Y,e_{i})$ in \eqref{RXYZ} vanish. \ On the other hand, if $X\in \Gamma ({%
\mathcal{D}}_{\tilde{h}}(\tilde{\lambda}))$ and $Y,Z\in \Gamma ({\mathcal{D}}%
_{\tilde{h}}(-\tilde{\lambda}))$, then applying \eqref{RHZ} we get
\begin{equation*}
\tilde{R}_{XY}\tilde{h}Z-\tilde{h}\tilde{R}_{XY}Z=-(\tilde{\lambda}\tilde{R}%
_{XY}Z+\tilde{h}\tilde{R}_{XY}Z)=-2\tilde{\lambda}(\tilde{\kappa}\tilde{g}(X,%
\tilde{\varphi}Z)\tilde{\varphi}Y-\tilde{\mu}\tilde{g}(X,\tilde{\varphi}Y)%
\tilde{\varphi}Z)
\end{equation*}%
and, taking the inner product with $W\in \Gamma ({\mathcal{D}}_{\tilde{h}}(%
\tilde{\lambda}))$, we get
\begin{equation}
\tilde{g}(\tilde{R}_{XY}Z,W)=\tilde{\kappa}\tilde{g}(X,\tilde{\varphi}Z)%
\tilde{g}(\tilde{\varphi}Y,W)-\tilde{\mu}\tilde{g}(X,\tilde{\varphi}Y)\tilde{%
g}(\tilde{\varphi}Z,W)  \label{1equ}
\end{equation}%
for any $X,W\in \Gamma ({\mathcal{D}}_{\tilde{h}}(\tilde{\lambda}))$ and $%
Y,Z\in \Gamma (\mathcal{D}_{\tilde{h}}(-\tilde{\lambda}))$. Therefore, using %
\eqref{RXYZ}, \eqref{1equ} and the first Bianchi identity we find
\begin{align*}
\tilde{R}_{XX^{\prime }}Y& =-\sum\limits_{i=1}^{r}\left( \tilde{g}(\tilde{R}%
_{YX}X^{\prime },\tilde{\varphi}e_{i})\tilde{\varphi}e_{i}+\tilde{g}(\tilde{R%
}_{X^{\prime }Y}X,\tilde{\varphi}e_{i})\tilde{\varphi}e_{i}\right) \\
& \quad +\sum\limits_{i=r+1}^{n}\left( \tilde{g}(\tilde{R}_{YX}X^{\prime },%
\tilde{\varphi}e_{i})\tilde{\varphi}e_{i}+\tilde{g}(\tilde{R}_{X^{\prime
}Y}X,\tilde{\varphi}e_{i})\tilde{\varphi}e_{i}\right) \\
& =-\sum\limits_{i=1}^{r}\left( \tilde{g}(\tilde{R}_{XY}\tilde{\varphi}%
e_{i},X^{\prime })\tilde{\varphi}e_{i}-\tilde{g}(\tilde{R}_{X^{\prime }Y}%
\tilde{\varphi}e_{i},X)\tilde{\varphi}e_{i}\right) \\
& \quad +\sum\limits_{i=r+1}^{n}\left( \tilde{g}(\tilde{R}_{XY}\tilde{\varphi%
}e_{i},X^{\prime })\tilde{\varphi}e_{i}-\tilde{g}(\tilde{R}_{X^{\prime }Y}%
\tilde{\varphi}e_{i},X)\tilde{\varphi}e_{i}\right) \\
& =-\sum\limits_{i=1}^{r}\bigl(\tilde{\kappa}\tilde{g}(X,\tilde{\varphi}%
^{2}e_{i})\tilde{g}(\tilde{\varphi}Y,X^{\prime })\tilde{\varphi}e_{i}-\tilde{%
\mu}\tilde{g}(X,\tilde{\varphi}Y)\tilde{g}(\tilde{\varphi}%
^{2}e_{i},X^{\prime })\tilde{\varphi}e_{i} \\
& \quad -\tilde{\kappa}\tilde{g}(X^{\prime },\tilde{\varphi}^{2}e_{i})\tilde{%
g}(\tilde{\varphi}Y,X)\tilde{\varphi}e_{i}+\tilde{\mu}\tilde{g}(X^{\prime },%
\tilde{\varphi}Y)\tilde{g}(\tilde{\varphi}^{2}e_{i},X)\tilde{\varphi}e_{i}%
\bigr) \\
& \quad +\sum\limits_{i=r+1}^{n}\bigl(\tilde{\kappa}\tilde{g}(X,\tilde{%
\varphi}^{2}e_{i})\tilde{g}(\tilde{\varphi}Y,X^{\prime })\tilde{\varphi}%
e_{i}-\tilde{\mu}\tilde{g}(X,\tilde{\varphi}Y)\tilde{g}(\tilde{\varphi}%
^{2}e_{i},X^{\prime })\tilde{\varphi}e_{i} \\
& \quad -\tilde{\kappa}\tilde{g}(X^{\prime },\tilde{\varphi}^{2}e_{i})\tilde{%
g}(\tilde{\varphi}Y,X)\tilde{\varphi}e_{i}+\tilde{\mu}\tilde{g}(X^{\prime },%
\tilde{\varphi}Y)\tilde{g}(\tilde{\varphi}^{2}e_{i},X)\tilde{\varphi}e_{i}%
\bigr)\\
& =\tilde{\kappa}\tilde{g}(\tilde{\varphi}Y,X^{\prime })\tilde{\varphi}X-%
\tilde{\mu}\tilde{g}(X,\tilde{\varphi}Y)\tilde{\varphi}X^{\prime }-\tilde{%
\kappa}\tilde{g}(\tilde{\varphi}Y,X)\tilde{\varphi}X^{\prime }+\tilde{\mu}%
\tilde{g}(\tilde{\varphi}Y,X^{\prime })\tilde{\varphi}X \\
& =(\tilde{\kappa}+\tilde{\mu})(-\tilde{g}(\tilde{\varphi}X^{\prime },Y)%
\tilde{\varphi}X+\tilde{g}(\tilde{\varphi}X,Y)\tilde{\varphi}X^{\prime }).
\end{align*}%
Thus \eqref{i} is proved. Now let us prove \eqref{iii}. We have
\begin{align}
\tilde{R}_{XY}Y^{\prime }& =\tilde{g}(\tilde{R}_{XY}Y^{\prime },\xi )\xi
-\sum_{i=1}^{r}\tilde{g}(\tilde{R}_{XY}Y^{\prime
},e_{i})e_{i}+\sum_{i=r+1}^{n}\tilde{g}(\tilde{R}_{XY}Y^{\prime },e_{i})e_{i}
\notag  \label{RXYZ2} \\
& \quad +\sum_{i=1}^{r}\tilde{g}(\tilde{R}_{XY}Y^{\prime },\tilde{\varphi}%
e_{i})\tilde{\varphi}e_{i}-\sum_{i=r+1}^{n}\tilde{g}(\tilde{R}_{XY}Y^{\prime
},\tilde{\varphi}e_{i})\tilde{\varphi}e_{i}.
\end{align}%
Arguing as before we have that $\tilde{g}(\tilde{R}_{XY}Y^{\prime },\xi )=%
\tilde{g}(\tilde{R}_{XY}Y^{\prime },e_{i})=0$ for each $i\in \{1,\ldots ,n\}$%
. On the other hand, if $X\in \Gamma ({\mathcal{D}}_{\tilde{h}}(\tilde{%
\lambda}))$ and $Y,Z\in \Gamma ({\mathcal{D}}_{\tilde{h}}(-\tilde{\lambda}))$%
, then applying \eqref{RHZ} we get
\begin{equation*}
\tilde{R}_{XY}\tilde{h}Z-\tilde{h}\tilde{R}_{XY}Z=-(\tilde{\lambda}\tilde{R}%
_{XY}Z+\tilde{h}\tilde{R}_{XY}Z)=-2\tilde{\lambda}(\tilde{\kappa}\tilde{g}(X,%
\tilde{\varphi}Z)\tilde{\varphi}Y-\tilde{\mu}\tilde{g}(X,\tilde{\varphi}Y)%
\tilde{\varphi}Z)
\end{equation*}%
and, taking the inner product with $W\in \Gamma ({\mathcal{D}}_{\tilde{h}}(%
\tilde{\lambda}))$, we have%
\begin{equation}
\tilde{g}(\tilde{R}_{XY}Z,W)=\tilde{\kappa}\tilde{g}(X,\tilde{\varphi}Z)%
\tilde{g}(\tilde{\varphi}Y,W)-\tilde{\mu}\tilde{g}(X,\tilde{\varphi}Y)\tilde{%
g}(\tilde{\varphi}Z,W)  \label{2equ}
\end{equation}%
for any $X,W\in \Gamma ({\mathcal{D}}_{\tilde{h}}(\tilde{\lambda}))$ and $%
Y,Z\in \Gamma ({\mathcal{D}}_{\tilde{h}}(-\tilde{\lambda}))$. Using %
\eqref{RXYZ2}, \eqref{2equ} and the first Bianchi identity we get
\begin{align*}
\tilde{R}_{XY}Y^{\prime }& =\sum\limits_{i=1}^{r}\left( \tilde{g}(\tilde{R}%
_{Y^{\prime }X}Y,e_{i})e_{i}+\tilde{g}(\tilde{R}_{YY^{\prime
}}X,e_{i})e_{i}\right) -\sum\limits_{i=r+1}^{n}\left( \tilde{g}(\tilde{R}%
_{Y^{\prime }X}Y,e_{i})e_{i}+\tilde{g}(\tilde{R}_{YY^{\prime
}}X,e_{i})e_{i}\right) \\
& =\sum\limits_{i=1}^{r}\left( \tilde{g}(\tilde{R}_{Xe_{i}}Y,Y^{\prime
})e_{i}-\tilde{g}(\tilde{R}_{XY^{\prime }}Y,e_{i})e_{i}\right)
-\sum\limits_{i=r+1}^{n}\left( \tilde{g}(\tilde{R}_{Xe_{i}}Y,Y^{\prime
})e_{i}-\tilde{g}(\tilde{R}_{XY^{\prime }}Y,e_{i})e_{i}\right) \\
& =\sum\limits_{i=1}^{r}\left( -\tilde{\kappa}\tilde{g}(X,\tilde{\varphi}Y)%
\tilde{g}(\tilde{\varphi}Y^{\prime },e_{i})+\tilde{\mu}\tilde{g}(X,\tilde{%
\varphi}Y^{\prime })\tilde{g}(\tilde{\varphi}Y,e_{i})\right) e_{i} \\
& \quad +\sum\limits_{i=r+1}^{n}\left( \tilde{\kappa}\tilde{g}(X,\tilde{%
\varphi}Y)\tilde{g}(\tilde{\varphi}Y^{\prime },e_{i})-\tilde{\mu}\tilde{g}(X,%
\tilde{\varphi}Y^{\prime })\tilde{g}(\tilde{\varphi}Y^{\prime
},e_{i})\right) e_{i} \\
& \quad +\sum\limits_{i=1}^{r}(\tilde{\kappa}+\tilde{\mu})\left( \tilde{g}(%
\tilde{\varphi}X,Y)\tilde{g}(\tilde{\varphi}e_{i},Y^{\prime })-\tilde{g}(%
\tilde{\varphi}e_{i},Y)\tilde{g}(\tilde{\varphi}X,Y^{\prime })\right) e_{i}
\\
& \quad -\sum\limits_{i=r+1}^{n}(\tilde{\kappa}+\tilde{\mu})\left( \tilde{g}(%
\tilde{\varphi}X,Y)\tilde{g}(\tilde{\varphi}e_{i},Y^{\prime })-\tilde{g}(%
\tilde{\varphi}e_{i},Y)\tilde{g}(\tilde{\varphi}X,Y^{\prime })\right) e_{i}
\\
& =\tilde{\kappa}\tilde{g}(X,\tilde{\varphi}Y)\tilde{\varphi}Y^{\prime }-%
\tilde{\mu}\tilde{g}(X,\tilde{\varphi}Y^{\prime })\tilde{\varphi}Y+(\tilde{%
\kappa}+\tilde{\mu})\left( \tilde{g}(\tilde{\varphi}X,Y)\tilde{\varphi}%
Y^{\prime }-\tilde{g}(\tilde{\varphi}X,Y^{\prime })\tilde{\varphi}Y\right) \\
& =-\tilde{\kappa}\tilde{g}(\tilde{\varphi}X,Y^{\prime })\tilde{\varphi}Y+%
\tilde{\mu}\tilde{g}(\tilde{\varphi}X,Y)\tilde{\varphi}Y^{\prime }.
\end{align*}%
Finally, we show \eqref{v}. By using \eqref{CURVATURE 11} one obtains
\begin{equation}
\tilde{R}_{XX^{\prime }}\tilde{\varphi}X^{\prime \prime }-\tilde{\varphi}%
\tilde{R}_{XX^{\prime }}X^{\prime \prime }=\tilde{g}(X^{\prime }-\tilde{h}%
X^{\prime },X^{\prime \prime })(\tilde{\varphi}X-\tilde{\varphi}\tilde{h}X)-%
\tilde{g}(X-\tilde{h}X,X^{\prime \prime })(\tilde{\varphi}X^{\prime }-\tilde{%
\varphi}\tilde{h}X^{\prime }).  \label{RXYFZ}
\end{equation}%
Then by applying $\tilde{\varphi}$ to \eqref{RXYFZ} we get
\begin{align*}
\tilde{\varphi}\tilde{R}_{XX^{\prime }}\tilde{\varphi}X^{\prime \prime }-%
\tilde{R}_{XX^{\prime }}X^{\prime \prime }& =\tilde{g}(X^{\prime }-\tilde{h}%
X^{\prime },X^{\prime \prime })(X-\tilde{h}X)-\tilde{g}(X-\tilde{h}%
X,X^{\prime \prime })(X^{\prime }-\tilde{h}X^{\prime }) \\
& =(1-\tilde{\lambda})^{2}\tilde{g}(X^{\prime },X^{\prime \prime })X-(1-%
\tilde{\lambda})^{2}\tilde{g}(X,X^{\prime \prime })X^{\prime }.
\end{align*}%
So that, by using \eqref{i}, one has
\begin{align*}
\tilde{R}_{XX^{\prime }}X^{\prime \prime }& =\tilde{\varphi}\tilde{R}%
_{XX^{\prime }}\tilde{\varphi}X^{\prime \prime }-(1-\tilde{\lambda})^{2}%
\tilde{g}(X^{\prime },X^{\prime \prime })X+(1-\tilde{\lambda})^{2}\tilde{g}%
(X,X^{\prime \prime })X^{\prime } \\
& =(\tilde{\kappa}+\tilde{\mu})(\tilde{g}(\tilde{\varphi}X,\tilde{\varphi}%
X^{\prime \prime })X^{\prime }-\tilde{g}(\tilde{\varphi}X^{\prime },\tilde{%
\varphi}X^{\prime \prime })X)+(1-\tilde{\lambda})^{2}(\tilde{g}(X,X^{\prime
\prime })X^{\prime }-\tilde{g}(X^{\prime },X^{\prime \prime })X) \\
& =(2(\tilde{\lambda}-1)+\tilde{\mu})(\tilde{g}(X^{\prime },X^{\prime \prime
})X-\tilde{g}(X,X^{\prime \prime })X^{\prime }).
\end{align*}%
The proofs of remaining cases are similar.
\end{proof}

Using Theorem \ref{curv} one can easily prove the following corollaries.

\begin{corollary}
Let $(M,\tilde{\varphi},\xi ,\eta ,\tilde{g})$ be a paracontact metric $(%
\tilde{\kappa},\tilde{\mu})$-manifold such that $\tilde{\kappa}>-1$. Then
its Riemannian curvature tensor $\tilde{R}$ is given by following formula

\begin{eqnarray}
\tilde{g}(\tilde{R}_{XY}Z,W) &=&\left(-1+\frac{\tilde{\mu}}{2}\right)\left(
\tilde{g}(Y,Z)\tilde{g}(X,W)-\tilde{g}(X,Z)\tilde{g}(Y,W)\right)  \notag \\
&&+\tilde{g}(Y,Z)\tilde{g}(\tilde{h}X,W)-\tilde{g}(X,Z)\tilde{g}(\tilde{h}%
Y,W)  \notag \\
&&-\tilde{g}(Y,W)\tilde{g}(\tilde{h}X,Z)+\tilde{g}(X,W)\tilde{g}(\tilde{h}%
Y,Z)  \notag \\
&&+ \frac{-1+\frac{\tilde{\mu}}{2}}{\tilde{\kappa}+1}\left( \tilde{g}(\tilde{%
h}Y,Z)\tilde{g}(\tilde{h}X,W)-\tilde{g}(\tilde{h}X,Z)\tilde{g}(\tilde{h}%
Y,W)\right)  \notag \\
&&-\frac{\tilde{\mu}}{2}\left( \tilde{g}(\tilde{\varphi}Y,Z)\tilde{g}(\tilde{%
\varphi}X,W)-\tilde{g}(\tilde{\varphi}X,Z)\tilde{g}(\tilde{\varphi}%
Y,W)\right)  \notag \\
&&+ \frac{-\tilde{\kappa}-\frac{\tilde{\mu}}{2}}{\tilde{\kappa}+1} \left(
\tilde{g}(\tilde{\varphi}\tilde{h}Y,Z)\tilde{g}(\tilde{\varphi}\tilde{h}X,W)-%
\tilde{g}(\tilde{\varphi}\tilde{h}Y,W)\tilde{g}(\tilde{\varphi}\tilde{h}%
X,Z)\right)  \label{RXYZW} \\
&&+\tilde{\mu}\tilde{g}(\tilde{\varphi}X,Y)\tilde{g}(\tilde{\varphi}Z,W)
\notag \\
&&+\eta (X)\eta (W)\left( (\tilde{\kappa}+1-\frac{\tilde{\mu}}{2})\tilde{g}%
(Y,Z)+(\tilde{\mu}-1)\tilde{g}(\tilde{h}Y,Z)\right)  \notag \\
&&-\eta (X)\eta (Z)\left( (\tilde{\kappa}+1-\frac{\tilde{\mu}}{2})\tilde{g}%
(Y,W)+(\tilde{\mu}-1)\tilde{g}(\tilde{h}Y,W)\right)  \notag \\
&&+\eta (Y)\eta (Z)\left( (\tilde{\kappa}+1-\frac{\tilde{\mu}}{2})\tilde{g}%
(X,W)+(\tilde{\mu}-1)\tilde{g}(\tilde{h}X,W)\right)  \notag \\
&&-\eta (Y)\eta (W)\left( (\tilde{\kappa}+1-\frac{\tilde{\mu}}{2})\tilde{g}%
(X,Z)+(\tilde{\mu}-1)\tilde{g}(\tilde{h}X,Z)\right)  \notag
\end{eqnarray}
for all vector fields $X$, $Y$, $Z$, $W$ on $M$.
\end{corollary}

\begin{proof}
We can decompose an arbitrary vector field $X$ on $M$ uniquely as \ $X=X_{%
\tilde{\lambda}}+X_{-\tilde{\lambda}}+\eta (X)\xi$, \ where $X_{\tilde{%
\lambda}}\in\Gamma({\mathcal{D}}_{\tilde{h}}(\tilde{\lambda}))$ and $X_{-%
\tilde{\lambda}}\in\Gamma({\mathcal{D}}_{\tilde{h}}(-\tilde{\lambda}))$. We
then write $\tilde{R}_{XY}Z$ as a sum of terms of the form $\tilde{R}_{X_{\pm
\tilde{\lambda}} Y_{\pm _{\tilde{\lambda}}}}Z_{_{\pm \tilde{\lambda}}}$, $%
\tilde{R}_{XY}\xi$, $\tilde{R}_{X\xi}Z$. Then by Theorem \ref{curv} and %
\eqref{R(X,zeta)Y}, and taking into account that, in fact
\begin{equation*}
X_{\tilde{\lambda}}=\frac{1}{2}\bigl(X-\eta (X)\xi +\frac{1}{{\tilde{\lambda}%
}}\tilde{h}X\bigr), \text{ \ \ \ }X_{-\tilde{\lambda}}=\frac{1}{2}\bigl(%
X-\eta (X)\xi -\frac{1}{{\tilde{\lambda}}}\tilde{h}X\bigr),
\end{equation*}
we obtain \eqref{RXYZW}.
\end{proof}

\begin{corollary}
Let $(M,\tilde{\varphi},\xi ,\eta ,\tilde{g})$ be a paracontact $(\tilde{%
\kappa},\tilde{\mu})$-manifold such that $\tilde{\kappa}>-1$. Then for any $%
Z\in \Gamma ({\mathcal{D}})$ the $\xi $-sectional curvature $\tilde{K}%
(Z,\xi) $ is given by
\begin{equation*}
\tilde{K}(Z,\xi)=\tilde{\kappa}+\tilde{\mu}\frac{\tilde{g}(\tilde{h}Z,Z)}{%
\tilde{g}(Z,Z)}=\left\{
\begin{array}{ll}
\tilde\kappa+\tilde\lambda\tilde\mu, &
\hbox{if $Z\in\Gamma({\mathcal
D}_{\tilde h}(\tilde\lambda))$;} \\
\tilde\kappa-\tilde\lambda\tilde\mu, &
\hbox{if $Z\in\Gamma({\mathcal
D}_{\tilde h}(-\tilde\lambda))$.}%
\end{array}
\right.
\end{equation*}
Moreover, the sectional curvature of plane sections normal to $\xi$ is given
by
\begin{gather*}
K(X,X^{\prime})=2(\tilde{\lambda}-1)+\tilde{\mu}, \ \ K(Y,Y^{\prime})=-2(%
\tilde{\lambda}+1)+\tilde{\mu}, \\
K(X,Y)=(\tilde{\kappa}-\tilde{\mu})\frac{\tilde{g}(X,\tilde{\varphi}Y)^{2}}{%
\tilde{g}(X,X)\tilde{g}(Y,Y)},
\end{gather*}
for any $X,X^{\prime}\in\Gamma({\mathcal{D}}_{\tilde h}(\tilde\lambda))$, $%
Y,Y^{\prime}\in\Gamma({\mathcal{D}}_{\tilde h}(-\tilde\lambda))$.
\end{corollary}

\begin{corollary}
\label{ricci0} In any $(2n+1)$-dimensional paracontact $(\tilde{\kappa},%
\tilde{\mu})$-manifold $(M,\tilde{\varphi},\xi ,\eta ,\tilde{g})$ such that $%
\tilde{\kappa}>-1$, the Ricci operator $\tilde{Q}$ is given by
\begin{equation}
\tilde{Q}=(2(1-n)+n\tilde{\mu})I+(2(n-1)+\tilde{\mu})\tilde{h}+(2(n-1)+n(2%
\tilde{\kappa}-\tilde{\mu}))\eta \otimes \xi  \label{RICCI
OPERATOR}
\end{equation}%
In particular, $(M,\tilde{g})$ is $\eta $-Einstein if and only if $\tilde{\mu%
}=2(1-n)$, Einstein if and only if $\tilde{\kappa}=\tilde{\mu}=0$ and $n=1$
(in this case the manifold is Ricci-flat).
\end{corollary}

In particular it follows that in dimension $3$ any paracontact $(\tilde{%
\kappa},0)$-manifold with $\tilde{\kappa}>-1$ is $\eta $-Einstein.
Notice that, in dimension greater than $3$ no paracontact $(\tilde{\kappa},%
\tilde{\mu})$-manifold such that $\tilde{\kappa}>-1$ can be Einstein, since
one would find $\tilde{\kappa}=\frac{1-n^{2}}{n}$ and only for $n=1$ one has
that $\tilde{\kappa}>-1$.

\section{Paracontact $(\tilde{\protect\kappa},\tilde{\protect\mu})$%
-manifolds with $\tilde{\protect\kappa}<-1$}

\label{secondcase}

In this section we deal with paracontact $(\tilde\kappa,\tilde\mu)$%
-manifolds such that $\tilde{\kappa}<-1$. In this case, as stated in
Corollary \ref{bipara2}, $\tilde\varphi\tilde{h}$ is diagonalizable with
eigenvectors $0$, $\pm \tilde{\lambda}$, where $\tilde{\lambda}:=\sqrt{-1-%
\tilde{\kappa}}$. As for the case $\tilde\kappa>-1$ we start by proving that
the distributions defined by the eigenspaces of $\tilde\varphi\tilde{h}$
define two mutually orthogonal Legendre foliations. The main difference with
the case $\tilde\kappa>-1$ and, more in general, with the theory of contact
metric $(\kappa,\mu)$-spaces, is that they are not totally geodesic,
but they are totally umbilical.

\begin{theorem}
\label{foliations} Let $(M,\tilde{\varphi},\xi,\eta,\tilde{g})$ be a
paracontact $(\tilde\kappa,\tilde\mu)$-manifold such that $\tilde\kappa<-1$.
Then the eigendistributions ${\mathcal{D}}_{\tilde\varphi\tilde
h}(\tilde\lambda)$ and ${\mathcal{D}}_{\tilde\varphi\tilde
h}(-\tilde\lambda) $ of the operator $\tilde\varphi\tilde h$ are integrable
and define two mutually orthogonal Legendre foliations of $M$ with totally
umbilical leaves. Moreover, for any $X,Y\in{\mathcal{D}}_{\tilde\varphi%
\tilde h}(\pm\tilde\lambda)$, $\tilde\nabla_{X}Y\in{\mathcal{D}}%
_{\tilde\varphi\tilde{h}}(\pm\tilde\lambda)\oplus\mathbb{R}\xi$.
\end{theorem}

\begin{proof}
From \eqref{NAMLAFI}, \eqref{NAMLA X H}, \eqref{namlaFIH} we get the formula
\begin{equation}  \label{formula3}
(\tilde\nabla_{X}\tilde\varphi\tilde{h})Y-(\tilde\nabla_{Y}\tilde\varphi%
\tilde{h})X=-(1+\tilde\kappa)(\eta(X)Y-\eta(Y)X)+(1-\tilde\mu)(\eta(X)\tilde{%
h}Y-\eta(Y)\tilde{h}X)
\end{equation}
which holds for any paracontact $(\tilde\kappa,\tilde\mu)$-manifold. From %
\eqref{formula4} it follows that, for any $X,Y,Z\in\Gamma({\mathcal{D}})$, $%
\tilde{g}((\tilde\nabla_{X}\tilde\varphi\tilde{h})\tilde\varphi
Y-(\tilde\nabla_{\tilde\varphi Y}\tilde\varphi\tilde{h})X,Z)=0$, that is
\begin{equation}  \label{formula4}
\tilde{g}(\tilde\nabla_{X}\tilde\varphi\tilde{h}\tilde\varphi Y -
\tilde\varphi\tilde{h}\tilde\nabla_{X}\tilde\varphi Y -
\tilde\nabla_{\tilde\varphi Y}\tilde\varphi\tilde{h}X + \tilde\varphi\tilde{h%
}\tilde\nabla_{\tilde\varphi Y}X,Z)=0.
\end{equation}
Now, let us assume that $X,Y,Z\in\Gamma({\mathcal{D}}_{\tilde\varphi\tilde{h}%
}(\tilde\lambda))$. Then \eqref{formula4} reduces to $0=-\tilde\lambda\tilde{%
g}(\tilde\nabla_{X}\tilde\varphi Y,Z)-\tilde\lambda\tilde{g}%
(\tilde\nabla_{X}\tilde\varphi Y,Z)-\tilde\lambda\tilde{g}%
(\tilde\nabla_{\tilde\varphi Y}X,Z)+\tilde\lambda\tilde{g}%
(\tilde\nabla_{\tilde\varphi Y}X,Z)=-2\tilde\lambda\tilde{g}%
(\tilde\nabla_{X}\tilde\varphi Y,Z)$. Thus $\tilde\nabla_{X}Z\in{\mathcal{D}}%
_{\tilde\varphi\tilde{h}}(\tilde\lambda)\oplus\mathbb{R}\xi$. In particular,
since $\tilde{g}([X,Y],\xi)=\tilde{g}(X,\tilde\nabla_{Y}\xi)-\tilde{g}%
(Y,\tilde\nabla_{X}\xi)=\tilde\lambda\tilde{g}(Y,X)-\tilde\lambda\tilde{g}%
(X,Y)=0$ for all $X,Y\in\Gamma({\mathcal{D}}_{\tilde\varphi\tilde{h}%
}(\tilde\lambda))$, we have that ${\mathcal{D}}_{\tilde\varphi\tilde{h}%
}(\tilde\lambda)$ defines a foliation on $M$. Moreover, $\dim({\mathcal{D}}%
_{\tilde\varphi\tilde{h}}(\tilde\lambda))=n$ due to \cite[Proposition 3.2]%
{CAP2}. Hence, being an $n$-dimensional integrable subbundle of the contact
distribution, ${\mathcal{D}}_{\tilde\varphi\tilde{h}}(\tilde\lambda)$ is a
Legendre foliation of $M$. Similar arguments work also for ${\mathcal{D}}%
_{\tilde\varphi\tilde{h}}(-\tilde\lambda)$. In order to complete the proof,
let us consider $X\in\Gamma({\mathcal{D}}_{\tilde\varphi\tilde{h}%
}(\tilde\lambda))$ and $Y\in\Gamma({\mathcal{D}}_{\tilde\varphi\tilde{h}%
}(-\tilde\lambda))$. Then, for any $Z\in\Gamma({\mathcal{D}}_{\tilde\varphi%
\tilde{h}}(\tilde\lambda))$, $\tilde{g}(\tilde\nabla_{X}Y,Z)=-\tilde{g}%
(Y,\tilde\nabla_{X}Z)=0$, since $\tilde\nabla_{X}Z\in\Gamma({\mathcal{D}}%
_{\tilde\varphi\tilde{h}}(\tilde\lambda)\oplus\mathbb{R}\xi)$. Hence $%
\tilde\nabla_{X}Y\in\Gamma({\mathcal{D}}_{\tilde\varphi\tilde{h}%
}(-\tilde\lambda)\oplus\mathbb{R}\xi)$. In the same way one proves that $%
\tilde\nabla_{Y}X\in\Gamma({\mathcal{D}}_{\tilde\varphi\tilde{h}%
}(\tilde\lambda)\oplus\mathbb{R}\xi)$. \ Finally we prove that the leaves of
${\mathcal{D}}_{\varphi h}(\tilde\lambda)$ and ${\mathcal{D}}_{\varphi
h}(-\tilde\lambda)$ are totally umbilical. Since for any $%
X,X^{\prime}\in\Gamma({\mathcal{D}}_{\varphi h}(\tilde\lambda))$ $%
\tilde\nabla_{X}X^{\prime}\in\Gamma({\mathcal{D}}_{\varphi
h}(\tilde\lambda)\oplus\mathbb{R}\xi)$, $B(X,X^{\prime})$ is a section of $%
\mathbb{R}\xi$, where $B$ denotes the second fundamental form. Actually $%
B(X,X^{\prime})=-\tilde\lambda \tilde{g}(X,X^{\prime})\xi$. Indeed by using %
\eqref{nablaxi} one has
\begin{equation*}
\tilde{g}(B(X,X^{\prime}),\xi)=\tilde{g}(\tilde\nabla_{X}X^{\prime},\xi)=-%
\tilde{g}(X^{\prime},\tilde\nabla_{X}\xi)=\tilde{g}(X^{\prime},\tilde\varphi
X)-\tilde\lambda\tilde{g}(X,X^{\prime})=-\tilde\lambda\tilde{g}%
(X,X^{\prime}).
\end{equation*}
Then the mean curvature vector field is given by $H=-\tilde\lambda \xi$.
Hence $B(X,X^{\prime})= H \tilde{g}(X,X^{\prime})$. The proof for the other
foliation  is similar.
\end{proof}

\begin{remark}
Notice that the foliations ${\mathcal{D}}_{\tilde{\varphi}\tilde{h}}(\tilde{%
\lambda})$ and ${\mathcal{D}}_{\tilde{\varphi}\tilde{h}}(-\tilde{\lambda})$
are not totally geodesic. In fact a straightforward computation shows that,
for all $X,Y\in \Gamma ({\mathcal{D}}_{\tilde{\varphi}\tilde{h}}(\pm \tilde{%
\lambda}))$, $\tilde{g}(\tilde{\nabla}_{X}Y,\xi )=-\tilde{\lambda}\tilde{g}%
(X,\tilde{\varphi}Y)$, in general different from zero.
\end{remark}

Now we find the explicit expressions of the Pang invariants of the Legendre
foliations ${\mathcal{D}}_{\tilde{\varphi}\tilde{h}}(\tilde{\lambda})$ and ${%
\mathcal{D}}_{\tilde{\varphi}\tilde{h}}(-\tilde{\lambda})$. The proof is
similar to that of Theorem \ref{pang3} hence we omit it.

\begin{theorem}
\label{pang4} The Pang invariants of the Legendre foliations ${\mathcal{D}}%
_{\tilde\varphi\tilde{h}}(\tilde\lambda)$ and ${\mathcal{D}}_{\tilde\varphi%
\tilde{h}}(-\tilde\lambda)$ are given by
\begin{gather}
\Pi_{{\mathcal{D}}_{\tilde\varphi\tilde{h}}(\tilde\lambda)}=(\tilde\mu-2)%
\tilde{g}|_{{\mathcal{D}}_{\tilde\varphi\tilde{h}}(\tilde\lambda)\times{%
\mathcal{D}}_{\tilde\varphi\tilde{h}}(\tilde\lambda)}  \label{pang5} \\
\Pi_{{\mathcal{D}}_{\tilde\varphi\tilde{h}}(-\tilde\lambda)}=(\tilde\mu-2)%
\tilde{g}|_{{\mathcal{D}}_{\tilde\varphi\tilde{h}}(-\tilde\lambda)\times{%
\mathcal{D}}_{\tilde\varphi\tilde{h}}(-\tilde\lambda)}.  \label{pang6}
\end{gather}
\end{theorem}

\begin{proposition}
In any paracontact $(\tilde{\kappa},\tilde{\mu})$-manifold such that $\tilde{%
\kappa}<-1$ one has
\begin{equation}
(\tilde{\nabla}_{X}\tilde{\varphi}\tilde{h})Y=((1+\tilde{\kappa})\tilde{g}%
(X,Y)-\tilde{g}(\tilde{h}X,Y))\xi +\eta (Y)\tilde{h}(\tilde{h}X-X)-\tilde{\mu%
}\eta (X)\tilde{h}Y.  \label{formula8}
\end{equation}
\end{proposition}

\begin{proof}
Let $\{X_{1},\ldots ,X_{n},Y_{1},\ldots ,Y_{n},\xi \}$ be a $\tilde{\varphi}$%
-basis as in Lemma \ref{basis}. By using Theorem \ref{foliations} we have,
for any $X,Y\in \Gamma ({\mathcal{D}}_{\tilde{\varphi}\tilde{h}}(\tilde{%
\lambda}))$,
\begin{align*}
\tilde{\varphi}\tilde{h}\tilde{\nabla}_{X}Y& =\tilde{\varphi}\tilde{h}\left(
-\sum_{i=1}^{r}\tilde{g}(\tilde{\nabla}_{X}Y,X_{i})X_{i}+\sum_{i=r+1}^{n}%
\tilde{g}(\tilde{\nabla}_{X}Y,X_{i})X_{i}+\tilde{g}(\tilde{\nabla}_{X}Y,\xi
)\xi \right) \\
& =-\tilde{\lambda}\sum_{i=1}^{n}\tilde{g}(\tilde{\nabla}_{X}Y,X_{i})X_{i}+%
\tilde{\lambda}\sum_{i=r+1}^{n}\tilde{g}(\tilde{\nabla}_{X}Y,X_{i})X_{i} \\
& =\tilde{\lambda}\tilde{\nabla}_{X}Y-\tilde{\lambda}\tilde{g}(\tilde{\nabla}%
_{X}Y,\xi )\xi \\
& =\tilde{\nabla}_{X}\tilde{\varphi}\tilde{h}Y-\tilde{\lambda}\tilde{g}(%
\tilde{\nabla}_{X}Y,\xi )\xi .
\end{align*}%
It follows that
\begin{equation}
(\tilde{\nabla}_{X}\tilde{\varphi}\tilde{h})Y=-\tilde{\lambda}\tilde{g}(Y,%
\tilde{\nabla}_{X}\xi )\xi =-\tilde{\lambda}\tilde{g}(Y,-\tilde{\varphi}X+%
\tilde{\varphi}\tilde{h}X)\xi =-\tilde{\lambda}^{2}\tilde{g}(X,Y)\xi .
\label{formula5}
\end{equation}%
Now, let us consider $X\in \Gamma ({\mathcal{D}}_{\tilde{\varphi}\tilde{h}}(%
\tilde{\lambda}))$ and $Y\in \Gamma ({\mathcal{D}}_{\tilde{\varphi}\tilde{h}%
}(-\tilde{\lambda}))$. Arguing as in the previous case, one finds
\begin{equation}
(\tilde{\nabla}_{X}\tilde{\varphi}\tilde{h})Y=(\tilde{\nabla}_{Y}\tilde{%
\varphi}\tilde{h})X=-\tilde{\lambda}\tilde{g}(X,\tilde{\varphi}Y)\xi .
\label{formula6}
\end{equation}%
Finally, for any $X,Y\in \Gamma ({\mathcal{D}}_{\tilde{\varphi}\tilde{h}}(-%
\tilde{\lambda}))$ one has
\begin{equation}
(\tilde{\nabla}_{X}\tilde{\varphi}\tilde{h})Y=-{\tilde{\lambda}}^{2}\tilde{g}%
(X,Y)\xi .  \label{formula7}
\end{equation}%
Then \eqref{formula8} follows from \eqref{NMBLA ZETAH}, \eqref{formula5}, %
\eqref{formula6} and \eqref{formula7}.
\end{proof}

\begin{corollary}
In any paracontact $(\tilde\kappa,\tilde\mu)$-manifold such that $%
\tilde\kappa\neq -1$ one has
\begin{equation}  \label{nablah1}
(\tilde\nabla_{X}\tilde{h})Y=-((1+\tilde\kappa)\tilde{g}(X,\tilde\varphi Y)+%
\tilde{g}(X,\tilde\varphi\tilde{h}Y))\xi + \eta(Y)\tilde\varphi\tilde{h}(%
\tilde{h}X-X)-\tilde\mu\eta(X)\tilde\varphi\tilde{h}Y.
\end{equation}
\end{corollary}

\begin{proof}
If $\tilde\kappa>-1$, \eqref{nablah1} is just \eqref{nablah} and there is
nothing to prove. Next, if $\tilde\kappa<-1$, as $(\tilde\nabla_{X}\tilde{h}%
)Y = (\tilde\nabla_{X}\tilde\varphi)\tilde\varphi\tilde{h}Y +
\tilde\varphi(\tilde\nabla_{X}\tilde\varphi\tilde{h})Y$, the assertion
follows directly from \eqref{NAMLAFI} and \eqref{formula8}.
\end{proof}

Even if the Legendre foliations ${\mathcal{D}}_{\tilde\varphi\tilde{h}%
}(\pm\tilde\lambda)$ are not totally geodesic and thus many properties are
missing compared to the case $\tilde\kappa>-1$, also in this case we can
find an interesting relationship with contact Riemannian structures. We have
in fact the following result.

\begin{theorem}
Any positive or negative definite paracontact $(\tilde\kappa,\tilde\mu)$%
-manifold such that $\tilde\kappa<-1$ carries a canonical contact Riemannian
structure $(\phi,\xi,\eta,g)$ given by
\begin{equation}  \label{def2}
\phi:=\pm\frac{1}{\sqrt{-1-\tilde\kappa}}\tilde{h}, \ \
g:=-d\eta(\cdot,\phi\cdot)+\eta\otimes\eta,
\end{equation}
where the sign $\pm$ depends on the positive or negative definiteness of the
paracontact $(\tilde\kappa,\tilde\mu)$-manifold. Moreover, $%
(\phi,\xi,\eta,g) $ is a contact metric $(\kappa,\mu)$-structure, where
\begin{equation*}
\kappa=\tilde\kappa + 2-\left(1-\frac{\tilde\mu}{2}\right)^2, \ \ \ \mu=2.
\end{equation*}
\end{theorem}

\begin{proof}
Let us define a $(1,1)$-tensor field $\phi$ and a tensor $g$ of type $(0,2)$
by \eqref{def2}. First of all, using \eqref{H2}, one can easily prove that $%
\phi^{2}=-I+\eta\otimes\xi$. Next we prove that $g$ is a Riemannian metric.
By using the symmetry of the operator $\tilde\varphi\tilde{h}$ with respect
to $\tilde{g}$, one has, for any $X,Y\in\Gamma(TM)$,
\begin{align*}
g(X,Y)&=\mp\frac{1}{\sqrt{-1-\tilde\kappa}}d\eta(X,\tilde{h}Y)+\eta(X)\eta(Y)
\\
&=\mp\frac{1}{\sqrt{-1-\tilde{\kappa}}}\tilde{g}(X,\tilde\varphi\tilde{h}%
Y)+\eta(X)\eta(Y) \\
&=\mp\frac{1}{\sqrt{-1-\tilde\kappa}}\tilde{g}(\tilde\varphi\tilde{h}%
X,Y)+\eta(X)\eta(Y) \\
&=g(Y,X),
\end{align*}
so that $g$ is symmetric. In order to prove that it is also positive
definite, let us consider a $\tilde\varphi$-basis \ $\{X_{1},%
\ldots,X_{n},Y_{1},\ldots,Y_{n},\xi\}$ \ as \ in \ Lemma \ref{basis}. \ Then
\ we \ have \ that \ $g(\xi,\xi)=1$, \ $\tilde{g}(X_{i},X_{i})=\mp\frac{1}{%
\sqrt{-1-\tilde\kappa}}\tilde{g}(X_{i},\tilde\varphi\tilde{h}X_{i})$ \ $=$ \
$\mp\frac{1}{\sqrt{-1-\tilde\kappa}}\tilde\lambda\tilde{g}(X_{i},X_{i})$ \ $%
= $ \ $\mp\tilde{g}(X_{i},X_{i})$ \ $=$ \ $(\mp 1)(\mp 1)=1$ \ and \ $\tilde{%
g}(Y_{i},Y_{i})$ \ $=$ \ $\mp\frac{1}{\sqrt{-1-\tilde\kappa}}\tilde{g}%
(Y_{i},\tilde\varphi\tilde{h}Y_{i})=\pm\tilde{g}(Y_{i},Y_{i})=(\pm 1)(\pm
1)=1$. Finally one can straightforwardly check that $g(\phi X,\phi Y)$ $%
=g(X,Y)-\eta(X)\eta(Y)$ and $g(X,\phi Y)=d\eta(X,Y)$. Thus $%
(\phi,\xi,\eta,g) $ is a contact Riemannian structure. In order to prove
that it satisfies a $(\kappa,\mu)$-nullity condition, let us compute the
operator $h:=\frac{1}{2}{\mathcal{L}}_{\xi}\phi=\pm\frac{1}{2\sqrt{%
-1-\tilde\kappa}}{\mathcal{L}}_{\xi}\tilde{h}$. Notice that, for all $%
X\in\Gamma(TM)$,
\begin{align}  \label{formula11}
({\mathcal{L}}_{\xi}\tilde{h})X&=[\xi,\tilde{h}X]-\tilde{h}[\xi,X]  \notag \\
&=\tilde\nabla_{\xi}\tilde{h}X-\tilde\nabla_{\tilde{h}X}\xi-\tilde{h}%
\tilde\nabla_{\xi}X+\tilde{h}\tilde\nabla_{X}\xi \\
&=(\tilde\nabla_{\xi}\tilde{h})X-(-\tilde\varphi\tilde{h}X+\tilde\varphi%
\tilde{h}^{2}X)+\tilde{h}(-\tilde\varphi X + \tilde\varphi\tilde{h}X)  \notag
\\
&=(2-\tilde\mu)\tilde\varphi\tilde{h}X-2(1+\tilde\kappa)\tilde\varphi X.
\notag
\end{align}
Consequently
\begin{equation*}
h=\pm\left(\frac{1-\frac{\tilde\mu}{2}}{\sqrt{-1-\tilde\kappa}}\tilde\varphi%
\tilde{h} + \sqrt{-1-\tilde\kappa}\tilde\varphi\right).
\end{equation*}
We prove that $h$ is diagonalizable. Note that the matrix of $h$ with
respect to the $\tilde\varphi$-basis $\{X_i,Y_i,\xi\}$ is given by
\begin{equation*}
\pm\left(
\begin{array}{ccccccc}
1-\frac{\tilde\mu}{2} & \cdots & 0 & \sqrt{-1-\tilde\kappa} & \cdots & 0 & 0
\\
\vdots & \ddots & \vdots & \vdots & \ddots & \vdots & \vdots \\
0 & \cdots & 1-\frac{\tilde\mu}{2} & 0 & \cdots & \sqrt{-1-\tilde\kappa} & 0
\\
\sqrt{-1-\tilde\kappa} & \cdots & 0 & -1+\frac{\tilde\mu}{2} & \cdots & 0 & 0
\\
\vdots & \ddots & \vdots & \vdots & \ddots & \vdots & \vdots \\
0 & \cdots & \sqrt{-1-\tilde\kappa} & 0 & \cdots & -1+\frac{\tilde\mu}{2} & 0
\\
0 & \cdots & 0 & 0 & \cdots & 0 & 0 \\
&  &  &  &  &  &
\end{array}
\right).
\end{equation*}
Thus the characteristic polynomial is given by $P(\lambda)=\mp\lambda\left(%
\lambda^2-\left(1-\frac{\tilde\mu}{2}\right)^2+\tilde\kappa+1\right)^n$ and $%
h$ admits the eigenvalues $0$ and $\pm\sqrt{\left(1-\frac{\tilde\mu}{2}%
\right)^{2}-1-\tilde\kappa}$. After a long computation, one can prove that
the corresponding eigendistributions are ${\mathcal{D}}_{h}(0)=\mathbb{R}\xi$%
, and
\begin{equation*}
{\mathcal{D}}_{h}(\lambda)=\text{span}\{X_{1}+\alpha
Y_{1},\ldots,X_{n}+\alpha Y_{n}\}, \ \ \ {\mathcal{D}}_{h}(-\lambda)=\text{%
span}\{X_{1}-\beta Y_{1},\ldots,X_{n}-\beta Y_{n}\},
\end{equation*}
where $\alpha:=\frac{1}{\tilde\lambda}\sqrt{\left(1-\frac{\tilde\mu}{2}%
\right)^{2}-1-\tilde\kappa}-\frac{1}{\tilde\lambda}\left(1-\frac{\tilde\mu}{2%
}\right)$ and $\beta:=\frac{1}{\tilde\lambda}\sqrt{\left(1-\frac{\tilde\mu}{2%
}\right)^{2}-1-\tilde\kappa}+\frac{1}{\tilde\lambda}\left(1-\frac{\tilde\mu}{%
2}\right)$. Hence one can easily conclude that $h$ is diagonalizable. Next,
we prove that ${\mathcal{D}}_{h}(\lambda)$ and ${\mathcal{D}}_{h}(-\lambda)$
are Legendre foliations. Taking Theorem \ref{foliations} into account, we
can write $\tilde\nabla_{X_i}X_{j}=\sum_{k=1}^{n}f^{k}_{ij}X_{k}+\tilde{g}%
(\tilde\nabla_{X_i}X_{j},\xi)\xi$ and $\tilde\nabla_{Y_i}X_{j}=%
\sum_{k=1}^{n}g^{k}_{ij}X_{k}+\tilde{g}(\tilde\nabla_{Y_i}X_{j},\xi)\xi$,
for some functions $f_{ij}^k$, $g_{ij}^k$. Then one finds
\begin{align*}
\tilde\nabla_{X_i}Y_j=\tilde\nabla_{X_i}\tilde\varphi{X_j}%
=\tilde\varphi\tilde\nabla_{X_i}X_j - \tilde{g}(X_i-\tilde{h}%
X_i,X_j)\xi=\sum_{i=1}^{n}f_{ij}^{k}Y_{k} - \delta_{ij}\xi
\end{align*}
and, analogously, \ $\tilde\nabla_{Y_i}Y_{j}=\sum_{i=1}^{n}g^{k}_{ij}Y_{k} -
\tilde\lambda\delta_{ij}\xi$. \ Notice that $\tilde{g}(\tilde%
\nabla_{X_i}X_j,\xi)=-\tilde{g}(X_j,\tilde\nabla_{X_i}\xi)=-\tilde{g}%
(X_j,-\tilde\varphi X_i+\tilde\varphi\tilde{h}X_i)=-\tilde\lambda\delta_{ij}$
and $\tilde{g}(\tilde\nabla_{X_i}X_j,\xi)=\delta_{ij}$. After a
straightforward computation one can get
\begin{equation*}
\tilde\nabla_{X_i+\alpha Y_i}(X_j+\alpha
Y_j)=\sum_{k=1}^{n}((f_{ij}^{k}+\alpha g^{k}_{ij})(X_k + \alpha Y_k)) -
\tilde\lambda(1-\alpha^2)\delta_{ij}\xi.
\end{equation*}
Then
\begin{align*}
[X_{i}+\alpha Y_{i}, X_{j}+\alpha Y_{j}] &= \tilde\nabla_{X_i+\alpha
Y_i}(X_j+\alpha Y_j) - \tilde\nabla_{X_j+\alpha Y_j}(X_i+\alpha Y_i) \\
&=\sum_{k=1}^{n}(f^{k}_{ij}-f^{k}_{ji} +
\alpha(g^{k}_{ik}-g^{k}_{ji}))(X_{k}+\alpha Y_{k}),
\end{align*}
consequently ${\mathcal{D}}_{h}(\lambda)$ is involutive. In a similar way
one proves the integrability of ${\mathcal{D}}_{h}(-\lambda)$. Moreover, for
any $X\in\Gamma({\mathcal{D}}_{h}(\pm\lambda))$, $\eta(X)=\pm\frac{1}{\lambda%
}\eta(hX)=0$. Being $n$-dimensional integrable subbundles of the contact
distribution, ${\mathcal{D}}_{h}(\lambda)$ and ${\mathcal{D}}_{h}(-\lambda)$
are Legendre foliations. In order to prove that $(M,\varphi,\xi,\eta,g)$ is
a contact metric $(\kappa,\mu)$-space, we show that the bi-Legendrian
structure $({\mathcal{D}}_{h}(\lambda),{\mathcal{D}}_{h}(-\lambda))$
satisfies the hypotheses of Theorem \ref{characterization}. First, due to
the property $\phi h =-h\phi$, we have that ${\mathcal{D}}_{h}(\lambda)$ and
${\mathcal{D}}_{h}(-\lambda)$ are \emph{conjugate} Legendre foliations, that
is $\phi{\mathcal{D}}_{h}(\pm\lambda)={\mathcal{D}}_{h}(\mp\lambda)$. Let us
consider the bi-Legendrian connection $\nabla^{bl}$ associated to $({%
\mathcal{D}}_{h}(\lambda),{\mathcal{D}}_{h}(-\lambda))$. By definition of
bi-Legendrian connection, $\nabla^{bl}$ satisfies (i) and (iii) of Theorem %
\ref{characterization}. Moreover, $\nabla^{bl}\eta=\nabla^{bl}d\eta=0$ and,
since $\nabla^{bl}$ preserves ${\mathcal{D}}_{h}(\pm\lambda)$, also $%
\nabla^{bl}h=0$. It remains to prove that $\nabla^{bl}$ preserves the tensor
field $\phi$ and is a metric connection. Notice that, by virtue of \cite[%
Proposition 2.9]{CAP3}, in order to prove that $\nabla^{bl}\phi=0$ and $%
\nabla^{bl}g=0$, it is enough to show that ${\mathcal{D}}_{h}(\pm\lambda)$
are totally geodesic foliations (with respect to $g$). Let $\nabla$ be the
Levi-Civita connection of $g$ and $X,X^{\prime}\in\Gamma({\mathcal{D}}%
_{h}(\lambda))$, $Y\in\Gamma({\mathcal{D}}_{h}(-\lambda))$. As $g=\mp\frac{1%
}{\tilde\lambda}\tilde{g}(\cdot,\tilde\varphi\tilde{h}\cdot)+\eta\otimes\eta$%
, after a very long computation one can get
\begin{align}  \label{formula9}
g(\nabla_{X}X^{\prime},Y)&=-\frac{2\lambda}{2-\tilde\mu}\tilde{g}%
(\tilde\nabla_{X}X^{\prime},Y) - \frac{3\tilde\lambda}{2-\tilde\mu}%
d^2\eta(X,X^{\prime},Y)  \notag \\
&\quad - \frac{1}{2\tilde\lambda}(X(d\eta(X^{\prime},\tilde{h}%
Y))+X^{\prime}(d\eta(X,\tilde{h}Y))+d\eta([X,X^{\prime}],\tilde{h}Y))
\end{align}
Since $\tilde{h}$ maps ${\mathcal{D}}_{h}(\pm\lambda)$ onto ${\mathcal{D}}%
_{h}(\mp\lambda)$ and due to Theorem \ref{foliations}, \eqref{formula9}
implies that $\tilde{g}(\nabla_{X}X^{\prime},Y)=0$. Moreover, $%
g(\nabla_{X}X^{\prime},\xi)=-g(X^{\prime},\nabla_{X}\xi)=-g(X^{\prime},-\phi
X-\phi hX)=(1+\lambda)g(X^{\prime},\phi X)=0$. Thus ${\mathcal{D}}%
_{h}(\lambda)$ is totally geodesic and in a similar way one can prove that
also ${\mathcal{D}}_{h}(-\lambda)$ is totally geodesic. It follows that the
bi-Legendrian connection $\nabla^{bl}$ preserves $g$ and $\phi$. Therefore
all the assumptions of Theorem \ref{characterization} are satisfied and we
conclude that $(\phi,\xi,\eta,g)$ is a contact metric $(\kappa,\mu)$%
-structure. It remains to find explicitly $\kappa$ and $\mu$. We have
immediately that $\kappa=1-\lambda^2=\tilde\kappa+2-\left(1-\frac{\tilde\mu}{%
2}\right)^2$. In order to find $\mu$, notice that by \cite[(3.19)]{BKP}, one
has
\begin{equation}  \label{formula10}
(\nabla_{\xi}h)X=\mu h\phi X=\mu\left(\left(\frac{\tilde\mu}{2}%
-1\right)\tilde\varphi X + \tilde\varphi\tilde{h}X\right),
\end{equation}
for all $X\in\Gamma(TM)$. On the other hand, by using \eqref{B2} and the
relation $h^{2}=(\kappa-1)\phi$ (cf. \cite{BKP}), we have
\begin{align}  \label{formula13}
(\nabla_{\xi}h)X&=\nabla_{hX}\xi+[\xi,hX]-h\nabla_{X}\xi-h[\xi,X]  \notag \\
&=-2\phi hX-2\phi h^{2}X+({\mathcal{L}}_{\xi}h)X \\
&=-2\phi hX - 2(1-\kappa)\phi X + ({\mathcal{L}}_{\xi}h)X.  \notag
\end{align}
But, due to \eqref{formula11},
\begin{align}  \label{formula12}
({\mathcal{L}}_\xi h)X&=\pm\left(\frac{2-\tilde\mu}{2\sqrt{-1-\tilde\kappa}}(%
{\mathcal{L}}_{\xi}\tilde\varphi)\tilde{h}X+\frac{2-\tilde\mu}{2\sqrt{%
-1-\tilde\kappa}}\tilde\varphi({\mathcal{L}}_{\xi}\tilde{h})X+\tilde\lambda({%
\mathcal{L}}_{\xi}\tilde\varphi)X\right)  \notag \\
&=\pm\left(\frac{2-\tilde\mu}{\sqrt{-1-\tilde\kappa}}\tilde{h}^{2} + \frac{%
(2-\tilde\mu)^{2}}{2\sqrt{-1-\tilde\kappa}}\tilde\varphi^{2}\tilde{h}X +%
\sqrt{-1-\tilde\kappa})(2-\tilde\mu)\tilde\varphi^{2}X + 2\tilde\lambda
\tilde{h}X\right) \\
&=\pm\left(\frac{(2-\tilde\mu)^2}{2\sqrt{-1-\tilde\kappa}}+2\sqrt{%
-1-\tilde\kappa}\right)\tilde{h}X.  \notag
\end{align}
Thus by \eqref{formula12} and \eqref{formula13} we get
\begin{equation}  \label{formula14}
(\nabla_{\xi}h)X=-2\phi hX - 2(1-\kappa)\phi X \pm\left(\frac{(2-\tilde\mu)^2%
}{\sqrt{-1-\tilde\kappa}}+2\sqrt{-1-\tilde\kappa}\right)\tilde{h}X,
\end{equation}
where the sign $\pm$ depends on the positive or negative definiteness of the
paracontact $(\tilde\kappa,\tilde\mu)$-manifold $(M,\tilde\varphi,\xi,\eta,%
\tilde{g})$. Thus comparing \eqref{formula10} with \eqref{formula14} we
obtain $\mu=2$ both in the positive and in the negative definite case.
\end{proof}

\begin{remark}
Notice that, since $\tilde\kappa<-1$, in this case we can not perform a
construction analogous to that one described in Remark \ref{remarkreferee}.
\end{remark}

We now pass to study the curvature properties of a paracontact $%
(\tilde\kappa,\tilde\mu)$-manifold with $\tilde\kappa<-1$. We start
observing that \eqref{RHZ} holds also in the case $\tilde\kappa<-1$ since
its proof only needs the expression of the covariant derivative of $\tilde h$%
, which is the same in the cases $\tilde\kappa<-1$ and $\tilde\kappa>-1$
(cf. \eqref{nablah} and \eqref{nablah1}). Moreover, by combining %
\eqref{CURVATURE 10} with \eqref{nablah1} we get

\begin{align}  \label{curv1}
\tilde{R}_{XY}\tilde{\varphi}Z - \tilde{\varphi}\tilde{R}_{XY}Z &= \bigl(%
\eta(Y)\tilde{g}((1+\tilde{\kappa})\tilde{\varphi}X+(\tilde\mu-1)\tilde%
\varphi\tilde{h}X,Z)-\eta(X)\tilde{g}((1+\tilde\kappa)\tilde{\varphi}Y
\notag \\
&\quad+(\tilde\mu-1)\tilde\varphi\tilde{h}Y,Z)\bigr)\xi -
\eta(Y)\eta(Z)((1+\tilde\kappa)\tilde\varphi X+(\tilde\mu-1)\tilde{\varphi}%
\tilde{h}X)  \notag \\
&\quad+\eta(X)\eta(Z)((1+\tilde\kappa)\tilde{\varphi}Y+(\tilde\mu-1)\tilde%
\varphi\tilde{h}Y) \\
&\quad+\tilde{g}(Y-\tilde{h}Y,Z)\tilde\varphi(X-\tilde{h}X)-\tilde{g}(X-%
\tilde{h}X,Z)\tilde\varphi(Y-\tilde{h}Y)  \notag \\
&\quad-\tilde{g}(\tilde\varphi(X-\tilde{h}X),Z)(Y-\tilde{h}Y)+\tilde{g}(%
\tilde{\varphi}(Y-\tilde{h}Y),Z)(X-\tilde{X}) .  \notag
\end{align}
From \eqref{RHZ} and \eqref{curv1}, using \eqref{H2} and the properties of
the operator $\tilde{h}$, it follows that
\begin{align}  \label{curv2}
\tilde{R}_{XY}\tilde\varphi\tilde{h}Z-\tilde\varphi\tilde{h}\tilde{R}%
_{XY}Z&= \tilde{R}_{XY}\tilde\varphi\tilde{h}Z-\tilde\varphi\tilde{R}_{XY}%
\tilde{h}Z + \tilde\varphi(\tilde{R}_{XY}\tilde{h}Z-\tilde{h}\tilde{R}_{XY}Z)
\notag \\
&=\bigl(\eta(Y)\tilde{g}((1+2\tilde\kappa)\tilde\varphi X
+(\tilde\mu-1)\tilde\varphi\tilde{h}X,\tilde{h}Z)  \notag \\
&\quad-\eta(X)\tilde{g}((1+2\tilde\kappa)\tilde\varphi
Y+(\tilde\mu-1)\tilde\varphi\tilde{h}Y,\tilde{h}Z)\bigr)\xi  \notag \\
&\quad+\bigl(\tilde{g}(Y-\tilde{h}Y,\tilde{h}Z)-\tilde\mu(1+\tilde\kappa)%
\eta(Y)\eta(Z)\bigr)\tilde\varphi X  \notag \\
&\quad-\bigl(\tilde{g}(X-\tilde{h}X,\tilde{h}Z)-\tilde\mu(1+\tilde\kappa)%
\eta(X)\eta(Z)\bigr)\tilde\varphi Y \\
&\quad-\bigl(\tilde{g}(Y-\tilde{h}Y,\tilde{h}Z)+\tilde\kappa\eta(Y)\eta(Z)%
\bigr)\tilde\varphi\tilde{h}X +\bigl(\tilde{g}(X-\tilde{h}X,\tilde{h}%
Z)+\tilde\kappa\eta(X)\eta(Z)\bigr)\tilde\varphi\tilde{h}Y  \notag \\
&\quad-(1+\tilde\kappa)\tilde{g}(Y,\tilde\varphi Z+\tilde\varphi\tilde{h}Z)X
+ (1+\tilde\kappa)\tilde{g}(X,\tilde\varphi Z+\tilde\varphi\tilde{h}Z)Y
\notag \\
&\quad+\tilde{g}(Y,\tilde\varphi Z+\tilde\varphi\tilde{h}Z)\tilde{h}X-\tilde{%
g}(X,\tilde\varphi Z+\tilde\varphi\tilde{h}Z)\tilde{h}Y-2\tilde\mu\tilde{g}%
(X,\tilde\varphi Y)\tilde{h}Z.  \notag
\end{align}

We can prove now that also in the case $\tilde\kappa<-1$ the $%
(\tilde\kappa,\tilde\mu)$-nullity condition determines the curvature tensor
field completely.

\begin{theorem}
\label{curvatura} Let $(M,\tilde{\varphi},\xi ,\eta ,\tilde{g})$ be a
paracontact metric $(\tilde{\kappa},\tilde{\mu})$-manifold such that $\tilde{%
\kappa}<-1$. Then the curvature tensor field of $M$ satisfies the following
relations:
\begin{gather}
\tilde{R}_{XX^{\prime}}X^{\prime \prime}=(\tilde{\kappa}-1+\tilde{\mu}%
)\left( \tilde{g}(X^{\prime},X^{\prime\prime })X-\tilde{g}(X,X^{\prime
\prime})X^{\prime }\right) + \tilde\lambda(\tilde{g}(X^{\prime },X^{\prime
\prime })\tilde{\varphi}X -\tilde{g}(X,X^{\prime \prime })\tilde{\varphi}%
X^{\prime })  \label{first} \\
\tilde{R}_{XX^{\prime }}Y=-\tilde{\lambda}\left(\tilde{g}(X^{\prime },\tilde{
\varphi}Y)X - \tilde{g}(X,\tilde{\varphi}Y)X^{\prime }\right)-(1- \tilde{\mu}%
)\left(\tilde{g}(X^{\prime },\tilde{\varphi}Y)\tilde{\varphi}X -\tilde{g}(X,%
\tilde{\varphi}Y)\tilde{\varphi} X^{\prime }\right)  \label{second} \\
\tilde{R}_{XY}X^{\prime } =-\tilde{\lambda}\tilde{g}(X^{\prime },\tilde{
\varphi}Y)X-\tilde{g}(X^{\prime },\tilde{\varphi}Y)\tilde{\varphi}X-\tilde{
\lambda}^{2}\tilde{g}(X,X^{\prime })Y +\tilde{\lambda}\tilde{g}(X,X^{\prime
})\tilde{\varphi}Y-\tilde{\mu} \tilde{g}(X,\tilde{\varphi}Y)\tilde{\varphi}%
X^{\prime } \\
\tilde{R}_{XY}Y^{\prime } =\tilde{\lambda}^{2}\tilde{g}(Y,Y^{\prime })X+
\tilde{\lambda}\tilde{g}(Y,Y^{\prime })\tilde{\varphi}X+\tilde{\lambda}
\tilde{g}(X,\tilde{\varphi}Y^{\prime })Y-\tilde{g}(X,\tilde{\varphi}
Y^{\prime })\tilde{\varphi}Y -\tilde{\mu}\tilde{g}(X,\tilde{\varphi}Y)\tilde{%
\varphi}Y^{\prime } \\
\tilde{R}_{YY^{\prime}}X=-\tilde{\lambda}\left(\tilde{g}(X,\tilde{\varphi}
Y^{\prime })Y -\tilde{g}(X,\tilde{\varphi}Y)Y^{\prime }\right) +(1-\tilde{\mu%
})\left(\tilde{g}(X,\tilde{\varphi} Y^{\prime })\tilde{\varphi}Y-\tilde{g}(X,%
\tilde{\varphi}Y) \tilde{\varphi}Y^{\prime }\right) \\
\tilde{R}_{YY^{\prime }}Y^{\prime \prime }=(\tilde{\kappa}-1+\tilde{\mu}%
)\left( \tilde{g}(Y^{\prime },Y^{\prime \prime })Y - \tilde{g}(Y,Y^{\prime
\prime })Y^{\prime }\right) -\tilde{\lambda}(\tilde{g} (Y^{\prime
},Y^{\prime \prime })\tilde{\varphi}Y -\tilde{g}(Y,Y^{\prime \prime })\tilde{%
\varphi} Y^{\prime })
\end{gather}
for any $X,X^{\prime },X^{\prime \prime }\in \Gamma ({\mathcal{D}}_{\tilde{
\varphi}\tilde{h}}(\tilde{\lambda}))$ and $Y,Y^{\prime },Y^{\prime \prime
}\in \Gamma ({\mathcal{D}}_{\tilde{\varphi}\tilde{h}}(-\tilde{\lambda}))$.
\end{theorem}

\begin{proof}
We prove \eqref{first} and \eqref{second}, the remaining relations being
analogous. First notice that, since ${\mathcal{D}}_{\tilde{\varphi}\tilde{h}%
}(\tilde{\lambda})$ and ${\mathcal{D}}_{\tilde{\varphi}\tilde{h}}(-\tilde{%
\lambda})$ are not totally geodesic, $\tilde{R}_{XX^{\prime }}Y$ has
components along both ${\mathcal{D}}_{\tilde{\varphi}\tilde{h}}(\tilde{%
\lambda})$ and ${\mathcal{D}}_{\tilde{\varphi}\tilde{h}}(-\tilde{\lambda})$.
Moreover it has no components along $\mathbb{R}\xi $ because
\begin{align*}
\tilde{g}(\tilde{R}_{XX^{\prime }}Y,\xi )& =-\tilde{R}(\xi ,Y,X,X^{\prime
})=-\tilde{R}(Y,\xi ,X,X^{\prime })=-\tilde{g}(\tilde{R}_{XX^{\prime }}\xi
,Y) \\
& =-\tilde{g}(\tilde{\kappa}(\eta (X^{\prime })X-\eta (X)X^{\prime })+\tilde{%
\mu}(\eta (X^{\prime })\tilde{h}X-\eta (X)\tilde{h}X^{\prime },Y)=0.
\end{align*}%
Due to the fact that $\tilde{h}{\mathcal{D}}_{\tilde{\varphi}\tilde{h}}(\pm
\tilde{\lambda})={\mathcal{D}}_{\tilde{\varphi}\tilde{h}}(\mp \tilde{\lambda}%
)$, \eqref{curv2} implies
\begin{align*}
\tilde{\lambda}\tilde{R}_{XX^{\prime }}Y+\tilde{\varphi}\tilde{h}\tilde{R}%
_{XX^{\prime }}Y& =-\tilde{g}(X^{\prime },\tilde{h}Y)\tilde{\varphi}X+\tilde{%
g}(X,\tilde{h}Y)\tilde{\varphi}X^{\prime }+\tilde{g}(X^{\prime },\tilde{%
\lambda}\tilde{h}Y+(1+\tilde{\kappa})\tilde{\varphi}Y)X \\
& \quad -\tilde{g}(X,\tilde{\lambda}\tilde{h}Y+(1+\tilde{\kappa})\tilde{%
\varphi}Y)X^{\prime }-\tilde{g}(X^{\prime },\tilde{\varphi}Y)\tilde{h}X+%
\tilde{g}(X,\tilde{\varphi}Y)\tilde{h}X^{\prime }.
\end{align*}%
Taking the inner product with any $U\in \Gamma ({\mathcal{D}}_{\tilde{\varphi%
}\tilde{h}}(\tilde{\lambda}))$ and using the symmetry of the operator $%
\tilde{\varphi}\tilde{h}$, we have
\begin{equation*}
2\tilde{\lambda}\tilde{g}(\tilde{R}_{XX^{\prime }}Y,U)=\tilde{g}(X^{\prime },%
\tilde{\lambda}\tilde{h}Y+(1+\tilde{\kappa})\tilde{\varphi}Y)\tilde{g}(X,U)-%
\tilde{g}(X,\tilde{\lambda}\tilde{h}Y+(1+\tilde{\kappa})\tilde{\varphi}Y)%
\tilde{g}(X^{\prime },U).
\end{equation*}%
Now notice that $\tilde{h}Y=-\frac{1}{\tilde{\lambda}}\tilde{h}\tilde{\varphi%
}\tilde{h}Y=\frac{1}{\tilde{\lambda}}\tilde{\varphi}\tilde{h}^{2}Y=-\tilde{%
\lambda}\tilde{\varphi}Y$. Hence the previous equations yields
\begin{equation}
\tilde{g}(\tilde{R}_{XX^{\prime }}Y,U)=\tilde{\lambda}\tilde{g}(X,\tilde{%
\varphi}Y)\tilde{g}(X^{\prime },U)-\tilde{\lambda}\tilde{g}(X^{\prime },%
\tilde{\varphi}Y)\tilde{g}(X,U).  \label{curv3}
\end{equation}%
Arguing in a similar way one can prove that
\begin{equation}
\tilde{g}(\tilde{R}_{XY}X^{\prime },V)=\tilde{g}(X^{\prime },\tilde{\varphi}%
Y)\tilde{g}(X,\tilde{\varphi}V)+(1+\tilde{\kappa})\tilde{g}(X,X^{\prime })%
\tilde{g}(Y,V)+\tilde{\mu}\tilde{g}(X,\tilde{\varphi}Y)\tilde{g}(X^{\prime },%
\tilde{\varphi}V)  \label{curv4}
\end{equation}%
for all $V\in \Gamma ({\mathcal{D}}_{\tilde{\varphi}\tilde{h}}(-\tilde{%
\lambda}))$. Now let us consider a $\tilde{\varphi}$-basis $\left\{ e_{i},%
\tilde{\varphi}e_{i},\xi \right\} $, $i\in \left\{ 1,\ldots ,n\right\} $, as
in Lemma \ref{basis}. Note that by \eqref{curv4} we have
\begin{align}
\tilde{g}(\tilde{R}_{XX^{\prime }}Y,\tilde{\varphi}e_{i})& =-\tilde{g}(%
\tilde{R}_{X^{\prime }Y}X,\tilde{\varphi}e_{i})+\tilde{g}(\tilde{R}%
_{XY}X^{\prime },\tilde{\varphi}e_{i})  \notag  \label{curv5} \\
& =-\tilde{g}(Y,\tilde{\varphi}X)\tilde{g}(\tilde{\varphi}X^{\prime },\tilde{%
\varphi}e_{i})-(1+\tilde{\kappa})\tilde{g}(X,X^{\prime })\tilde{g}(Y,\tilde{%
\varphi}e_{i})+\tilde{\mu}\tilde{g}(X^{\prime },\tilde{\varphi}Y)\tilde{g}(%
\tilde{\varphi}X,\tilde{\varphi}e_{i})  \notag \\
& \quad +\tilde{g}(Y,\tilde{\varphi}X^{\prime })\tilde{g}(\tilde{\varphi}X,%
\tilde{\varphi}e_{i})+(1+\tilde{\kappa})\tilde{g}(X,X^{\prime })\tilde{g}(Y,%
\tilde{\varphi}e_{i})-\tilde{\mu}\tilde{g}(X,\tilde{\varphi}Y)\tilde{g}(%
\tilde{\varphi}X^{\prime },\tilde{\varphi}e_{i})  \notag \\
& =(\tilde{\mu}-1)\tilde{g}(X,\tilde{\varphi}Y)\tilde{g}(X^{\prime },e_{i})-(%
\tilde{\mu}-1)\tilde{g}(X^{\prime },\tilde{\varphi}Y)\tilde{g}(X,e_{i}).
\end{align}%
Then, by using \eqref{curv3} and \eqref{curv5} we get
\begin{align*}
\tilde{R}_{XX^{\prime }}Y& =\sum_{i=1}^{r}\tilde{g}(\tilde{R}_{XX^{\prime
}}Y,e_{i})e_{i}-\sum_{i=r+1}^{n}\tilde{g}(\tilde{R}_{XX^{\prime
}}Y,e_{i})e_{i}-\sum_{i=1}^{r}\tilde{g}(\tilde{R}_{XX^{\prime }}Y,\tilde{%
\varphi}e_{i})\tilde{\varphi}e_{i}+\sum_{i=r+1}^{n}\tilde{g}(\tilde{R}%
_{XX^{\prime }}Y,\tilde{\varphi}e_{i})\tilde{\varphi}e_{i} \\
& =-\tilde{\lambda}\tilde{g}(X^{\prime },\tilde{\varphi}Y)X+\tilde{\lambda}%
\tilde{g}(X,\tilde{\varphi}Y)X^{\prime }-(1-\tilde{\mu})\tilde{g}(X^{\prime
},\tilde{\varphi}Y)\tilde{\varphi}X+(1-\tilde{\mu})\tilde{g}(X,\tilde{\varphi%
}Y)\tilde{\varphi}X^{\prime }.
\end{align*}%
In order to prove \eqref{first}, we use \eqref{curv1} and we get, after a
long computation,
\begin{align*}
\tilde{R}_{XX^{\prime }}X^{\prime \prime }& =\tilde{g}(\tilde{R}_{XX^{\prime
}}X^{\prime \prime },\xi )\xi +\tilde{\varphi}\tilde{R}_{XX^{\prime }}\tilde{%
\varphi}X^{\prime \prime }-\tilde{\varphi}\bigl(\tilde{g}(X^{\prime }-\tilde{%
h}X^{\prime },X^{\prime \prime })\tilde{\varphi}(X-\tilde{h}X) \\
& \quad -\tilde{g}(X-\tilde{h}X,X^{\prime \prime })\tilde{\varphi}(X^{\prime
}-\tilde{h}X^{\prime })-\tilde{g}(\tilde{\varphi}(X-\tilde{h}X),X^{\prime
\prime })(X^{\prime }-\tilde{h}X^{\prime }) \\
& \quad +\tilde{g}(\tilde{\varphi}(X^{\prime }-\tilde{h}X^{\prime
}),X^{\prime \prime })(X-\tilde{h}X)\bigr) \\
& =(\tilde{\kappa}-1+\tilde{\mu})\tilde{g}(X^{\prime },X^{\prime \prime })X+%
\tilde{g}(X^{\prime },X^{\prime \prime })\tilde{h}X-(\tilde{\kappa}-1+\tilde{%
\mu})\tilde{g}(X,X^{\prime \prime })X^{\prime }-\tilde{g}(X,X^{\prime \prime
})\tilde{h}X^{\prime }.
\end{align*}
\end{proof}

\begin{corollary}
Let $(M,\tilde{\varphi},\xi ,\eta ,\tilde{g})$ be a paracontact metric $(%
\tilde{\kappa},\tilde{\mu})$-manifold such that $\tilde{\kappa}<-1$. Then
its Riemannian curvature tensor $\tilde{R}$ is given by following formula

\begin{eqnarray}
\tilde{g}(\tilde{R}_{XY}Z,W) &=&\left( -1+\frac{\tilde{\mu}}{2}\right)
\left( \tilde{g}(Y,Z)\tilde{g}(X,W)-\tilde{g}(X,Z)\tilde{g}(Y,W)\right)
\notag \\
&&+\tilde{g}(Y,Z)\tilde{g}(\tilde{h}X,W)-\tilde{g}(X,Z)\tilde{g}(\tilde{h}%
Y,W)  \notag \\
&&-\tilde{g}(Y,W)\tilde{g}(\tilde{h}X,Z)+\tilde{g}(X,W)\tilde{g}(\tilde{h}%
Y,Z)  \notag \\
&&+\frac{-1+\frac{\tilde{\mu}}{2}}{\tilde{\kappa}+1}\left( \tilde{g}(\tilde{h%
}Y,Z)\tilde{g}(\tilde{h}X,W)-\tilde{g}(\tilde{h}X,Z)\tilde{g}(\tilde{h}%
Y,W)\right)  \notag \\
&&-\frac{\tilde{\mu}}{2}\left( \tilde{g}(\tilde{\varphi}Y,Z)\tilde{g}(\tilde{%
\varphi}X,W)-\tilde{g}(\tilde{\varphi}X,Z)\tilde{g}(\tilde{\varphi}%
Y,W)\right)  \notag \\
&&+\frac{-\tilde{\kappa}-\frac{\tilde{\mu}}{2}}{\tilde{\kappa}+1}\left(
\tilde{g}(\tilde{\varphi}\tilde{h}Y,Z)\tilde{g}(\tilde{\varphi}\tilde{h}X,W)-%
\tilde{g}(\tilde{\varphi}\tilde{h}Y,W)\tilde{g}(\tilde{\varphi}\tilde{h}%
X,Z)\right)  \label{RXYZW2} \\
&&+\tilde{\mu}\tilde{g}(\tilde{\varphi}X,Y)\tilde{g}(\tilde{\varphi}Z,W)
\notag \\
&&+\eta (X)\eta (W)\left( (\tilde{\kappa}+1-\frac{\tilde{\mu}}{2})\tilde{g}%
(Y,Z)+(\tilde{\mu}-1)\tilde{g}(\tilde{h}Y,Z)\right)  \notag \\
&&-\eta (X)\eta (Z)\left( (\tilde{\kappa}+1-\frac{\tilde{\mu}}{2})\tilde{g}%
(Y,W)+(\tilde{\mu}-1)\tilde{g}(\tilde{h}Y,W)\right)  \notag \\
&&+\eta (Y)\eta (Z)\left( (\tilde{\kappa}+1-\frac{\tilde{\mu}}{2})\tilde{g}%
(X,W)+(\tilde{\mu}-1)\tilde{g}(\tilde{h}X,W)\right)  \notag \\
&&-\eta (Y)\eta (W)\left( (\tilde{\kappa}+1-\frac{\tilde{\mu}}{2})\tilde{g}%
(X,Z)+(\tilde{\mu}-1)\tilde{g}(\tilde{h}X,Z)\right)  \notag
\end{eqnarray}%
for all vector fields $X$, $Y$, $Z$, $W$ on $M$.
\end{corollary}

\begin{proof}
We can decompose an arbitrary vector field $X$ on $M$ uniquely as \ $X=X_{%
\tilde{\lambda}}+X_{-\tilde{\lambda}}+\eta (X)\xi $, where $X_{\tilde{\lambda%
}}\in \Gamma ({\mathcal{D}}_{\tilde{\varphi}\tilde{h}}(\tilde{\lambda}))$
and $X_{-\tilde{\lambda}}\in \Gamma ({\mathcal{D}}_{\tilde{\varphi}\tilde{h}%
}(-\tilde{\lambda}))$. We then write $\tilde{R}_{XY}Z$ as a sum of terms of
the form $\tilde{R}_{X_{\pm \tilde{\lambda}} Y_{\pm _{\tilde{\lambda}%
}}}Z_{_{\pm \tilde{\lambda}}}$, $\tilde{R}_{XY}\xi $, $\tilde{R}_{X\xi}Z$.
Then by Theorem 5.8, and taking into account that, in fact
\begin{equation*}
X_{\tilde{\lambda}}=\frac{1}{2}\bigl(X-\eta (X)\xi +\frac{1}{\sqrt{-1-\tilde{%
\kappa}}}\tilde{\varphi}\tilde{h}X\bigr),\text{ \ \ \ }X_{-\tilde{\lambda}}=%
\frac{1}{2}\bigl(X-\eta (X)\xi -\frac{1}{\sqrt{-1-\tilde{\kappa}}}\tilde{%
\varphi}\tilde{h}X\bigr),
\end{equation*}%
we obtain \eqref{RXYZW2}.
\end{proof}

\begin{remark}
We point out that the surprising fact that formula \eqref{RXYZW2} is the same as \eqref{RXYZW}, tough the cases $\tilde\kappa<-1$ and $\tilde\kappa >-1$ are geometrically very different from each other.
\end{remark}

\begin{corollary}
Let $(M,\tilde{\varphi},\xi ,\eta ,\tilde{g})$ be a paracontact $(\tilde{%
\kappa},\tilde{\mu})$-manifold such that $\tilde{\kappa}<-1$. Then for any $%
X\in \Gamma ({\mathcal{D}})$ the $\xi $-sectional curvature $\tilde{K}(X,\xi
)$ is constant and is given by $\tilde{K}(X,\xi )=\tilde{\kappa}$. Moreover,
the sectional curvature of plane sections normal to $\xi $ is given by
\begin{gather*}
\tilde{K}(X,X^{\prime })=\tilde{K}(Y,Y^{\prime })=\tilde{\kappa}-1+\tilde{\mu%
} \\
\tilde{K}(X,Y)=\tilde{\lambda}^{2}-(\tilde{\mu}+1)\frac{\tilde{g}(X,\tilde{%
\varphi}Y)^{2}}{\tilde{g}(X,X)\tilde{g}(Y,Y)}
\end{gather*}%
for any $X,X^{\prime }\in \Gamma ({\mathcal{D}}_{\tilde{\varphi}\tilde{h}}(%
\tilde{\lambda}))$ and $Y,Y^{\prime }\in \Gamma ({\mathcal{D}}_{\tilde{%
\varphi}\tilde{h}}(-\tilde{\lambda}))$.
\end{corollary}

Using Theorem \ref{curvatura}, \eqref{Riczeta} and \eqref{KMU
QFI-FIQ}, one can easily prove the following result.

\begin{corollary}
\label{ricci} In any $(2n+1)$-dimensional paracontact $(\tilde{\kappa},%
\tilde{\mu})$-manifold $(M,\tilde{\varphi},\xi ,\eta ,\tilde{g})$ such that $%
\tilde{\kappa}<-1$, the Ricci operator $\tilde Q$ is given by
\begin{equation*}
\tilde{Q} = (2(1-n)+n\tilde\mu)I + (2(n+1)+ \tilde\mu) \tilde{h} +
(2(n-1)+n(2\tilde\kappa-\tilde\mu))\eta\otimes\xi.
\end{equation*}
In particular $(M,\tilde{g})$ is $\eta$-Einstein if and only if $%
\tilde\mu=2(1-n)$, Einstein if and only if $\tilde\kappa=\frac{1-n^2}{n}$
and $\tilde\mu=2(1-n)$.
\end{corollary}

\begin{remark}
We point out that, according to Corollary \ref{ricci}, if $\tilde\kappa<-1$,
there exist Einstein paracontact $(\tilde\kappa,\tilde\mu)$-manifolds also
in dimension greater than $3$. This is a relevant difference with respect
to the case $\tilde\kappa>-1$ (cf. Corollary \ref{ricci0}) and, moreover,
with respect to the contact metric case.
\end{remark}

We conclude with an example of paracontact $(\tilde{\kappa},\tilde{\mu})$%
-manifold such that $\tilde{\kappa}<-1$.

\begin{example}
Let $\mathfrak{g}$ be the Lie algebra with basis $%
\{e_{1},e_{2},e_{3},e_{4},e_{5}\}$ and non-zero Lie brackets
\begin{gather*}
\lbrack e_{1},e_{5}]=\alpha \beta e_{1}+\alpha \beta e_{2},\ \
[e_{2},e_{5}]=\alpha \beta e_{1}+\alpha \beta e_{2},\ \  \\
\lbrack e_{3},e_{5}]=-\alpha \beta e_{3}+\alpha \beta e_{4},\ \
[e_{4},e_{5}]=\alpha \beta e_{3}-\alpha \beta e_{4}, \\
\lbrack e_{1},e_{2}]=\alpha e_{1}+\alpha e_{2},\ \ [e_{1},e_{3}]=\beta
e_{2}+\alpha e_{4}-2e_{5},\ \ [e_{1},e_{4}]=\beta e_{2}+\alpha e_{3}, \\
\lbrack e_{2},e_{3}]=\beta e_{1}-\alpha e_{4},\ \ [e_{2},e_{4}]=\beta
e_{1}-\alpha e_{3}+2e_{5},\ \ [e_{3},e_{4}]=-\beta e_{3}+\beta e_{4}
\end{gather*}%
where $\alpha,\beta$ are non-zero real numbers such that $\alpha\beta>0$. Let $G$ be a Lie group whose Lie algebra is $\mathfrak{g}$. Define on $G$ a
left invariant paracontact metric structure $(\tilde{\varphi},\xi ,\eta ,%
\tilde{g})$ by imposing that, at the identity, $\tilde{g}(e_{1},e_{1})=%
\tilde{g}(e_{4},e_{4})=-\tilde{g}(e_{2},e_{2})=-\tilde{g}(e_{3},e_{3})=%
\tilde{g}(e_{5},e_{5})=1$, $\tilde{g}(e_{i},e_{j})=0$ for any $i\neq j$, and
$\tilde{\varphi}e_{1}=e_{3}$, $\tilde{\varphi}e_{2}=e_{4}$, $\tilde{\varphi}%
e_{3}=e_{1}$, $\tilde{\varphi}e_{4}=e_{2}$, $\tilde{\varphi}e_{5}=0$, $\xi
=e_{5}$ and $\eta =g(\cdot ,e_{5})$. A very long but straightforward
computation shows that
\begin{gather*}
\tilde{\nabla}_{e_{1}}\xi =\alpha \beta e_{1}-\tilde{\varphi}e_{1},\ \
\tilde{\nabla}_{e_{2}}\xi =\alpha \beta e_{2}-\tilde{\varphi}e_{2},\ \
\tilde{\nabla}_{\tilde{\varphi}e_{1}}\xi =-e_{1}-\alpha \beta \tilde{\varphi}%
e_{1},\ \ \tilde{\nabla}_{\tilde{\varphi}e_{2}}\xi =-e_{2}-\alpha \beta
\tilde{\varphi}e_{2}, \\
\tilde{\nabla}_{\xi }e_{1}=-\alpha \beta e_{2}-\tilde{\varphi}e_{1},\ \
\tilde{\nabla}_{\xi }e_{2}=-\alpha \beta e_{1}-\tilde{\varphi}e_{2},\ \
\tilde{\nabla}_{\xi }\tilde{\varphi}e_{1}=-e_{1}-\alpha \beta \tilde{\varphi}%
e_{2},\ \ \tilde{\nabla}_{\xi }\tilde{\varphi}e_{2}=-e_{2}-\alpha \beta
\tilde{\varphi}e_{1}, \\
\tilde{\nabla}_{e_{1}}e_{1}=\alpha e_{2}-\alpha \beta e_{5},\ \ \tilde{\nabla%
}_{e_{1}}e_{2}=\alpha e_{1},\ \ \tilde{\nabla}_{e_{1}}\tilde{\varphi}%
e_{1}=\alpha \tilde{\varphi}e_{2}-e_{5},\ \ \tilde{\nabla}_{e_{1}}\tilde{%
\varphi}e_{2}=\alpha \tilde{\varphi}e_{1}, \\
\tilde{\nabla}_{e_{2}}e_{1}=-\alpha e_{2},\ \ \tilde{\nabla}%
_{e_{2}}e_{2}=-\alpha e_{1}+\alpha \beta e_{5},\ \ \tilde{\nabla}_{e_{2}}%
\tilde{\varphi}e_{1}=-\alpha \tilde{\varphi}e_{2},\ \ \tilde{\nabla}_{e_{2}}%
\tilde{\varphi}e_{2}=-\alpha \tilde{\varphi}e_{1}+e_{5}, \\
\tilde{\nabla}_{\tilde{\varphi}e_{1}}e_{1}=-\beta e_{2}+e_{5},\ \ \tilde{%
\nabla}_{\tilde{\varphi}e_{1}}e_{2}=-\beta e_{1},\ \ \tilde{\nabla}_{\tilde{%
\varphi}e_{1}}\tilde{\varphi}e_{1}=-\beta \tilde{\varphi}e_{2}-\alpha \beta
e_{5},\ \ \tilde{\nabla}_{\tilde{\varphi}e_{1}}\tilde{\varphi}e_{2}=-\beta
\tilde{\varphi}e_{1}, \\
\tilde{\nabla}_{\tilde{\varphi}e_{2}}e_{1}=-\beta e_{2},\ \ \tilde{\nabla}_{%
\tilde{\varphi}e_{2}}e_{2}=-\beta e_{1}-e_{5},\ \ \tilde{\nabla}_{\tilde{%
\varphi}e_{2}}\tilde{\varphi}e_{1}=-\beta \tilde{\varphi}e_{2},\ \ \tilde{%
\nabla}_{\tilde{\varphi}e_{2}}\tilde{\varphi}e_{2}=-\beta \tilde{\varphi}%
e_{1}+\alpha \beta e_{5}.
\end{gather*}%
where \ $\tilde{\lambda}=\alpha \beta$ and $\tilde{\mu}=2$. \ Then one
can prove that the curvature tensor field of the Levi-Civita connection of $%
(G,\tilde{g})$ satisfies the $(\tilde{\kappa},\tilde{\mu})$-nullity
condition \eqref{PARAKMU}, with $\tilde{\kappa}=-1-(\alpha \beta )^{2}$ and $%
\tilde{\mu}=2$.
\end{example}

\textbf{Acknowledgements. }The authors are grateful to the referee for their
valuable comments and suggestions.

\end{document}